\documentclass[11pt]{alggeom}
\usepackage{amsmath}
\usepackage{amssymb,bm,mathrsfs}
\usepackage[mathscr]{euscript}
\usepackage{xcolor}

\usepackage{mathabx}
\usepackage[all]{xy}
\usepackage{hyperref}
\newcommand{\map}[1]{\xrightarrow{#1}}

\newcommand{\iso}{\cong}
\newcommand{\define}{\stackrel{\mathrm{def}}{=}}

\newcommand{\red}{\mathrm{red}}
\newcommand{\Q}{\mathbb Q}
\newcommand{\Z}{\mathbb Z}

\newcommand{\F}{\mathbb F}

\newcommand{\co}{\mathcal O}

\DeclareMathOperator{\Hom}{Hom}

\DeclareMathOperator{\End}{End}
\DeclareMathOperator{\Spec}{Spec}
\DeclareMathOperator{\Lie}{Lie}
\DeclareMathOperator{\Fil}{Fil}
\DeclareMathOperator{\ord}{ord}
\DeclareMathOperator{\GL}{GL}

\DeclareMathOperator{\inv}{inv}

\DeclareMathOperator{\GU}{GU}

\begin{document}

\title{Rapoport-Zink spaces of type $\mathrm{GU}(2,n-2)$}
\date{\today}

\author{Maria Fox}
\email{maria.fox@okstate.edu}
\address{Department of Mathematics, Oklahoma State University, Stillwater, OK 74078, USA}

\author{Benjamin Howard}
\email{howardbe@bc.edu}
\address{Department of Mathematics, Boston College, 140 Commonwealth Ave, Chestnut Hill, MA 02467, USA}

\author{Naoki Imai}
\email{naoki@ms.u-tokyo.ac.jp}
\address
{Graduate School of Mathematical Sciences, The University of Tokyo, 3-8-1 Komaba, Meguro-ku, Tokyo, 153-8914, Japan}

\thanks{
M.F. was supported by NSF MSPRF Grant
2103150. B.H. was supported in part by NSF grant DMS-2101636. N.I. was supported by JSPS KAKENHI Grant Number 22H00093. 
}


\begin{abstract}
We describe the structure of the supersingular Rapoport-Zink space associated to the group of unitary similitudes of signature $(2,n-2)$ for an unramified quadratic extension of  $p$-adic fields.
In earlier work, two of the authors described the irreducible components in the category of schemes-up-to-perfection.  
The goal of this work is to remove the qualifier ``up-to-perfection".
\end{abstract}

\maketitle

\setcounter{tocdepth}{1}
\tableofcontents

\theoremstyle{plain}
\newtheorem{theorem}{Theorem}[subsection]
\newtheorem{bigtheorem}{Theorem}
\newtheorem{proposition}[theorem]{Proposition}
\newtheorem{lemma}[theorem]{Lemma}
\newtheorem{corollary}[theorem]{Corollary}
\newtheorem{conjecture}[theorem]{Conjecture}

\theoremstyle{definition}
\newtheorem{definition}[theorem]{Definition}
\newtheorem{hypothesis}[theorem]{Hypothesis}

\theoremstyle{remark}
\newtheorem{remark}[theorem]{Remark}
\newtheorem{example}[theorem]{Example}
\newtheorem{question}[theorem]{Question}

\renewcommand{\thebigtheorem}{\Alph{bigtheorem}}

\numberwithin{equation}{subsection}


\section{Introduction}


Rapoport-Zink formal schemes \cite{RZbook} are  defined as moduli spaces of $p$-divisible groups with additional structure.  While general existence theorems are known for these spaces, it is a very difficult problem to determine their structures  as formal schemes, or even  the structure of their underlying reduced schemes, in any  explicit way.  
The history of this problem is long,  and has its origins in Drinfeld's work (before Rapoport and Zink) on the $p$-adic uniformization of quaternionic Shimura curves.  
The reader can find a thorough guide to the older literature in the introduction to Vollaard's work \cite{Vollaard} on Rapoport-Zink spaces of type $\mathrm{GU}(1,n-1)$.

This work of Vollaard, and the subsequent work of Vollaard-Wedhorn \cite{VollaardWedhorn}, introduced significant new ideas into the subject,  which were extended further in \cite{Fox}, \cite{HP14}, \cite{HP17}, \cite{RTW},  and \cite{Wang}.
In all of these works, the main results assert that  the irreducible components of (the reduced scheme underlying) a particular Rapoport-Zink space are isomorphic to  Deligne-Lusztig varieties, a term which we always understand in the the generalized sense of \S 4.4 of  \cite{VollaardWedhorn}.
G\"{o}rtz-He \cite{GH} and G\"{o}rtz-He-Nie \cite{GHN} gave a classification of those Rapoport-Zink spaces for which one should expect the irreducible components to have this form, and called such Rapoport-Zink spaces \emph{fully Hodge-Newton decomposable}. 

Xiao-Zhu \cite{XZ} proved quite general results on the structure of Rapoport-Zink spaces (and more general affine Deligne-Lustztig varieties).  The results of Xiao-Zhu provide a parametrization of the irreducible components,
but do not  provide a  description of their scheme-theoretic structure.
 The Xiao-Zhu  parametrization of components has since been generalized by other  authors; see   \cite{HV}, \cite{HZZ}, \cite{Nie}, and \cite{ZZ}.

Among the simplest examples of Rapoport-Zink spaces that are not fully Hodge-Newton decomposable are those of type  $\mathrm{GU}(2,n-2)$.  Because of the results of  G\"{o}rtz-He-Nie mentioned above, there is no expectation that the irreducible components in this setting are Deligne-Lusztig varieties.
Despite this,  earlier work of two of the authors  \cite{FI}  showed that an open dense subset of each irreducible component can be fibered  over a Deligne-Lusztig variety, with the fibers made explicit.

One sense in which the descriptions of irreducible components  in \cite{FI} is incomplete is that everything is understood  in the  category of schemes-up-to-perfection.  
Loosely speaking, this means that for each irreducible component, a scheme (fibered over a Deligne-Lusztig variety) is exhibited with the property that it has the same functor of points \emph{when restricted to perfect algebras in characteristic $p$}.  For example, when viewed as an object of this category, each irreducible component is   indistinguishable  from its Frobenius twists.  
This ambiguity in the scheme-theoretic structure of the components arises because \cite{FI} adheres closely to the framework of \cite{XZ}, in which Rapoport-Zink spaces are replaced by their corresponding affine Deligne-Lustzig varieties, which are thought of not as moduli spaces of  $p$-divisible groups, but as closed subsets  of  the  Witt vector affine Grassmannians of \cite{Zhu} and \cite{BhattScholze}.
  These  Witt vector affine Grassmannians
are only defined as objects in the category of (ind-)schemes-up-to-perfection.

The primary purpose of this paper is to revisit the results of \cite{FI}, in order to pin down the precise scheme-theoretic structure,  not just up to perfection, of the irreducible components of the  $\mathrm{GU}(2,n-2)$ Rapoport-Zink space.   

In practice, this requires making  systematic use of  the moduli interpretation  of the Rapoport-Zink space (as opposed to the interpretation as a closed subset of a Witt vector affine Grassmannian), in order to exploit the Grothendieck-Messing deformation theory of the universal $p$-divisible group that lives over it.


\subsection{Statement of the results}


Throughout this paper we denote by $\breve{\Q}_p$ the completion of the maximal unramified extension of $\Q_p$, and by 
$\breve{\Z}_p$ its ring of integers.  
Let $\sigma : \breve{\Q}_p \to \breve{\Q}_p$ be the Frobenius, inducing the $p$-power automorphism of the  residue field 
 \[
 \breve{\F}_p= \breve{\Z}_p / p\breve{\Z}_p.
 \]
Note that $\breve{\F}_p$ is just an algebraic closure of $\F_p$.

Fix an unramified quadratic field extension $E$ of $\Q_p$.  
We are interested in the Rapoport-Zink formal scheme 
\[
\mathrm{RZ} \to \mathrm{Spf}( \breve{\F}_p) 
\]
parametrizing $p$-divisible groups $X$ of dimension $n\ge 2$ over $\F_p$-algebras, endowed with principal polarizations, an action of $\co_E$ satisfying a signature $(2,n-2)$ condition, and a quasi-isogeny $\varrho_X : X \dashrightarrow \mathbb{X}$ to a fixed framing object.  See \S \ref{ss:RZbasics} for the precise definitions.

For our framing object $\mathbb{X}$ we make a nonstandard choice: an $n$-dimensional $p$-divisible group, again with an $\co_E$-action and principal polarization, but satisfying a signature $(0,n)$-condition.   We will show that $\mathbb{X}$ is unique up to isomorphism, not just isogeny, and has the form
\[
\mathbb{X} \iso \Lambda \otimes_{\co_E} \overline{\mathbb{Y}}
\]
for  $\overline{\mathbb{Y}}$ a supersingular $p$-divisible group of height $2$ and dimension $1$ endowed with an action of $\co_E$ of signature $(0,1)$, and $\Lambda$  a  rank $n$ self-dual hermitian $\co_E$-lattice.    
By the comments preceding  Definition \ref{def:fullRZ}, our $\mathbb{X}$ is  also quasi-isogenous to a $p$-divisible group of signature $(2,n-2$), so this unusual choice of framing object yields the usual Rapoport-Zink space for $\mathrm{GU}(2,n-2)$.

Let $G=\mathrm{GU}(\Lambda)$ be the group of unitary similitudes of $\Lambda$, a reductive group over $\Z_p$.  Its group of $\Q_p$-points acts on the framing object  $\Lambda \otimes \overline{\mathbb{Y}}$ by quasi-isogenies, and hence also acts on $\mathrm{RZ}$.

Denote by  $\mathrm{RZ}_\Lambda \subset  \mathrm{RZ}$ the projective closed  subscheme parametrizing commutative diagrams 
\[
\xymatrix{
&{ X }\ar@{-->}[d]^{\varrho_X}  \ar[dr] \\
{ p\Lambda \otimes_{\co_E} \overline{\mathbb{Y}}}  \ar[r] \ar[ur] & { \Lambda \otimes_{\co_E} \overline{\mathbb{Y}}}   \ar[r] & { p^{-1}\Lambda \otimes_{\co_E}  \overline{\mathbb{Y}}} ,
}
\]
in which all solid arrows are isogenies, and the horizontal arrows are induced by the inclusions
$p\Lambda \subset \Lambda \subset p^{-1}\Lambda$.
It will turn out (see the proof of Corollary \ref{cor:component cover}) that
\[
\mathrm{RZ}^\red = \bigcup_{\gamma \in G(\Q_p) / G(\Z_p) } \gamma \cdot \mathrm{RZ}_\Lambda^\red,
\]
where the superscript $\red$ indicates underlying reduced scheme.
(We suspect that $\mathrm{RZ}_\Lambda$ is already reduced, and that this might be proved by arguing as in Corollary 3.2.3 of \cite{LZ}, but we have not checked this.) 
Hence the irreducible components of the left hand side are precisely the $G(\Q_p)$-translates of irreducible components of $\mathrm{RZ}_\Lambda^\red$.

To describe these, we decompose
\[
\mathrm{RZ}_\Lambda^\red = \bigsqcup_{1\le k \le \lfloor n/2 \rfloor} \mathrm{RZ}_\Lambda^{k,\red} 
\]
into locally closed subschemes.  The precise definition  of the schemes on the right hand side appears in \S \ref{ss:hyperspecial strata}, but 
loosely speaking, as $k$ increases the  points of $\mathrm{RZ}_\Lambda^{k,\red}$ get farther from the framing object.

The following result is stated in the text as Theorem \ref{thm:main components}.  
 The Deligne-Lusztig variety  is that of Definition  \ref{def:non-minuscule DL}, and the vector bundle $\mathscr{V}$ on it is constructed in \S \ref{ss:special bundle}.

\begin{bigtheorem}\label{bigthmA}
If $k < n/2$ then $\mathrm{RZ}_\Lambda^{k,\red}$ is a smooth and irreducible $\breve{\F}_p$-scheme of dimension $n-2$, and its  closure 
 \[
 \overline{\mathrm{RZ}}_\Lambda^{k,\red} \subset  \mathrm{RZ}^\red
 \]
  is an irreducible component with stabilizer  $G(\Z_p)\subset G(\Q_p)$.
Moreover, there is a smooth morphism 
\[
\mathrm{RZ}_\Lambda^{k,\red} \to \mathrm{DL}_\Lambda^k
\]
of relative dimension $k-1$ to a smooth and projective Deligne-Lusztig variety with the following property: over  $ \mathrm{DL}_\Lambda^k$ there is a vector bundle $\mathscr{V}$ of rank $2k-1$, endowed with a rank $k$ local direct summand $\mathscr{V}^{(k)} \subset \mathscr{V}$ and a morphism
\[
\beta : \mathscr{V} \otimes \sigma^* \mathscr{V} \to \co_{ \mathrm{DL}_\Lambda^k } ,
\]
such that $\mathrm{RZ}_\Lambda^{k,\red}$ is identified with the moduli space parametrizing 
complementary summands 
\[
\mathscr{V} = \mathscr{F} \oplus \mathscr{V}^{(k)}
\]
that are  totally isotropic, in the sense that $\beta(\mathscr{F} \otimes \sigma^*\mathscr{F})=0$.  Here $\sigma^*$ denotes pullback of coherent sheaves with the respect to the $p$-power Frobenius on the structure sheaf of $\mathrm{DL}_\Lambda^k$.
\end{bigtheorem}

 Although the technical details obscure it, 
the basic idea for constructing a map from $\mathrm{RZ}_\Lambda^{k,\red}$ to a Deligne-Lusztig variety is quite simple.  
An $\breve{\F}_p$-valued point of $\mathrm{RZ}_\Lambda^{k,\red}$ corresponds to a quasi-isogeny of $p$-divisible groups $X \dashrightarrow \mathbb{X}$.  
This realizes the covariant Dieudonn\'e modules of these $p$-divisible groups as lattices in  a common $\breve{\Q}_p$-vector space, and the intersection 
 $
 D(H) = D(X) \cap D( \mathbb{X})
 $
 is the Dieudonn\'e module of a  $p$-divisible group $H$ endowed with an isogeny $H \to  \mathbb{X}$.  
 This latter isogeny is quite small, in the sense that 
 \[
pD( \mathbb{X} ) \subset   D(H)  \subset  D( \mathbb{X}), 
 \]
and so  is determined by  the subspace
 \[
   \frac{ D(H)  }{ p D( \mathbb{X} )  }  \subset 
   \frac{  D( \mathbb{X}) }{  p D( \mathbb{X}) } \iso \Lambda\otimes_{\Z_p} \breve{\F}_p.
 \]
 The isomorphism here comes from our particular choice of framing object, as  a choice of isomorphism $D( \overline{\mathbb{Y}} ) \iso \co_E\otimes_{\Z_p} \breve{\Z}_p$ identifies
 \[
 D( \mathbb{X}) = D( \Lambda \otimes_{\co_E} \overline{\mathbb{Y}} ) 
 = \Lambda\otimes_{\co_E} D(  \overline{\mathbb{Y}} )  \iso \Lambda\otimes_{\Z_p} \breve{\Z}_p.
 \]
For the rough purposes of this introduction,  one can think of the Deligne-Lusztig variety  as parametrizing all subspaces   that arise from this construction.

 If $n$ is odd Theorem \ref{bigthmA} completes our description of $\mathrm{RZ}^\red$, as 
 \[
 \mathrm{RZ}^\red= 
  \bigcup_{ \substack{ 1\le k < n/2  \\  \gamma \in G(\Q_p)/G(\Z_p)} }   \gamma \cdot \overline{\mathrm{RZ}}_\Lambda^{k,\red} 
 \]
exhibits the left hand side as the union of its irreducible components.
Ideally one would like to have a description not just of $\mathrm{RZ}_\Lambda^{k,\red}$, but of its closure.  This seems quite difficult for  $k>1$.
When  $k=1$, Theorem \ref{bigthmA}  implies that 
\[ \mathrm{RZ}_\Lambda^{1,\red}\iso \mathrm{DL}_\Lambda^1\]  is   projective, so no closure is needed.

 When $n$ is even we must also examine $\mathrm{RZ}_\Lambda^{n/2,\red}$.     For every intermediate lattice 
 $
 p\Lambda \subsetneq \Lambda' \subsetneq \Lambda
 $
 such that $\Lambda'/p\Lambda \subset \Lambda/p\Lambda$ is maximal isotropic, we define  
 \[
 \mathrm{RZ}^\heartsuit_{\Lambda'} \subset \mathrm{RZ}_\Lambda
 \]
 as the projective closed subscheme parametrizing commutative diagrams
 \[
\xymatrix{
&{ X }\ar@{-->}[d]^{\varrho_X}  \ar[dr] \\
{ \Lambda' \otimes_{\co_E}\overline{\mathbb{Y}}}  \ar[r] \ar[ur] & { \Lambda \otimes_{\co_E}\overline{\mathbb{Y}}}   \ar[r] & { p^{-1}\Lambda' \otimes_{\co_E}\overline{\mathbb{Y}}} ,
}
\]
in which all solid arrows are isogenies, and the horizontal arrows are induced by the inclusions
$\Lambda' \subset \Lambda \subset p^{-1}\Lambda'$.

The following result is stated in the text as Theorem \ref{thm:aux components}.  
 The Deligne-Lusztig variety in the theorem is that of Definition  \ref{def:heartDL}.

\begin{bigtheorem}\label{bigthmB}
If $n$ is even then
\[
\mathrm{RZ}_\Lambda^{n/2,\red} \subset 
 \bigcup_{ p\Lambda  \subsetneq \Lambda'  \subsetneq \Lambda}  \mathrm{RZ}_{\Lambda'}^{\heartsuit ,\red},
\]
where the union is over   intermediate lattices  for which $\Lambda'/p\Lambda \subset \Lambda/p\Lambda$ is maximal isotropic.  
All of the $ \mathrm{RZ}_{\Lambda'}^{\heartsuit ,\red}$  appearing on the right hand side lie in the same $G(\Q_p)$-orbit,  each is 
an irreducible component of $\mathrm{RZ}^\red$ with stabilizer a $G(\Q_p)$-conjugate of $G(\Z_p)$, and each is isomorphic to a smooth projective Deligne-Lusztig variety of dimension $n-2$.
\end{bigtheorem}

When $n$ is even, if we fix one $\Lambda'\subset \Lambda$ as above then
\[
 \mathrm{RZ}^\red= 
\Big(   \bigcup_{ \substack{ 1\le k < n/2  \\  \gamma \in G(\Q_p)/G(\Z_p)} }   \gamma \cdot \overline{\mathrm{RZ}}_\Lambda^{k,\red}  \Big)
\cup 
\Big(  \bigcup_{  \gamma \in G(\Q_p)/ hG(\Z_p)h^{-1}   }   \gamma \cdot \mathrm{RZ}_{\Lambda'}^{\heartsuit,\red}  \Big)
 \]
 exhibits the left hand side as the union of its irreducible components. 
 Here $h\in G(\Q_p)$ is any element satisfying $h \Lambda = p^{-1} \Lambda'$, as in Theorem \ref{thm:aux components} and its proof.

Regardless of whether $n$ is even or odd, the above results imply that $\mathrm{RZ}^\red$ is equidimensional of dimension $n-2$.
This  is already known by the work of Xiao-Zhu \cite{XZ}.


\subsection{General notation}


Throughout the paper, $E$ is an unramified quadratic extension of $\Q_p$.
Denote by 
\[
i_0,i_1 : \co_E \to \breve{\Z}_p
\]
 the two embeddings.
The nontrivial automorphism of $E$ is denoted $x\mapsto \overline{x}$, so that $ i_0 ( \overline{x} ) = \sigma ( i_1(x))$ and similarly with the indices $0$ and $1$ reversed.
Abbreviate
\[
\breve{E} = E \otimes_{\Q_p}  \breve{\Q }_p,
 \] 
 and $\co_{\breve{E}} = \co_E \otimes_{\Z_p} \breve{\Z}_p$.
Label the orthogonal idempotents 
 \begin{equation}\label{orthopotents}
 e_0,e_1 \in \co_{\breve{E}} \iso \breve{\Z}_p \times \breve{\Z}_p
 \end{equation}
 in such a way that for any $\breve{\Z}_p$-module $M$ with a commuting action of $\co_E$, the actions of $\co_E$ on the direct summands
 \begin{equation}\label{general decomp}
 M_0 = e_0 M \quad \mbox{and}\quad M_1=e_1M 
 \end{equation}
 are through  $i_0$ and $i_1$, respectively.

A hermitian form on an $E$-vector space or  $\co_E$-module is always $E$-linear in the first variable and conjugate-linear in the second variable.

If $\mathscr{F}$ is a coherent sheaf on a scheme $X$, the notation $\mathscr{F}_x$ always means the fiber (not the stalk) at a point $x\in X$.
If $X$ is any scheme (or formal scheme) we denote by $X^\red$ the underlying reduced scheme.

If $M$ is a $\Z_p$-module, we  usually abbreviate $M[1/p] = M\otimes_{\Z_p} \Q_p$ and 
\[
\breve{M}=M\otimes_{\Z_p}\breve{\Z}_p.
\]

Suppose $X$ is a $p$-divisible group equipped with an action $\co_E \to \End(X)$.
We  endow the dual $p$-divisible group $X^\vee$ with the \emph{conjugate} of the naive $\co_E$-action; that is to say, the action $\co_E \to \End(X^\vee)$ sends $x\in \co_E$ to the endomorphism 
\begin{equation}\label{conjugate convention}
\overline{x}^\vee : X^\vee \to X^\vee
\end{equation}
 dual to $\overline{x} : X \to X$.  When we speak of an $\co_E$-linear polarization  $\lambda_X : X \to X^\vee$, we always understand $X^\vee$ to have this action of $\co_E$.

If $\phi : X\to Y$ is an isogeny of $p$-divisible groups over a field $k$, the kernel $\ker(\phi)$ is a finite group scheme over $k$ of order $p^{\mathrm{ht}(\phi)}$ for a positive integer $\mathrm{ht}(\phi)$ called the \emph{height} of $\phi$.
More generally, if $\phi : X\to Y$ is an isogeny of $p$-divisible groups over a scheme $S$, we denote by $\mathrm{ht}(\phi)$ the locally constant function on $S$ whose value at a point $s\in S$ is the height of the isogeny $\phi_s:X_s \to Y_s$
of $p$-divisible groups over the residue field $k(s)$.

 
 \subsection{Acknowledgements}
 

 The authors thank the anonymous referee for helpful comments and corrections.


\section{The Rapoport-Zink space}
\label{s:partialRZ}


In this section we define our main object of study, a  Rapoport-Zink space $\mathrm{RZ}$ of type $\mathrm{GU}(2,n-2)$, as a formal $\breve{\F}_p$-scheme  parametrizing $p$-divisible groups with additional structure.
This additional structure includes a quasi-isogeny to a fixed framing object, which determines distinguished closed subscheme $\mathrm{RZ}_\Lambda \subset \mathrm{RZ}$.

We describe the $\breve{\F}_p$-valued points of this subscheme in terms of lattices in a free $\breve{\Z}_p$-module, and then
express $\mathrm{RZ}_\Lambda$ as a union of locally closed subschemes.  
The   examination of these locally closed subschemes will occupy the rest of the paper.

 
 \subsection{Framing objects}
 \label{ss:frame}
 

For any $p$-divisible group $X$ over $\breve{\F}_p$, let $D(X)$ be the covariant Dieudonn\'e module.
It is a $\breve{\Z}_p$-module endowed with semi-linear operators  $F$ and $V$ satisfying $FV=VF=p$.  Under the covariant conventions, 
\begin{equation}\label{dieudonne-lie}
\Lie(X) \iso D(X) / V  D(X).
\end{equation}
 If $X$ has an action of $\co_E$, the induced action of $\co_E\otimes_{\Z_p} \breve{\Z}_p$ on its Dieudonn\'e module $D(X)$ determines, by \eqref{general decomp}, a decomposition 
\[
D(X) = D(X)_0 \oplus D(X)_1.
\]

Suppose  $X$ is a $p$-divisible group over an $\breve{\F}_p$-scheme $S$, equipped with an  action $\co_E \to \End(X)$.
As above, the induced action of $\co_E\otimes_{\Z_p} \co_S$ on  $\Lie(X)$ determines  a decomposition 
\begin{equation}\label{lie decomp}
\Lie(X) = \Lie(X)_0 \oplus \Lie(X)_1
\end{equation}
as a sum of locally free $\co_S$-modules.

\begin{definition}
The action $\co_E \to \End(X)$ has \emph{signature $(s_0,s_1)$} if  the summands on the right hand side of \eqref{lie decomp} are  locally free of ranks $s_0$ and $s_1$, respectively.
\end{definition}

There is an analogue of Serre's tensor construction for $p$-divisible groups.  
Suppose $X$ is any $p$-divisible group over an $\breve{\F}_p$-scheme $S$, endowed with an action of $\co_E$.  If $\Lambda$ is any free $\co_E$-module of finite rank, one can form a new $p$-divisible group $\Lambda \otimes X$ with $\co_E$-action over $S$ with functor of points 
\[
(\Lambda \otimes X)(T) =  \Lambda  \otimes   X(T) 
\] 
for any $T\to S$.  
Here both tensor products are over $\co_E$, but we typically  omit this from the notation.
Of course choosing an $\co_E$-basis of $\Lambda$ identifies $ \Lambda \otimes X$ with a product of $\mathrm{rank}_{\co_E}(\Lambda)$ copies of $X$.

Suppose further that $X$ is endowed with an $\co_E$-linear polarization $\lambda_X :X \to X^\vee$ (recall our convention  \eqref{conjugate convention} for the induced action of $\co_E$ on $X^\vee$), while $\Lambda$ is endowed with a hermitian form $h : \Lambda \times \Lambda \to \co_E$. 
This data induces an $\co_E$-linear polarization 
\[
\lambda_{ \Lambda  \otimes X} :  \Lambda \otimes  X \to (  \Lambda\otimes X)^\vee
\]
constructed as follows.
First note that if $G$ is any $p$-divisible group over $S$ endowed with an $\co_E$-action, there are isomorphisms of finite flat group schemes
\begin{equation}\label{trace iso}
\underline{\Hom}_{\overline{\co}_E}(G[p^k] , \co_E \otimes \mu_{p^k})
\iso \underline{\Hom}(G[p^k] ,  \mu_{p^k}) \iso G^\vee[p^k], 
\end{equation}
compatible as $k$ varies, where the  first isomorphism is defined by the trace map $\co_E \to \Z_p$, and the second is the definition of $G^\vee$.  The subscript $\overline{\co}_E$ indicates  $\co_E$-conjugate-linear homomorphisms. 

Note that the composition \eqref{trace iso} is $\co_E$-linear, where the action of $\co_E$ on $\underline{\Hom}_{\overline{\co}_E}(G[p^k] , \co_E \otimes \mu_{p^k})$ is through its natural action on the $p$-divisible group $\co_E \otimes \mu_{p^k}$, and its action on $G^\vee[p^k]$ is via \eqref{conjugate convention}.

In particular, the principal polarization $\lambda_X$  induces, for every $k$, an $\co_E$-linear isomorphism
\[
\tilde{\lambda}_X: X[p^k]  \to 
\underline{\Hom}_{\overline{\co}_E}(X[p^k] , \co_E \otimes  \mu_{p^k}).
\]
The polarization  $\lambda_{\Lambda \otimes X}$  is defined, on $p^k$-torsion, as the composition of $\co_E$-linear maps
\[
\xymatrix{
{    ( \Lambda \otimes X) [ p^k]   } \ar@{=} [r]&  {  \Lambda \otimes  X[p^k]   }  \ar[d]   \\  
& {   \underline{\Hom}_{\overline{\co}_E} (      \Lambda \otimes X[p^k] , \co_E \otimes \mu_{p^k} )  }  \ar@{=}[r]&   {   (\Lambda \otimes X)^\vee[p^k] } 
}
\]
in which the vertical arrow sends  $a\otimes x \in  \Lambda \otimes X[p^k]$ to the $\co_E$-conjugate linear homomorphism
$ b\otimes y \mapsto h(a,b)\cdot \tilde{\lambda}_X(x)(y).$

 Let $\mathbb{Y}$ be the $p$-divisible group of a supersingular elliptic curve over $\breve{\F}_p$.
 In other words, $\mathbb{Y}$ is the (unique up to isomorphism) connected $p$-divisible group of dimension $1$ and height $2$.
Fix an action  $\co_E \to \End(\mathbb{Y})$ of signature $(1,0)$, and a principal polarization $\lambda_{\mathbb{Y}} : \mathbb{Y} \to \mathbb{Y}^\vee$.
Any such   polarization is necessarily $\co_E$-linear (under the convention of \eqref{conjugate convention}, as always).

Denote by $\overline{\mathbb{Y}}$ the $p$-divisible group $\mathbb{Y}$, but now endowed with the conjugate action of $\co_E$.
 Thus $\overline{\mathbb{Y}}$ has signature $(0,1)$.  The principal polarization $\lambda_\mathbb{Y}$ can  be viewed as a principal polarization $\lambda_{ \overline{\mathbb{Y}}} : \overline{\mathbb{Y}} \to \overline{\mathbb{Y}} ^\vee$, which is again  $\co_E$-linear.

\begin{proposition}\label{prop:frame}
Fix an integer $n\ge 2$.  
\begin{enumerate}
\item
If $\Lambda$ is a self-dual hermitian $\co_E$-lattice of rank $n$,  the natural $\co_E$-action on  
$\Lambda   \otimes  \overline{\mathbb{Y}}$
 has signature $(0,n)$, and the natural   $\co_E$-linear polarization   is principal.
\item
Conversely, suppose $\mathbb{X}$ is any $p$-divisible group  over $\breve{\F}_p$,  equipped with an $\co_E$-action of signature $(0,n)$.
There exists a  self-dual hermitian lattice $\Lambda$ of rank $n$, and an $\co_E$-linear isomorphism 
$
\mathbb{X} \iso \Lambda   \otimes  \overline{\mathbb{Y}}.
$
Moreover, if $\mathbb{X}$ is endowed with   an $\co_E$-linear principal polarization, this isomorphism may be chosen to identify the polarizations on source and target.
\end{enumerate}
We remark that there is a unique (up to isometry) self-dual hermitian lattice  $\Lambda$ of rank $n$, and so the $\mathbb{X}$ in (2), whether polarized or not, is unique up to isomorphism.
\end{proposition}

\begin{proof}
The first claim is obvious from the definitions. 

For the second claim, the signature condition on $\mathbb{X}$ implies that its Dieudonn\'e module satisfies
\[
V D(\mathbb{X})_0 = p D(\mathbb{X})_1 \quad \mbox{and}\quad V D(\mathbb{X})_1 = D(\mathbb{X})_0 .
\]
It follows that  $\tau = p^{-1} F^2$ defines a $\sigma^2$-semi-linear automorphism of $D(\mathbb{X})_0$,
whose fixed points $D(\mathbb{X})_0^{\tau=\mathrm{id}}$
form a free $\Z_{p^2} = \breve{\Z}_p^{\sigma^2=\mathrm{id}}$-module with
\[
D(\mathbb{X})_0^{\tau=\mathrm{id}}\otimes_{\Z_{p^2}} \breve{\Z}_p = D(\mathbb{X})_0.
\]

Choose a basis
\[
y_1,\ldots, y_n \in D(\mathbb{X})_0^{\tau=\mathrm{id}}\subset D(\mathbb{X})_0.
\]
If  we define $x_i \in D(\mathbb{X})_1$ by $V x_i =y_i$, then each pair $x_i,y_i$ generates a $\breve{\Z}_p$-submodule 
$D^{(i)} \subset D(\mathbb{X})$ stable under $F$ and $V$.   The decomposition
\[
D(\mathbb{X}) = D^{(1)} \oplus \cdots \oplus D^{(n)}
\]
of Dieudonn\'e modules corresponds to a decomposition $\mathbb{X} = X^{(1)} \times \cdots \times X^{(n)}$ of $p$-divisible groups, with each factor $\co_E$-linearly isomorphic to $\overline{\mathbb{Y}}$.
  In other words, 
\begin{equation}\label{bbXsplit}
\mathbb{X} \iso \overline{\mathbb{Y}} \times \cdots \times  \overline{\mathbb{Y}}. 
\end{equation}

If we also assume that $\mathbb{X}$ admits an $\co_E$-linear principal polarization, there is an induced  perfect alternating pairing
\[
\lambda_\mathbb{X}: D(\mathbb{X})  \times D(\mathbb{X}) \to \breve{\Z}_p
\]
satisfying $\lambda_\mathbb{X}( F x, y) = \lambda_\mathbb{X}( x, Vy)^\sigma$ and 
$\lambda_\mathbb{X}( \alpha x, y) = \lambda_\mathbb{X}( x, \bar{\alpha} y)$ for all $\alpha\in \co_E$ and $x,y\in D(\mathbb{X})$.
Using these properties, we see that 
\[
\langle x,y \rangle  \define  p^{-1} \lambda_\mathbb{X}( x, V y)
\]
defines a $\Z_{p^2}$-valued skew-hermitian form on $D(\mathbb{X})_0^{\tau=\mathrm{id}}$ of unit determinant.  
Such a form admits a diagonal basis, and we assume that $y_1,\ldots, y_n$ is chosen in such a way.
The  polarization on the left hand side of \eqref{bbXsplit} then identified with the the product of \emph{some} principal polarizations on the factors of the right hand side.
 As the principal polarization on $\overline{\mathbb{Y}}$ is unique up to scaling by $\Z_p^\times$, we may  choose the isomorphism \eqref{bbXsplit} so that the given polarization on $\mathbb{X}$  matches the product of the principal polarization of $\overline{\mathbb{Y}}$ that we originally fixed.

Thus, whether $\mathbb{X}$ was polarized or not, we may take $\Lambda=\co_E^n$  with the hermitian form determined by the $n\times n$ identity matrix. 
\end{proof}

\begin{remark}\label{rem:KR isometry}
Why  have we expressed the second claim in Proposition \ref{prop:frame} in such a peculiar way, when the proof actually shows that $\mathbb{X}$ is isomorphic to $n$ copies of $\overline{\mathbb{Y}}$?
One reason is that such an isomorphism is not canonical, but the presentation $\mathbb{X} \iso\Lambda   \otimes  \overline{\mathbb{Y}}$ is.  Indeed, once such a presentation is known to exist, it is easy to see that 
\[
\Lambda \iso \Hom_{\co_E} ( \overline{\mathbb{Y}} , \mathbb{X} ),
\]
in which the right hand side is endowed with the  hermitian form  characterized by the relation  (3.1) of  \cite{KR1}.
\end{remark}

 
 \subsection{The Rapoport-Zink space}
 \label{ss:RZbasics}
 

For the remainder of \S \ref{s:partialRZ},  we work with a fixed  self-dual hermitian $\co_E$-lattice  $\Lambda$  of rank $n$.  It is unique up to isomorphism.

We next define a $\mathrm{GU}(2,n-2)$ Rapoport-Zink space $\mathrm{RZ}$, parametrizing $p$-divisible groups with extra structure.  This extra structure would normally include a  quasi-isogeny to a fixed supersingular $p$-divisible group $\mathbb{X}$ (the \emph{framing object}), endowed   with an $\co_E$-action of signature $(2,n-2)$ and an $\co_E$-linear principal polarization.  Such a framing object is unique up to quasi-isogeny by 
Proposition 1.15 of \cite{Vollaard}, but that result actually proves more:
any such $\mathbb{X}$  is  also  quasi-isogenous to  the $p$-divisible group   $\Lambda\otimes\overline{\mathbb{Y}}$ of signature $(0,n)$, which can therefore also serve as a framing object.  This is what we choose to do.

\begin{definition}\label{def:fullRZ}
 For an $\breve{\F}_p$-scheme $S$,  denote by $\mathrm{RZ}(S)$ the set of isomorphism classes of triples $(X,\lambda_X,\varrho_X)$ in which 
\begin{itemize}
\item
$X$ is a $p$-divisible group over $S$ equipped with an $\co_E$-action of signature $(2,n-2)$, 
\item
$\lambda_X : X \to X^\vee$ is an $\co_E$-linear principal polarization (using the convention of \eqref{conjugate convention}, as always), 
\item
$
\varrho_X :  X \dashrightarrow     \Lambda   \otimes  \overline{\mathbb{Y}}_S  
$
is an $\co_E$-linear quasi-isogeny identifying polarizations up to $\Q_p^\times$-scaling.
\end{itemize}
\end{definition}

The results of \cite{RZbook} show that the functor $\mathrm{RZ}$ is represented by a formal scheme over $\breve{\F}_p$,  locally formally of finite type, and formally smooth of dimension $2n-4$.
As noted in the introduction, its underlying reduced scheme will turn out to have dimension $n-2$.
When no confusion can arise, we usually write $X \in \mathrm{RZ}(S)$ instead of $(X,\lambda_X,\varrho_X)\in \mathrm{RZ}(S)$.

For any point $X \in \mathrm{RZ}(S)$ there is a commutative diagram 
\begin{equation}\label{more quasi}
\xymatrix{
&   {  X  }  \ar@{-->}[rd]^{j^\vee}  \ar@{-->}[dd]_{\varrho_X} \\
{p\Lambda \otimes \overline{\mathbb{Y}}_S  } \ar@{-->}[ur]^{ j } \ar[dr]_i  & &  {p^{-1}\Lambda \otimes \overline{\mathbb{Y}}_S  }  \\
 &  { \Lambda \otimes \overline{\mathbb{Y} }_S   }  \ar[ur]_{i^\vee},
}
\end{equation}
in which the isogenies  $i$ and $i^\vee$ are induced by the inclusions $p\Lambda \subset \Lambda \subset p^{-1}\Lambda$, and the quasi-isogenies  $j$ and $j^\vee$ are uniquely determined by the commutativity.  
One can easily check that the triangle on the right is, as the notation suggests, obtained by dualizing the triangle on the left  using the principal polarizations on $X$ and $\Lambda \otimes \overline{\mathbb{Y}}_S$.

\begin{definition}\label{def:partialRZ}
The \emph{partial Rapoport-Zink space of signature $(2,n-2)$}  is the closed formal subscheme $\mathrm{RZ}_\Lambda \subset \mathrm{RZ}$ cut out by the following conditions:
\begin{itemize}
\item
The quasi-isogeny $\varrho_X :  X \dashrightarrow    \mathbb{X}_S$ identifies the polarizations on source and target (not just up to scaling).
\item
The arrows $j$ and $j^\vee$ in \eqref{more quasi} are isogenies (not just quasi-isogenies).
\end{itemize} 
\end{definition}

\begin{remark}\label{rem:projectivity}
By the argument of Lemma 4.2 of  \cite{VollaardWedhorn}  (or \S 2.2 of \cite{RZbook}), the functor $\mathrm{RZ}_\Lambda$ on $\breve{\F}_p$-schemes is represented by a  projective scheme (not just a formal scheme), denoted the same way.
\end{remark}

\begin{remark}\label{rem:domain}
  Let  $G=\mathrm{GU}(\Lambda)$, a reductive group scheme  over $\breve{\Z}_p$ whose  group of $\Q_p$-points  acts on $\mathrm{RZ}$.
Indeed, any  $ \gamma \in G(\Q_p)$  determines an $\co_E$-linear quasi-isogeny from  $\mathbb{X}=  \Lambda  \otimes \overline{\mathbb{Y}} $ to itself, respecting the polarization up to $\Q_p^\times$-scaling, and hence defines an automorphism of $\mathrm{RZ}$ by
\[
 (X,\lambda_X,\varrho_X)  \mapsto (X,\lambda_X, \gamma \circ \varrho_X).
\]
The partial Rapoport-Zink space of Definition \ref{def:partialRZ} satisfies
\begin{equation}\label{translate cover}
\mathrm{RZ}(\breve{\F}_p) =
 \bigcup_{\gamma \in G(\Q_p)} \gamma\cdot \mathrm{RZ}_\Lambda(\breve{\F}_p)  ,
\end{equation}
and so  serves as a kind of approximate fundamental domain for this action.

The equality \eqref{translate cover}  should be thought of as a signature $(2,n-2)$ analogue of the crucial Lemma 2.1 of \cite{Vollaard}, and it is possible to give a direct (if quite technical) linear algebraic proof along the same lines.  
We omit this argument because  \eqref{translate cover} will fall out during the proof of Corollary \ref{cor:component cover} below, but that argument makes use of the  results of \cite{XZ}, and so is of a far less elementary nature than the approach of \cite{Vollaard}. 
\end{remark}

 
 \subsection{Description of the closed points}
 \label{ss:partial points}
 

In this subsection we give a concrete description of the closed points of the partial Rapoport-Zink space $\mathrm{RZ}_\Lambda$ of Definition \ref{def:partialRZ}  in terms of lattices.
Endow 
\[
\breve{\Lambda} = \Lambda \otimes_{\Z_p} \breve{\Z}_p
\]
 with the $\sigma$-semi-linear automorphism $\Phi = \mathrm{id} \otimes \sigma$.   Recalling \eqref{general decomp}, this operator interchanges the direct summands 
$ \breve{\Lambda}_0$ and $\breve{\Lambda}_1$.

The hermitian form $h$ on $\Lambda$ extends $\breve{\Z}_p$-bilinearly to a pairing
\[
h : \breve{\Lambda} \times \breve{\Lambda} \to \co_{\breve{E}}
\]
satisfying 
\begin{equation}\label{hPhi}
h (\Phi x,\Phi y) = h ( x, y) ^\sigma 
\end{equation}
 where the $\sigma$ on the right  is the Frobenius on the second factor of $\co_{\breve{E} }=  \co_E\otimes_{\Z_p} \breve{\Z}_p$.
We further extend $h$ to a  $\breve{\Q}_p$-bilinear pairing on $\breve{\Lambda}[1/p]$.

\begin{proposition}\label{prop:hyperspecial points}
There is a bijection $X \mapsto L(X)$  from $\mathrm{RZ}_\Lambda(\breve{\F}_p)$ to the set of  self-dual $\co_{\breve{E}}$-lattices $L \subset \breve{\Lambda} [1/p]$  lying between $p\breve{\Lambda}$ and $p^{-1}\breve{\Lambda}$, and satisfying
\[
 pL_0  \stackrel{n-2}{\subset} \Phi^{-1} L_1 \stackrel{2}{\subset} L_0
\quad  \mbox{and}\quad
p L_1  \stackrel{2}{\subset} p \Phi^{-1} L_0  \stackrel{n-2}{\subset}   L_1.
\]
\end{proposition}

\begin{proof}
As this is the routine identification (see especially  Proposition 1.10 of \cite{Vollaard}) of points of $\mathrm{RZ}(\breve{\F}_p)$ with lattices in the isocrystal of the framing object, 
we only  explain how to account for our unusual choice of framing object.

Fix an $\co_{\breve{E}}$-linear isomorphism
\[
D(\overline{\mathbb{Y}}) \iso \co_{\breve{E}},
\]
and use this to identify
\begin{equation}\label{contraction}
D(\Lambda \otimes \overline{\mathbb{Y}})
\iso 
\breve{\Lambda} \otimes_{\co_{\breve{E}}} D(\overline{\mathbb{Y}}) 
\iso \breve{\Lambda}.
\end{equation}
This identification is only well-defined up to scaling by $\co_{\breve{E}}^\times$, but as we are only interested in viewing lattices in the left hand side as lattices in the right hand side, this ambiguity is harmless.

The signature $(0,1)$ condition on $\overline{\mathbb{Y}}$ implies that 
\[
V D( \overline{\mathbb{Y}})_1 = D( \overline{\mathbb{Y}})_0 
\quad\mbox{and}\quad
V D( \overline{\mathbb{Y}})_0 = pD( \overline{\mathbb{Y}})_1 .
\]
Using this, one checks that the identification  \eqref{contraction} restricts to identifications
\begin{equation}\label{contraction diagrams}
\xymatrix{
{  \breve{\Lambda}_0 } \ar[d]_{  p \Phi   }   \ar@{=}[r]   &  { D(\Lambda \otimes \overline{\mathbb{Y}})_0     } \ar[d]^{  F  }    \\
{   \breve{\Lambda}_1 } \ar@{=}[r]     &  { D(\Lambda \otimes \overline{\mathbb{Y}})_1    } 
}
\qquad \qquad
\xymatrix{
{    \breve{\Lambda}_1} \ar[d]_{ \Phi   }   \ar@{=}[r]   &  { D(\Lambda \otimes \overline{\mathbb{Y}})_1    } \ar[d]^{  F  }    \\
{  \breve{\Lambda}_0  } \ar@{=}[r]     &  { D(\Lambda \otimes \overline{\mathbb{Y}})_0   } 
} 
\end{equation}
in which the diagrams  commute up to scaling by $\breve{\Z}_p^\times$.   In fact, one can choose the trivialization of $D(\overline{\mathbb{Y}})$ so that they commute without any scaling factor, but we have no need to do this.

Given a point  $X \in \mathrm{RZ}_\Lambda(\breve{\F}_p)$, we use the quasi-isogeny $\varrho_X :X \dashrightarrow \Lambda \otimes \overline{\mathbb{Y}}$ to define 
\begin{equation}\label{Ldef}
 L(X) = D(X)  \stackrel{\varrho_X}{\subset} D(\Lambda \otimes \overline{\mathbb{Y}}) [1/p]
 \stackrel{ \eqref{contraction}} {=} \breve{\Lambda}[1/p] . 
 \end{equation}
The signature condition on the $\co_E$-action is equivalent to
\[
pD(X)_0 \stackrel{n-2}{\subset} VD(X)_1 \stackrel{2}{\subset} D(X)_0 \quad \mbox{and}\quad
pD(X)_1 \stackrel{2}{\subset} VD(X)_0 \stackrel{n-2}{\subset} D(X)_1,
\]
which, using \eqref{contraction diagrams}, translates to the inclusions
\[
 pL_0  \stackrel{n-2}{\subset} \Phi^{-1} L_1 \stackrel{2}{\subset} L_0
\quad  \mbox{and}\quad
p L_1  \stackrel{2}{\subset} p \Phi^{-1} L_0  \stackrel{n-2}{\subset}   L_1.
\]
This defines the desired bijection.
\end{proof}

\begin{remark}
A more functorial characterization of $L(X)$, not requiring a choice of \eqref{contraction},  is 
\begin{equation}\label{canonical L}
L(X) = \Hom_{\co_{\breve{E}}}(D(\mathbb{\overline{Y}}) , D(X) ),
\end{equation}
viewed as a lattice in
\[
\Hom_{\co_{\breve{E}}}(D(\mathbb{\overline{Y}}) , D( \Lambda \otimes \overline{\mathbb{Y}} ))[1/p] 
= \breve{\Lambda}[1/p]
\]
using the quasi-isogeny $\varrho_X :X \dashrightarrow \Lambda \otimes \overline{\mathbb{Y}}$.
\end{remark}

To  rewrite Proposition \ref{prop:hyperspecial points} in terms of lattices in $\breve{\Lambda}_0[1/p]$, 
define 
\begin{equation}\label{twisted pairing}
b   : \breve{\Lambda}_0  \times \breve{\Lambda}_0 \to \breve{\Z}_p
\end{equation}
 by 
\[
b( x,y ) = \mathrm{Tr}\, h(x,\Phi y) ,
\]
where $\mathrm{Tr} : \co_{\breve{E}} \to \breve{\Z}_p$ is the trace.
As $h(x  , \Phi y ) \in  e_0 \co_{\breve{E}}$,  we can equivalently characterize \eqref{twisted pairing} by the relation
\[
h(x,\Phi y) = (  b(x,y) , 0  )  \in  \breve{\Z}_p \times \breve{\Z}_p \stackrel{\eqref{orthopotents}}{=}  \co_{\breve{E}} .
\]
The  pairing \eqref{twisted pairing} is $\breve{\Z}_p$-linear in the first variable,  $\sigma$-linear in the second, and satisfies
\[
b( \Phi^2 x,y )  =  \mathrm{Tr}\,  h( \Phi^2 x, \Phi y ) \stackrel{\eqref{hPhi}}{=} 
 \mathrm{Tr}\,  h( \Phi  x,  y )^\sigma 
 =  \mathrm{Tr}\,   h( y, \Phi  x  ) ^\sigma
=b( y,x ) ^\sigma.
\]

 Any $\breve{\Z}_p$-lattice  $L_0 \subset \breve{\Lambda}_0[1/p]$ has a right dual lattice
\begin{equation}\label{right dual}
L_0^* = \{ x \in  \breve{\Lambda}_0[1/p] :  b(   L_0, x ) \subset \breve{\Z}_p \}
\end{equation}
with respect to \eqref{twisted pairing}, and one can easily verify the relations  
\begin{equation}\label{twisted dual properties}
L_0^{**} = \Phi^{-2} L_0 \quad \mbox{and}\quad \Phi^2(L_0^*) = (\Phi^2 L_0)^*.
\end{equation}
If $L \subset \breve{\Lambda}[1/p]$ is an $\co_{\breve{E}}$-lattice self-dual under $h$, then 
\begin{equation}\label{dual shift}
\Phi^{-1} L_1 = L_0^* .
\end{equation}
In particular, 
$
\breve{\Lambda}_0 = \Phi^{-1} \breve{\Lambda}_1 = \breve{\Lambda}_0^*.
$

\begin{corollary}\label{cor:half hyperspecial}
There is a bijection
\[
\mathrm{RZ}_\Lambda( \breve{\F}_p) 
\iso
\left\{
\begin{array}{c}  \breve{\Z}_p\mbox{-lattices }     L_0  \subset  \breve{\Lambda}_0[1/p]    \\
\mbox{satisfying} \\
  p  \breve{\Lambda}_0  \subset L_0^* \stackrel{2}{\subset} L_0  \subset p^{-1} \breve{\Lambda}_0  
\end{array}
\right\} ,
\]
under which  $X \in \mathrm{RZ}_\Lambda( \breve{\F}_p)$ corresponds, as in \eqref{Ldef}, to
\[
L_0^*  = VD(X)_1  \quad \mbox{and} \quad 
L_0    = D(X)_0   .
\]
\end{corollary}

\begin{proof}
Using  \eqref{dual shift},   Proposition \ref{prop:hyperspecial points} is equivalent to 
\[
\mathrm{RZ}_\Lambda( \breve{\F}_p) 
\iso
\left\{
\begin{array}{c}  \breve{\Z}_p\mbox{-lattices }   \\  L_0  \subset  \breve{\Lambda}_0[1/p]  \end{array} : 
\begin{array}{c}
pL_0 \subset L_0^* \stackrel{2}{\subset} L_0\\
  p  \breve{\Lambda}_0  \subset  L_0  \subset p^{-1} \breve{\Lambda}_0 
  \end{array}  \right\} .
\]
The two chains of inclusions on the right hand side are equivalent to the single chain of inclusions in the statement of the corollary.
\end{proof}


\subsection{Decomposition by locally closed subsets}
\label{ss:hyperspecial strata}


Now let $X$ be the universal $p$-divisible group over $\mathrm{RZ}_\Lambda$.
 The bijection of Proposition \ref{prop:hyperspecial points} associates to every  $s \in \mathrm{RZ}_\Lambda(\breve{\F}_p)$ an inclusion of $\co_{\breve{E}}$-lattices
 \[
 p \breve{\Lambda} \subset L(X_s),
 \]
and hence a map of  $\breve{\F}_p$-vector spaces 
 \begin{equation}\label{stratum map}
  p \breve{\Lambda} / p^2 \breve{\Lambda} \to L(X_s)/pL(X_s).
 \end{equation}
 One can use Grothendieck-Messing crystals  to construct a morphism of vector bundles interpolating these maps as $s$ varies.

If $S$ is a scheme on which $p$ is locally nilpotent,   covariant  Grothendieck-Messing theory functorially associates to any $p$-divisible group $G$ over $S$ a short exact sequence
\begin{equation}\label{GMbundles}
0 \to \mathrm{Fil}^0 \mathscr{D}(G)  \to \mathscr{D}(G) \to \Lie(G) \to 0
\end{equation}
 of locally free $\co_S$-modules.  When $S=\Spec(\breve{\F}_p)$, this sequence is canonically identified with
 \[
 0 \to \frac{VD(G)}{pD(G)}   \to\frac{D(G)}{pD(G)} \to \frac{D(G)}{VD(G)} \to 0.
 \]

Over $\mathrm{RZ}_\Lambda$ there is a universal diagram
\begin{equation}\label{complete point diagram}
\xymatrix{
&   {  X  }  \ar[rd]^{j^\vee}  \ar@{-->}[dd]_{\varrho_X} \\
{p\Lambda \otimes \overline{\mathbb{Y}}  } \ar[ur]^{ j } \ar[dr]_i  & &  {p^{-1}\Lambda \otimes \overline{\mathbb{Y}}_S  }  \\
 &  { \Lambda \otimes \overline{\mathbb{Y} }   }  \ar[ur]_{i^\vee}
}
\end{equation}
(where we now view  $\overline{\mathbb{Y}}$ as a constant $p$-divisible group over $\mathrm{RZ}_\Lambda$).
Mimicking \eqref{canonical L}, we define a  vector bundle
 \begin{equation}\label{the L bundle}
\mathscr{L} =  \underline{\Hom}_{\co_E} ( \mathscr{D}(  \overline{\mathbb{Y}} ) ,  \mathscr{D}(X) ) 
 \end{equation}
with $\co_E$-action on $\mathrm{RZ}_\Lambda$.
The isogeny $j$ determines an inclusion 
$
p \Lambda  \subset \Hom_{\co_E} ( \overline{\mathbb{Y}} , X ) ,
$
and hence an  $\co_E$-linear  morphism 
\[
 p \Lambda   \otimes_{ \Z_p} \co_{ \mathrm{RZ}_\Lambda} 
\to  \mathscr{L} 
\]
whose fiber at any $s\in \mathrm{RZ}_\Lambda(\breve{\F}_p)$ is  canonically identified with \eqref{stratum map}.
It restricts to morphisms of vector bundles
\begin{equation}\label{c halves}
 p \breve{\Lambda}_0  \otimes_{ \breve{\Z}_p} \co_{ \mathrm{RZ}_\Lambda}  \to \mathscr{L}_0 
 \quad\mbox{and}\quad
p \breve{\Lambda}_1  \otimes_{ \breve{\Z}_p} \co_{ \mathrm{RZ}_\Lambda}  \to \mathscr{L}_1 .
\end{equation}

\begin{remark}\label{rem:untwist L}
A choice of isomorphism $D(\overline{\mathbb{Y}}) \iso \co_{\breve{E}}$ determines an isomorphism
$\mathscr{D}( \overline{\mathbb{Y}} ) \iso  \co_E \otimes_{\Z_p} \co_{\mathrm{RZ}_\Lambda}$, and hence also an isomorphism
\[
\mathscr{L} \iso \mathscr{D}(X).
\]
This  is  well-defined up to scaling by $( \co_{\breve{E}}  /  p\co_{\breve{E}}) ^\times$.
\end{remark}

\begin{definition}\label{def:RZk}
For any $k\ge 1$, define
\[
 \mathrm{RZ}_\Lambda^{\le k} \subset  \mathrm{RZ}_\Lambda
 \]
to be the largest closed subscheme over which the first map in \eqref{c halves}  has rank $\le k-1$.  More precisely, $\mathrm{RZ}_\Lambda^{\le k}$ is the closed subscheme cut out by the vanishing of 
\[
  \bigwedge\nolimits^k     ( p \breve{\Lambda}_0  \otimes_{ \breve{\Z}_p} \co_{ \mathrm{RZ}_\Lambda}   )   \to  \bigwedge\nolimits^k  \mathscr{L}_0 .
\]
We further define
$
\mathrm{RZ}_\Lambda^k = \mathrm{RZ}_\Lambda^{\le k} \smallsetminus \mathrm{RZ}_\Lambda^{\le k-1},
$
so that we have a decomposition 
\[
\mathrm{RZ}^\red_\Lambda= \bigsqcup_{k \ge 1} \mathrm{RZ}_\Lambda^{k,\red}
\]
into locally closed subschemes.
\end{definition}

\begin{definition}\label{def:GLinvariant}
Two $\breve{\Z}_p$-lattices $L$ and $L'$ in an $n$-dimensional $\breve{\Q}_p$-vector space (e.g.~$\breve{\Lambda}_0[1/p]$),  have \emph{relative position invariant} 
\[
\inv(L , L') = (a_1, \ldots, a_n)  \in   \Z^n
\]
if there is a  $\breve{\Z}_p$-basis $x_1,\ldots, x_n \in L'$ such that $p^{a_1} x_1 , \ldots, p^{a_n} x_n$  is a basis of $L$, and  $a_1 \ge a_2 \ge \cdots \ge a_n$.
\end{definition}

We now give a concrete description of the $\breve{\F}_p$-valued points of $\mathrm{RZ}_\Lambda^k$, in terms of lattices in $\breve{\Lambda}_0[1/p]$.
In  the notation of Definition \ref{def:GLinvariant},  the bijection of Corollary \ref{cor:half hyperspecial} becomes
\begin{align}\label{better half hyperspecial}
\mathrm{RZ}_\Lambda( \breve{\F}_p) 
 &  \iso
\left\{
\begin{array}{c}  \breve{\Z}_p\mbox{-lattices }   \\  L_0  \subset  \breve{\Lambda}_0[1/p]  \end{array} : 
\begin{array}{c}
  p  \breve{\Lambda}_0  \subset  L_0  \subset p^{-1} \breve{\Lambda}_0 \\
\mathrm{inv}(  L_0^*,   L_0) = (1,1,0, \ldots,0)  
  \end{array}  \right\} .
\end{align}

 \begin{proposition}\label{prop:strata points}
 Under the above bijection,  for any $1\le k \le \lfloor n/2\rfloor$ we have
    \[
\mathrm{RZ}_\Lambda^k( \breve{\F}_p) 
\iso  
\{ L_0  \in \mathrm{RZ}_\Lambda( \breve{\F}_p) : 
\inv(L_0 , \breve{\Lambda}_0) = \lambda_k \} ,
\]
in which 
\[
\lambda_k=( \underbrace{1,\ldots, 1}_{k-1\mathrm{\ times}}, 0,\ldots, 0 , \underbrace{-1,\ldots, -1}_{k\mathrm{\ times}} ) \in \Z^n.
\]
If $k>\lfloor n/2\rfloor$ then $\mathrm{RZ}_\Lambda^k = \emptyset$.
\end{proposition}

\begin{proof}
 Fix any $L_0 \in \mathrm{RZ}_\Lambda( \breve{\F}_p)$.
 The conditions of \eqref{better half hyperspecial} imply 
\[
\inv(L_0 , \breve{\Lambda}_0)=(b_1, \ldots ,b_n)
\]
with all $-1\le b_i \le 1$, and dualizing shows that 
\[
 \inv( L_0^*  , \breve{\Lambda}_0 ) =   \inv(  L_1  , \breve{\Lambda}_1 ) =  (-b_n, \ldots , - b_1).
\]
These imply the equalities of lattices 
\[
  \bigwedge\nolimits^n L_0^*  = p^{-(b_1+\cdots + b_n)} 
\bigwedge\nolimits^n\breve{\Lambda}_0 = p^{-2(b_1+\cdots + b_n)}   \bigwedge\nolimits^n L_0
\]
in $\bigwedge\nolimits^n\breve{\Lambda}_0[1/p]$.  Combining this with $\mathrm{inv}(  L_0^*,   L_0) = (1,1,0, \ldots,0)$ shows that
\[ b_1+ \cdots + b_n =-1,\] which is equivalent to  $\inv(L_0 , \breve{\Lambda}_0)= \lambda_k$ for some $1\le k \le \lfloor n/2\rfloor$.

The fiber at $L_0 \in \mathrm{RZ}_\Lambda( \breve{\F}_p)$  of the first morphism in  \eqref{c halves} is identified with the linear map
\[
 \frac{ p \breve{\Lambda}_0 }{  p^2 \breve{\Lambda}_0 }    \to \frac{L_0}{pL_0}
\] 
of  rank
 $
 \mathrm{dim}_{\breve{\F}_p} ( \breve{\Lambda}_0 /   ( \breve{\Lambda}_0 \cap L_0  )  ) =k-1,
 $
proving that $L_0 \in \mathrm{RZ}^k_\Lambda( \breve{\F}_p)$.
\end{proof}


\section{Enhancing the moduli problem}
\label{s:enhanced}


We continue to study the partial Rapoport-Zink space $\mathrm{RZ}_\Lambda$ associated to a self-dual hermitian $\co_E$-lattice $\Lambda$ of rank $n\ge 2$, and the  locally closed subset
$
\mathrm{RZ}^k_\Lambda \subset \mathrm{RZ}_\Lambda
$
 determined by  a fixed integer $1\le k\le \lfloor n/2 \rfloor$.
Our goal is to construct a morphism from $\mathrm{RZ}^k_\Lambda$  to a Deligne-Lustzig variety.
The  fibers of this morphism will then be studied in \S \ref{s:fibration}.


\subsection{New moduli spaces}
\label{ss:new moduli}


In this subsection  we  will construct a commutative diagram 
\begin{equation}\label{augment diagram}
\xymatrix{
{ \mathrm{RZ}^ k_\Lambda }    \ar[d] & { \mathbf{RZ}^k_\Lambda } \ar[d]  \ar[l]_{\pi_0} \ar[r]^{\pi_1}   &  {  Y^k_\Lambda }  \ar@{=}[d] \\
{ \mathrm{RZ}^{\le k}_\Lambda }    & { \mathbf{RZ}^{\le k}_\Lambda } \ar[l]_{\pi_0} \ar[r]^{\pi_1}   &  {  Y^k_\Lambda } 
}
\end{equation}
of  $\breve{\F}_p$-schemes, in which the schemes in the bottom row are projective,  the vertical arrows are open immersions, and the vertical arrow on the left is that of Definition \ref{def:RZk}.
The definitions are somewhat elaborate;   they are concocted so that  the $\pi_0$ in the top row induces an isomorphism of underlying reduced schemes (Proposition \ref{prop:pi_0 isomorphism}),  while the reduced scheme underlying $Y_\Lambda^k$ is isomorphic to a Deligne-Lusztig variety (Theorem \ref{thm:YtoDL})

Define $\widetilde{Y}_\Lambda^{k}$ to be the $\breve{\F}_p$-scheme whose $S$-valued points are  commutative diagrams
\begin{equation}\label{Y tilde functor}
\xymatrix{
{p\Lambda \otimes \overline{\mathbb{Y}}_S  }\ar[r]^a    \ar@/_/[drr]_i  &  { H }  \ar[dr]^b  \ar[r]|{\, d\, }  &  {G}  \ar[r] | {\, d^\vee\, } & {H^\vee}  \ar[r]^{a^\vee}  & {p^{-1}\Lambda \otimes  \overline{\mathbb{Y}}_S  } \\
& &  { \Lambda \otimes  \overline{\mathbb{Y} }_S } \ar[ur]^{b^\vee} \ar@/_/[urr]_{i^\vee}
}
\end{equation}
in which $H$ and $G$ are $p$-divisible groups  over $S$  equipped  with  $\co_E$-actions, 
$H$ has  signature $(1,n-1)$, $G$ has signature $(2k, n-2k)$ and is equipped with an $\co_E$-linear principal polarization, 
the  isogenies   $a$, $b$, and $d$  are $\co_E$-linear  of heights  $\mathrm{ht}(a)=2n-2k+1$ and 
$
\mathrm{ht}(b) = \mathrm{ht}(d) =2k-1,
$
and 
\begin{equation}\label{Hpol}
\ker(  d^\vee \circ d : H \to H^\vee ) \subset H[p].
\end{equation}

Define $\widetilde{\mathbf{RZ}}^{\le k}_\Lambda$ to be the $\breve{\F}_p$-scheme whose $S$-valued points consist of a point \eqref{complete point diagram}  of $\mathrm{RZ}^{\le k}_\Lambda(S)$, a point \eqref{Y tilde functor} of $\widetilde{Y}^k_\Lambda(S)$, and an $\co_E$-linear isogeny 
$
c : H\to X
$
making the diagram
\begin{equation}\label{RZ tilde functor}
\xymatrix{
& &   {  X  } \ar[dr]_{c^\vee}  \ar@/^/[drr]^{j^\vee} \\
{p\Lambda \otimes \overline{\mathbb{Y}}_S  }\ar[r]^a  \ar@/^/[urr]^j   \ar@/_/[drr]_i  &  { H }  \ar[ur]_c \ar[dr]^b  \ar[r]|{\, d\, }  &  {G}  \ar[r] | {\, d^\vee\, } & {H^\vee}  \ar[r]^{a^\vee}  & {p^{-1}\Lambda \otimes  \overline{\mathbb{Y}}_S  } \\
& &  { \Lambda \otimes  \overline{\mathbb{Y} }_S } \ar[ur]^{b^\vee} \ar@/_/[urr]_{i^\vee}
}
\end{equation}
commute.

The moduli spaces we have constructed are  related to $\mathrm{RZ}_\Lambda$ by morphisms
\[
\xymatrix{
{ \mathrm{RZ}_\Lambda^{\le k} }  & { \widetilde{\mathbf{RZ}}_\Lambda^{\le k} } \ar[l]_{\pi_0} \ar[r]^{\pi_1} & {  \widetilde{Y}^k_\Lambda } ,
}
\]
where $\pi_0$ extracts from \eqref{RZ tilde functor} the diagram \eqref{complete point diagram}, while 
$\pi_1$ extracts  \eqref{Y tilde functor}.
To define the bottom row of \eqref{augment diagram}, we single out certain open and closed subschemes of the these moduli spaces.  This will require the following general lemma concerning the vector bundles of \eqref{GMbundles}.

\begin{lemma}\label{lem:free crystals}
Suppose $\phi : G \to H$ is an isogeny of $p$-divisible groups over a scheme $S$ satisfying $p\co_S=0$.
If the kernel of $\phi$ is annihilated by $p$, then the  cokernel of 
\[
\phi : \mathscr{D}(G) \to \mathscr{D}(H) 
\]
is a locally free $\co_S$-module of rank $\mathrm{ht}(\phi)$.
In particular,  its kernel and image are local direct summands of $\mathscr{D}(G)$ and $\mathscr{D}(H)$, respectively.
\end{lemma}

\begin{proof}
This is well-known.  See  Proposition 4.6 of \cite{VollaardWedhorn}, for example.
\end{proof}

Applying the functor $\mathscr{D}$  to the universal diagram over $\widetilde{Y}^k_\Lambda$  yields a diagram of vector bundles
\[
\xymatrix{
{p  \breve{\Lambda}  \otimes \mathscr{D}(\overline{\mathbb{Y}})   }\ar[r]^a    \ar@/_/[drr]_{i =0 }  &  {\mathscr{D}(H)} \ar[dr]^b  \ar[r]|{\, d\, }   &  
{\mathscr{D}(G) } \ar[r]|{\, d^\vee \, }  & {\mathscr{ D }(H^\vee)  }  \ar[r]^{a^\vee}  & {p^{-1}   \breve{\Lambda}    \otimes \mathscr{D}(\overline{\mathbb{Y}})     } \\
& &  {   \breve{\Lambda}   \otimes \mathscr{D}(\overline{\mathbb{Y}})    } \ar[ur]^{b^\vee} \ar@/_/[urr]_{i^\vee=0},
}
\]
where the tensor products are over $\co_{\breve{E}}$.  Define an open and closed subscheme  $Y^k_\Lambda \subset \widetilde{Y}^k_\Lambda$  by imposing the conditions
\begin{align}
\mathrm{rank}(  b  : \mathscr{D}(H)_{0} \to 
  \breve{\Lambda}_0   \otimes \mathscr{D}(\overline{\mathbb{Y}})_0   )  & =  n-k+1   \label{Y invariants} \\
\mathrm{rank}(  b  : \mathscr{D}(H)_{1} \to  \breve{\Lambda}_1   \otimes \mathscr{D}(\overline{\mathbb{Y}}) _1 )   & =  n-k  \nonumber \\
\mathrm{rank}(  d  : \mathscr{D}(H)_{0} \to  \mathscr{D}(G)_{0}  )  & = n-2k+1  \nonumber \\
\mathrm{rank}( d^\vee  : \mathscr{D}(G)_{1} \to \mathscr{D}(H^\vee)_{1}  ) & = n-2k+1  \nonumber ,
\end{align}
where \emph{rank} means the rank of the image as a vector bundle.  
Note we are using Lemma \ref{lem:free crystals}: the image of every arrow in the diagram is   a locally free sheaf, so the ranks here are  well-defined  locally constant functions on $\widetilde{Y}^k_\Lambda$.

Similarly, over $\widetilde{\mathbf{RZ}}^{\le k}_\Lambda$ there is a diagram of vector bundles
\[
\xymatrix{
& &   {  \mathscr{D}(X) } \ar[dr]_{c^\vee}  \ar@/^/[drr]^{j^\vee} \\
{p  \breve{\Lambda}  \otimes \mathscr{D}(\overline{\mathbb{Y}})   }\ar[r]^a  \ar@/^/[urr]^{j}   \ar@/_/[drr]_{i =0 }  &  {\mathscr{D}(H)} \ar[ur]_c \ar[dr]^b  \ar[r]|{\, d\, }   &  
{\mathscr{D}(G) } \ar[r]|{\, d^\vee \, }  & {\mathscr{ D }(H^\vee)  }  \ar[r]^{a^\vee}  & {p^{-1}   \breve{\Lambda}    \otimes \mathscr{D}(\overline{\mathbb{Y}})     } \\
& &  {   \breve{\Lambda}   \otimes \mathscr{D}(\overline{\mathbb{Y}})    } \ar[ur]^{b^\vee} \ar@/_/[urr]_{i^\vee=0},
}
\]
and we denote by $\mathbf{RZ}^{\le k}_\Lambda \subset \widetilde{\mathbf{RZ}}^{\le k}_\Lambda$  the open and closed subscheme defined by imposing the conditions  \eqref{Y invariants} as well as
\begin{align}
\mathrm{rank}(c  : \mathscr{D}(H)_{0} \to  \mathscr{D}(X)_{0}  )  & = n-k  \label{RZ invariants} \\
\mathrm{rank}(  c  : \mathscr{D}(H)_{1} \to \mathscr{D}(X)_{1}  ) & = n-k+1  \nonumber .
\end{align}

The above definitions complete the construction of the bottom row of  \eqref{augment diagram}.
All three of the schemes appearing there are projective, by the same reasoning as in Remark \ref{rem:projectivity}.  
To complete the construction of  the top row,  define 
 \[
 \mathbf{RZ}_\Lambda^k \subset  \mathbf{RZ}_\Lambda^{\le k}
  \]
 as the preimage   under $\pi_0 :  \mathbf{RZ}_\Lambda^{\le k}  \to  \mathrm{RZ}_\Lambda^{\le k}$
of the open subscheme  $\mathrm{RZ}_\Lambda^k  =  \mathrm{RZ}^{\le k}_\Lambda \smallsetminus  \mathrm{RZ}^{\le k-1}_\Lambda$  constructed in   \S \ref{ss:hyperspecial strata}.


\subsection{Description of the closed points}


We now provide a description of the closed points of the schemes in \eqref{augment diagram}, 
analogous the description of $\mathrm{RZ}^k_\Lambda(\breve{\F}_p)$ from Proposition  \ref{prop:strata points}.

\begin{proposition}\label{prop:RZ tilde points}
There is a   bijection from $Y^k_\Lambda(\breve{\F}_p)$ to the set of pairs of $\breve{\Z}_p$-lattices $(M_0,N_0)$ in $\breve{\Lambda}_0[1/p]$ 
such that
\[
 p\breve{\Lambda}_0 \stackrel{k-1}{\subset}  p M_0^* \stackrel{1}{\subset}  pN_0 \stackrel{n-2k}{\subset}  N_0^* \stackrel{1}{\subset}  M_0 \stackrel{k-1}{\subset} \breve{\Lambda}_0  .
 \]
Here  $M_0^*$ and $N_0^*$ are the right duals, as in \eqref{right dual}, of $M_0$ and $N_0$ with respect to the pairing \eqref{twisted pairing}.
 Under this bijection:
 \begin{enumerate}
 \item
 A point of $Y_\Lambda^k(\breve{\F}_p)$, represented by a diagram \eqref{Y tilde functor}, corresponds to 
 \begin{align*}
 M_0^* &= VD(H^\vee)_1  & N_0^*&= VD(H)_1   \\
   M_0 &= D(H)_0   &     N_0 &= D(H^\vee)_0
 \end{align*}
 all viewed as lattices in $\breve{\Lambda}_0[1/p]$ as in \eqref{Ldef}.
 
 \item
The points of $\mathbf{RZ}^{\le k}_\Lambda(\breve{\F}_p)$ above a given pair $(M_0,N_0)$ are in bijection with $\breve{\Z}_p$-lattices $ L_0 \subset \breve{\Lambda}_0[1/p]$ satisfying
 \[
 L_0^* \stackrel{2}{\subset} L_0\quad \mbox{and}\quad   M_0 \stackrel{k}{\subset} L_0 \stackrel{k-1}{\subset} N_0. 
 \]
 \item
 The points of $\mathbf{RZ}^k_\Lambda(\breve{\F}_p)$ above a given pair $(M_0,N_0)$ are in bijection with $\breve{\Z}_p$-lattices $ L_0 \subset \breve{\Lambda}_0[1/p]$ satisfying the conditions of (2) and, in the notation of  Proposition  \ref{prop:strata points},  $\inv(L_0 , \breve{\Lambda}_0) = \lambda_k$.
 \end{enumerate}
\end{proposition}

\begin{proof}
An element of  $Y^k_\Lambda(\breve{\F}_p)$, corresponding to a diagram \eqref{Y tilde functor} of $p$-divisible groups over $\breve{\F}_p$ satisfying \eqref{Y invariants},  determines inclusions of Dieudonn\'e modules 
\[
\xymatrix{
{p \breve{\Lambda} \otimes  D( \overline{\mathbb{Y}})  }\ar[r]     \ar@/_/[drr]_{2n}  &  {  D(H)  }  \ar[dr]^{2k-1}  \ar[r]^{ 2k-1 }&  {D(G)}  \ar[r]^{2k-1  } & { D(H^\vee)}  \ar[r]  & {p^{-1} \breve{\Lambda} \otimes  D(\overline{\mathbb{Y}})  } \\
& &  { \breve{\Lambda} \otimes  D(\overline{\mathbb{Y} }) } \ar[ur]^{2k-1} \ar@/_/[urr]_{2n}
}
\]
of the indicated colengths, and with $pD(H^\vee) \subset D(H)$.   Here all tensor products are over $\co_{\breve{E}}$.

The following lemma shows that, at least on the level of $\breve{\F}_p$-points, the $p$-divisible group  $G$ in the moduli problem defining $Y_\Lambda^k$ can be recovered from $H$ and $H^\vee$.  
In some sense $G$ plays only an auxiliary role, imposing constraints on the polarization $H\to H^\vee$.

\begin{lemma}\label{lem:intermediate crystal}
We have
$
D(G) = D(H^\vee)_0 \oplus D(H)_1. 
$
\end{lemma}

\begin{proof}
The final two conditions in \eqref{Y invariants} imply that the inclusions 
\[
D(H)_0  \subset D(G)_0 \quad \mbox{and}\quad D(G)_1\subset D(H^\vee)_1
\]
 each have colength $2k-1$.  Comparing with the colengths in the diagram above, we deduce that the inclusions 
 \[
 D(H)_1  \subset D(G)_1 \quad \mbox{and}\quad D(G)_0\subset D(H^\vee)_0
 \]
  are equalities.
\end{proof}

Using \eqref{contraction}, we view the Dieudonn\'e modules in the diagram above as $\co_{\breve{E}}$-lattices 
\begin{equation}\label{fake minuscule lattices}
\xymatrix{
{p \breve{\Lambda}   }\ar[r]     \ar@/_/[drr]  &  {M} \ar[dr]   \ar[r] & { N_0 \oplus M_1 }  \ar[r] & { N }  \ar[r]  & {p^{-1}  \breve{\Lambda }  } \\
& &  {  \breve{\Lambda} } \ar[ur] \ar@/_/[urr]
}
\end{equation}
in the hermitian space $ \breve{\Lambda}[1/p]$,  with $M$ and $N$  dual to one another.
As in  \eqref{canonical L}, one can characterize these lattices, without fixing \eqref{contraction}, by 
\begin{equation}\label{MNlattices}
M  = \Hom_{\co_{\breve{E}}} (D(\overline{\mathbb{Y}}) , D(H) ) 
\quad \mbox{and}\quad 
N  = \Hom_{\co_{\breve{E}}} (D(\overline{\mathbb{Y}}) , D(H^\vee) ) 
\end{equation}
viewed as lattices  in
\[
 \Hom_{\co_{\breve{E}}} (D(\overline{\mathbb{Y}}) ,    \Lambda \otimes  D(\overline{\mathbb{Y} })   )[1/p]  
\iso  \breve{\Lambda}[1/p]  .
\]

As $M$ and $N$ are dual to one another, each of   $M_0\oplus N_1$ and $N_0\oplus M_1$ is self-dual under the hermitian form on $\breve{\Lambda}[1/p]$.  Hence, by \eqref{dual shift},
\begin{equation}\label{fake dual shift}
\Phi^{-1} N_1 = M_0^* \quad  \mbox{and} \quad  \Phi^{-1} M_1 =N_0^*.
\end{equation}
 The signature $(1,n-1)$ conditions on $H$ and $H^\vee$ imply 
 \[
 VD(H)_1 \stackrel{1}{\subset}   D(H)_0   \quad  \mbox{and}\quad  VD(H^\vee)_1 \stackrel{1}{\subset}   D(H^\vee)_0 .
\]
Using \eqref{contraction diagrams}, these translate to 
\[
N_0^* \stackrel{\eqref{fake dual shift}}{=} \Phi^{-1} M_1 \stackrel{1}{\subset} M_0 
\quad\mbox{and}\quad  
M_0^*  \stackrel{\eqref{fake dual shift}}{=}   \Phi^{-1} N_1 \stackrel{1}{\subset} N_0.
\]
Similarly, the signature $(2k,n-2k)$ condition on $G$ implies 
\[
p  D(H^\vee)_0  =    p D(G)_0 
  \stackrel{n-2k}{\subset} 
  VD(G)_1 =   V D(H)_1,
\]
(the outer equalities are by Lemma \ref{lem:intermediate crystal}), which  translates to 
\[
pN_0 \stackrel{n-2k}{\subset}   \Phi^{-1} M_1 =N_0^*.
\]
The first  condition in \eqref{Y invariants} implies
$
  D(H)_0 \subset  \breve{\Lambda}_0   \otimes D(\overline{\mathbb{Y}})_0   ,
$
with colength $k-1$, 
which translates to $M_0 \subset \breve{\Lambda}_0$ with colength $k-1$, and  dualizing shows 
\[
\breve{\Lambda}_0 = \Phi^{-1} \breve{\Lambda}_1 
\stackrel{\eqref{dual shift}}{=} 
 \breve{\Lambda}_0^* \stackrel{k-1}{\subset} M_0^* .
\]
All of this shows that the pair $(M_0,N_0)$ satisfies the chain of inclusions in the statement of the proposition.

This process can be reversed.  Starting from a pair $(M_0,N_0)$ one uses \eqref{fake dual shift} to define $M_1$ and $N_1$, thereby  obtaining a diagram \eqref{fake minuscule lattices} of lattices in $\breve{\Lambda}[1/p]$.  
Converting these to lattices in $D(\mathbb{X})[1/p]$ using 
 \eqref{contraction}, one finds inclusions of Dieudonn\'e modules whose corresponding $p$-divisible groups define a point of  $Y^k_\Lambda(\breve{\F}_p)$.

The analysis of $\mathbf{RZ}^{\le k}_\Lambda$ is entirely similar.
A point  of $\mathbf{RZ}_\Lambda^{\le k}(\breve{\F}_p)$, corresponding to a diagram of $p$-divisible groups \eqref{RZ tilde functor},   determines inclusions of Dieudonn\'e modules 
\[
\xymatrix{
& &   {  D(X)  } \ar[dr]^{2k-1}  \ar@/^/[drr]^{2n} \\
{p\Lambda \otimes D(\overline{\mathbb{Y}})  }\ar[r]  \ar@/^/[urr]^{2n}   \ar@/_/[drr]_{2n}  &  { D(H) }  \ar[ur]^{2k-1} \ar[dr]^{2k-1}  \ar[r]^{ 2k -1 }  &  {D(G)}  \ar[r]^ { 2k-1} & {D( H^\vee) }  \ar[r]   & {p^{-1}\Lambda \otimes  D(\overline{\mathbb{Y}})  } \\
& &  { \Lambda \otimes  D(\overline{\mathbb{Y} }) , } \ar[ur]^{2k-1} \ar@/_/[urr]_{2n} 
}
\]
and we use \eqref{contraction}  to convert these  into $\co_{\breve{E}}$-lattices
\begin{equation}\label{fake minuscule lattices 2}
\xymatrix{
& &   {  L } \ar[dr]  \ar@/^/[drr] \\
{p \breve{\Lambda}   }\ar[r]  \ar@/^/[urr]   \ar@/_/[drr]  &  {M} \ar[ur] \ar[dr]   \ar[r] & { N_0 \oplus M_1 }  \ar[r] & { N }  \ar[r]  & {p^{-1}  \breve{\Lambda }  } \\
& &  {  \breve{\Lambda} .} \ar[ur] \ar@/_/[urr]
}
\end{equation}
The first condition in  \eqref{RZ invariants} implies $D(H)_0 \subset D(X)_0$ with colength $k$,  which translates to 
$M_0\subset L_0$ with colength $k$.  As the inclusion $L_0 \subset N_0$ is obvious, and $L_0^* \subset L_0$ with colength $2$ by \eqref{better half hyperspecial}, the triple $(L_0,M_0,N_0)$ satisfies the properties stated in (2).

The description of $\mathbf{RZ}^k_\Lambda$ in part (3) follows immediately from Proposition  \ref{prop:strata points} and the description of $\mathbf{RZ}^{\le k}_\Lambda$, completing the proof of Proposition \ref{prop:RZ tilde points}.
\end{proof}

\begin{corollary} \label{cor:augment bijection} 
The map $\pi_0 :  \mathbf{RZ}_\Lambda^k  \to  \mathrm{RZ}_\Lambda^k$ is bijective  on $\breve{\F}_p$-points.  
\end{corollary}

\begin{proof}
Recalling the bijection
\begin{align*}
\mathrm{RZ}^k_\Lambda( \breve{\F}_p) 
 &  \iso
\left\{
\begin{array}{c}  \breve{\Z}_p\mbox{-lattices }   \\  L_0  \subset  \breve{\Lambda}_0[1/p]  \end{array} : 
\begin{array}{c}
\mathrm{inv}(  L_0^*,   L_0) = (1,1,0, \ldots,0)  \\
\inv(L_0 , \breve{\Lambda}_0) = \lambda_k
  \end{array}  \right\} .
\end{align*}
of Proposition \ref{prop:strata points},  fix a point $L_0 \in \mathrm{RZ}^k_\Lambda( \breve{\F}_p)$.
If $(L_0,M_0,N_0) \in \mathbf{RZ}_\Lambda^k(\breve{\F}_p)$ lies above it (using the identifications of Proposition \ref{prop:RZ tilde points}), then $M_0 \subset L_0 \cap \breve{\Lambda}_0$ and $L_0 + \breve{\Lambda}_0 \subset N_0$.    The first inclusion is an equality because both lattices have colength $k-1$ in $\breve{\Lambda}_0$.  The second inclusion is an equality because both lattices contain $\breve{\Lambda}_0$ with colength $k$.  In other words,
\begin{equation}\label{this MN}
M_0 = L_0 \cap \breve{\Lambda}_0 \quad \mbox{and}\quad N_0=L_0 + \breve{\Lambda}_0 ,
\end{equation}
showing that there is a unique point above $L_0$.  This proves the injectivity of the map in question.

For surjectivity  we again start with a point $L_0 \in  \mathrm{RZ}^k_\Lambda( \breve{\F}_p)$, and now define $M_0$ and $N_0$ by \eqref{this MN}. 
An exercise in linear algebra, using  Proposition \ref{prop:RZ tilde points}, shows that the triple $(L_0,M_0,N_0)$
 defines a point of $\mathbf{RZ}_\Lambda^k(\breve{\F}_p)$ above $L_0$.
\end{proof}

The proof of Corollary \ref{cor:augment bijection} provides a description of the inverse to 
\[
\pi_0 :  \mathbf{RZ}_\Lambda^k(\breve{\F}_p)   \to  \mathrm{RZ}_\Lambda^k(\breve{\F}_p) .
\]
  By continuing the line of reasoning a bit further, one can describe the inverse in the language  of the  original moduli problems defining the source and target.
  
Recall that for any point of $\mathbf{RZ}_\Lambda^k(\breve{\F}_p)$, 
represented by a diagram \eqref{RZ tilde functor}, the corresponding lattices  $M$ and $N$ in \eqref{fake minuscule lattices 2} are dual to one another under the hermitian form on $\breve{\Lambda}[1/p]$, while $L$ and $\breve{\Lambda}$ are each self-dual.  Thus dualizing  the second equality in \eqref{this MN} shows $M_1=L_1\cap \breve{\Lambda}_1$, while dualizing  the first shows $N_1=L_1+\breve{\Lambda}_1$.  We deduce that 
\[
M=L\cap \breve{\Lambda} \quad \mbox{and}\quad N=L+ \breve{\Lambda}.
\]

Recalling how \eqref{fake minuscule lattices 2} was constructed from the diagram of Dieudonn\'e modules above it, we find that 
the inverse to the above function $\pi_0$ sends a point $X \in \mathrm{RZ}_\Lambda^k(\breve{\F}_p)$, corresponding to a diagram \eqref{complete point diagram}, to the diagram \eqref{RZ tilde functor} determined by
\begin{align}\label{H intersection}
D(H) & = D(X) \cap  D( \Lambda \otimes \overline{\mathbb{Y}})   \\
D(H^\vee) & = D(X) +   D( \Lambda \otimes \overline{\mathbb{Y}}), \nonumber
\end{align}
and $D(G) = D(H^\vee)_0\oplus D(H)_1$, as per Lemma \ref{lem:intermediate crystal}.


\subsection{Analysis of $\pi_0$}


Our goal in this subsection is to prove that the  arrow 
\[
\pi_0 :  \mathbf{RZ}_\Lambda^k  \to  \mathrm{RZ}_\Lambda^k
\]
in \eqref{augment diagram} is a closed immersion  inducing an isomorphism of underlying reduced schemes.
This will use the following general result of Grothendieck-Messing theory, in which 
 $\breve{\F}_p[\epsilon]$ denotes the usual ring of infinitesimals defined by $\epsilon^2=0$.

\begin{lemma}\label{lem:easy deformation}
Let  $f' : H_1' \to H_2'$  be a morphism of $p$-divisible groups over  $\breve{\F}_p[\epsilon]$, and denote by   $f:H_1 \to H_2$ its reduction  to $\breve{\F}_p$.
Assume that the induced morphism
\[
 f : \mathscr{D} (H_1) \to \mathscr{D} (H_2)
\]
of $\breve{\F}_p$-vector spaces satisfies
\begin{equation}\label{easy def kernel}
\mathrm{Fil}^0 \mathscr{D} (H_1) = \ker \left( \mathscr{D} (H_1) \map{f} 
 \frac{ \mathscr{D}(H_2) }{   \mathrm{Fil}^0 \mathscr{D}(H_2) }  \right) .
\end{equation}
If   $H_2'$ is isomorphic to the constant deformation of $H_2$, then   $H_1'$ is isomorphic to the constant deformation of $H_1$.
\end{lemma}

\begin{proof}
Let $f^\circ : H_1^\circ \to H_2^\circ$ be the constant deformation of $f:H_1 \to H_2$ to $\breve{\F}_p[\epsilon]$.
Grothendieck-Messing theory provides us with a  canonical commutative diagram
\[
\xymatrix{
{  \mathscr{D}(H_1') }     \ar[r]^{f'}     \ar[d]_\iso   & {    \mathscr{D}(H_2') }  \ar[d]^\iso    \\ 
  { \mathscr{D}(H_1^\circ) }   \ar[r]_{f^\circ}&   {  \mathscr{D}(H_2^\circ)   }      \\
}
\]
of $\breve{\F}_p[\epsilon]$-modules. 
 The vertical isomorphism on the right identifies the Hodge filtrations on source and target, as it is induced by an isomorphism of deformations $H_2^\circ \iso H_2'$, and 
hence there is an induced diagram
 \begin{equation}\label{easy def diagram}
\xymatrix{
{  \mathscr{D}(H_1') }     \ar[rr]^{f'}  \ar[d]_\iso  &  &  {    \mathscr{D}(H_2') / \mathrm{Fil}^0 \mathscr{D}(H'_2) }   \ar[d]^\iso     \\ 
  { \mathscr{D}(H_1^\circ) }   \ar[rr]_{f^\circ}    &  &   {  \mathscr{D}(H_2^\circ) / \mathrm{Fil}^0 \mathscr{D}(H^\circ_2)  } .    \\
}
\end{equation}

The top  horizontal arrow in \eqref{easy def diagram}, being induced by the morphism of $p$-divisible groups $f' : H_1' \to H_2'$,    satisfies
 \[
\mathrm{Fil}^0 \mathscr{D} (H'_1)  
\subset \ker \left(  \mathscr{D} (H'_1) \map{f'}  \frac{ \mathscr{D}(H'_2) } { \mathrm{Fil}^0 \mathscr{D}(H'_2) } \right) .
\]
On the other hand,  as $f^\circ : H_1^\circ \to H_2^\circ$ is the constant deformation, \eqref{easy def kernel} implies
\[
\mathrm{Fil}^0 \mathscr{D} (H^\circ_1)  
 =  \ker \left(  \mathscr{D} (H^\circ_1) \map{f^\circ}  \frac{ \mathscr{D}(H^\circ_2) }{  \mathrm{Fil}^0 \mathscr{D}(H^\circ_2) } \right) .
\]
It follows that the left vertical isomorphism in \eqref{easy def diagram} restricts to a map 
\[
\mathrm{Fil}^0 \mathscr{D} (H'_1) \to \mathrm{Fil}^0 \mathscr{D} (H^\circ_1), 
\]
which is an isomorphism because the Hodge filtrations are local direct summands of the same rank.
By  Grothendieck-Messing theory there is an isomorphism of deformations $H_1' \iso H_1^\circ$. 
\end{proof}

\begin{proposition}\label{prop:pi_0 isomorphism}
The map $\pi_0 :  \mathbf{RZ}_\Lambda^k  \to  \mathrm{RZ}_\Lambda^k$ is a closed immersion  inducing an isomorphism of underlying reduced schemes.
\end{proposition}

\begin{proof}
The key step is to show that the map in question is unramified.  For this, it suffices to show that the induced map on tangent spaces is injective.
Abbreviate 
\[
s = \Spec(\breve{\F}_p)  \quad\mbox{and}\quad  \tilde{s}= \Spec(\breve{\F}_p[\epsilon]).
\]
Given a point of $\mathbf{RZ}^k_\Lambda( \tilde{s} )$, represented by a diagram 
\[
\xymatrix{
& &   { \widetilde{X}  } \ar[dr]_{c^\vee}  \ar@/^/[drr]^{j^\vee} \\
{p\Lambda \otimes \overline{\mathbb{Y}}_{ \tilde{s}}  }\ar[r]^a  \ar@/^/[urr]^j   \ar@/_/[drr]_i  &  { \widetilde{H} }  \ar[ur]_c \ar[dr]^b  \ar[r]|{\, d\, }  &  {\widetilde{G}}  \ar[r] | {\, d^\vee\, } & {  \widetilde{H}^\vee }  \ar[r]^{a^\vee}  & {p^{-1}\Lambda \otimes  \overline{\mathbb{Y}}_{ \tilde{s}}  } \\
& &  { \Lambda \otimes  \overline{\mathbb{Y} }_{ \tilde{s}} } \ar[ur]^{b^\vee} \ar@/_/[urr]_{i^\vee}
}
\]
of $p$-divisible groups over $ \tilde{s}$,  and deforming  a diagram 
\[
\xymatrix{
& &   {  X  } \ar[dr]_{c^\vee}  \ar@/^/[drr]^{j^\vee} \\
{p\Lambda \otimes \overline{\mathbb{Y}}_s  }\ar[r]^a  \ar@/^/[urr]^j   \ar@/_/[drr]_i  &  { H }  \ar[ur]_c \ar[dr]^b  \ar[r]|{\, d\, }  &  {G}  \ar[r] | {\, d^\vee\, } & {  H^{ \vee} }  \ar[r]^{a^\vee}  & {p^{-1}\Lambda \otimes  \overline{\mathbb{Y}}_s  } \\
& &  { \Lambda \otimes  \overline{\mathbb{Y} }_s } \ar[ur]^{b^\vee} \ar@/_/[urr]_{i^\vee}
}
\]
of $p$-divisible groups over $S$.  We must show that if $\widetilde{X}$ is the constant deformation,  then so are $\widetilde{H}$ and $\widetilde{G}$.

The constancy of  $\widetilde{H}$ follows by applying Lemma \ref{lem:easy deformation}  to the morphism
\[
H \map{c\times b}  X  \times (\Lambda \otimes \overline{\mathbb{Y}}_s). 
\]
The only thing to check is that the hypothesis 
\[
\mathrm{Fil}^0 \mathscr{D} (H)
=\ker \left( 
 \mathscr{D} (H)   \to  \frac{ \mathscr{D}(X)  } {  \mathrm{Fil}^0 \mathscr{D}(X)   }
    \times \frac{    \mathscr{D}( \Lambda \otimes \overline{\mathbb{Y}} )  }{  \mathrm{Fil}^0     \mathscr{D}( \Lambda \otimes \overline{\mathbb{Y}} ) }  \right) 
\]
of that lemma is satisfied.  This hypothesis  is equivalent  to 
\[
\frac{VD(H)}{pD(H)} = 
\ker\left(  
\frac{D(H)}{pD(H)} \to
 \frac{D(X)}{VD(X)} \times \frac{D(\Lambda \otimes  \overline{\mathbb{Y}})  }{ V D(\Lambda \otimes  \overline{\mathbb{Y}}) } 
\right) ,
\]
which is clear from \eqref{H intersection}.

It remains to show that $\widetilde{G}$ is the constant deformation of $G$.
Denote by $H^\circ$ and $G^\circ$ be the constant deformations of $H$ and $G$ to $\tilde{s}$.
 Grothendieck-Messing theory provides us with  a commutative diagram 
\[
\xymatrix{
{  \mathscr{D}(H^\circ) }  \ar@{=}[r]  \ar[d]_{d}  &   {  \mathscr{D}(\widetilde{H}) } \ar[d]^{d} \\
{  \mathscr{D}(G^\circ) }    \ar@{=}[r]  &   {  \mathscr{D}(\widetilde{G}) }  
}
\]
of vector bundles on $\widetilde{s}$.
As we have already know $H^\circ \iso \widetilde{H}$,  the top isomorphism    respects  the Hodge filtrations.
 The vertical arrows respect Hodge filtrations, because they arise from morphisms of $p$-divisible groups.     
 over $\tilde{s}$.  We must show that the bottom isomorphism also respects Hodge filtrations.

 The essential point is that the vertical arrows  restrict to  surjections
\[
\mathrm{Fil}^0  \mathscr{D}(H^\circ)_0   \to  \mathrm{Fil}^0 \mathscr{D}(G^\circ)_0 \quad \mbox{and}\quad 
\mathrm{Fil}^0  \mathscr{D}(\widetilde{H})_0  \to  \mathrm{Fil}^0 \mathscr{D}( \widetilde{G})_0.
\]
Indeed, surjectivity can be checked on fibers.  On fibers both are equivalent to the surjectivity of
\[
\frac{ VD(H)_1 }{ pD(H)_0}  \to  \frac{ VD(G)_1 }  {  pD(G)_0 }, 
\]
which is clear from Lemma  \ref{lem:intermediate crystal}.

It follows that under the canonical identification  $\mathscr{D}(G^\circ) = \mathscr{D}(\widetilde{G})$ 
\begin{align*}
\mathrm{Fil}^0 \mathscr{D}(G^\circ)_0
 & = \mathrm{Image} ( \mathrm{Fil}^0  \mathscr{D}(H^\circ)_0 \to   \mathscr{D}(G^\circ)_0 )   \\
& = 
 \mathrm{Image} ( \mathrm{Fil}^0  \mathscr{D}(\widetilde{H})_0 \to   \mathscr{D}(\widetilde{G})_0 )  \\
 & = 
 \mathrm{Fil}^0 \mathscr{D}( \widetilde{G})_0.
\end{align*}
A similar argument, using   $d^\vee: G\to H^\vee$ in place of $d:H\to G$, shows that
$\mathrm{Fil}^0 \mathscr{D}(G^\circ)_1 = \mathrm{Fil}^0 \mathscr{D}(\widetilde{G})_1$.
Hence $G^\circ \iso \widetilde{G}$ as deformations of $G$.
 \end{proof}


\subsection{A Deligne-Lusztig variety}
\label{ss:DL}


In this subsection we identify the  reduced scheme $Y^{k,\red}_\Lambda$ underlying $Y^k_\Lambda$ with a Deligne-Lusztig variety.

Suppose $S$ is an $\breve{\F}_p$-scheme,  and let $\sigma :\co_S \to \co_S$ be the $p$-power Frobenius.
The pairing \eqref{twisted pairing}  induces a pairing of $\co_S$-modules
\[
 ( \breve{\Lambda}_0 \otimes_{ \breve{\Z}_p} \co_S ) \times ( \breve{\Lambda}_0 \otimes_{ \breve{\Z}_p}  \co_S )
  \to \co_S
\]
that is linear in the first variable and $\sigma$-linear in the second.
For any local direct summand $\mathscr{F} \subset  \breve{\Lambda}_0 \otimes_{ \breve{\Z}_p}  \co_S$,  its left annihilator 
\[
\mathscr{F}^\perp = \{ x \in  \breve{\Lambda}_0 \otimes_{ \breve{\Z}_p}   \co_S :  b( x, \mathscr{F} ) =0 \}
\]
is again a  local direct summand.  Taking into account the switch from right dual to left annihilator, 
the analogue of \eqref{twisted dual properties} is 
\[
\mathscr{F}^{\perp \perp}= \Phi^2 \mathscr{F} 
\quad \mbox{and}\quad
 \Phi^2(\mathscr{F}^\perp) = (\Phi^2 \mathscr{F})^\perp.
\]

\begin{definition}\label{def:non-minuscule DL}
Define $\mathrm{DL}_\Lambda^k$ to be the proper $\breve{\F}_p$-scheme whose functor of points assigns to any $\breve{\F}_p$-scheme $S$ the set of flags of $\co_S$-module local direct summands
\[
0 \stackrel{k-1}{\subset}  \mathscr{J} \stackrel{1}{\subset} \mathscr{K}^\perp 
\stackrel{n-2k}{\subset}      \mathscr{K} \stackrel{1}{\subset} \mathscr{J}^\perp \stackrel{k-1}{\subset}
 \breve{\Lambda}_0 \otimes \co_S
\]
of the indicated coranks.   
\end{definition}

\begin{proposition}\label{prop:DLstructure}
The scheme $\mathrm{DL}_\Lambda^k$ of Definition \ref{def:non-minuscule DL} is a  Deligne-Lusztig variety for the unitary group of the finite hermitian space $\Lambda/p\Lambda$.
 It is  smooth of dimension $n-k-1$, and is  irreducible if $k<n/2$. 
\end{proposition}

\begin{proof}
Let $H$ be the base change to $\breve{\F}_p$ of the unitary group of $\Lambda/p\Lambda$.
Fix an $\co_E$-basis $x_1,\ldots, x_n \in \Lambda$ with respect to which the hermitian form satisfies the anti-diagonal relation $h(x_i,x_j) = \delta_{i, n-i+1}$, and denote by
$
y_1,\ldots, y_n \in \breve{\Lambda}_0
$
the projections of $x_1,\ldots, x_n$ to the first factor in the decomposition $\breve{\Lambda}=\breve{\Lambda}_0\oplus \breve{\Lambda}_1$.
Use the reductions  $y_1,\ldots, y_n \in \breve{\Lambda}_0/p\breve{\Lambda}_0$  to identify
\[
H \iso \GL(\breve{\Lambda}_0/p\breve{\Lambda}_0) \iso \GL_n.
\]

The Weyl group $W$ of the diagonal torus is generated by the set of simple reflections $S=\{ s_1,\ldots, s_{n-1} \}$,  where $s_i \in  H(\breve{\F}_p)$ is the transposition matrix interchanging $y_i \leftrightarrow y_{i+1}$.  The action of Frobenius on $S$ is 
$
\sigma(s_i) = s_{n-i}.
$
Every subset $I\subset S$ determines a standard parabolic subgroup 
\[
P_I = B W_I B \subset H,
\]
where $W_I \subset W$ is the subgroup generated by $I$, and $B\subset H$ is the upper-triangular Borel.

We are interested in the subset
$
I  = S \smallsetminus \{ s_{k-1},  s_{n-k} \}
$
whose corresponding parabolic $P_I$ is the stabilizer of the standard flag 
\[
0 \stackrel{k-1}{\subset}  \mathscr{J} 
\stackrel{n-2k+1}{\subset}      \mathscr{K} \stackrel{k}{\subset}
 \breve{\Lambda}_0 / p  \breve{\Lambda}_0
\]
defined by
\[
\mathscr{J} = \mathrm{Span}\{ y_1,\ldots, y_{k-1} \} 
\quad \mbox{and} \quad
\mathscr{K} = \mathrm{Span}\{ y_1,\ldots, y_{n-k} \} .
\]
The parabolic $P_{\sigma(I)}$ defined by $\sigma(I) = S \smallsetminus \{ s_k , s_{n-k+1} \}$ is the stabilizer of 
\[
0 \stackrel{k}{\subset}   \mathscr{K}^\perp 
\stackrel{n-2k+1}{\subset}     \mathscr{J}^\perp \stackrel{k-1}{\subset}
 \breve{\Lambda}_0 / p  \breve{\Lambda}_0.
\]

Using Lemma 2.12 of \cite{Vollaard}, one sees that   $\mathrm{DL}_\Lambda^k$ sits in a cartesian diagram
\[
\xymatrix{
{  \mathrm{DL}_\Lambda^k }   \ar[r] \ar[d]  & {   H / P_I    }  \ar[d]^{\mathrm{id} \times \sigma}   \\
{   H / P_{ I \cap \sigma(I) }    }  \ar[r]  &  {     H / P_I   \times   H / P_{  \sigma(I) } ,} 
}
\]
and hence, using the notation of \S 4.4 of \cite{VollaardWedhorn},  
\[
\mathrm{DL}_\Lambda^k \iso X_I( \mathrm{id} ).
\]
All parts of the proposition now follow from the discussion of \emph{loc.~cit.}.  
Note that the irreducibility claim (a result of Bonnafe and Rouquier \cite{BR}) requires $k<n/2$ because this ensures   $I \cup   \sigma  (I) = S$.
\end{proof}

The remainder of this subsection is devoted to proving $Y_\Lambda^{k,\red} \iso \mathrm{DL}_\Lambda^k$.  We need two elementary lemmas from commutative algebra.

\begin{lemma}\label{lem:general iso}
Let $\kappa$ be an algebraically closed field, let  $\pi :X' \to X$ be a proper unramified morphism between $\kappa$-schemes of finite type, and suppose $\pi$  is bijective on closed points.
\begin{enumerate}
\item
The morphism $\pi$ is a closed immersion, and induces an isomorphism of underlying reduced schemes.
\item
If $X$ is reduced then so is $X'$, and $\pi$ is an isomorphism.
\end{enumerate}
\end{lemma}

\begin{proof}
A proper and quasi-finite morphism is finite, hence affine.
This reduces the proof to the case where $X=\Spec(A)$ and $X'=\Spec(B)$ with $A\to B$ a finite morphism of finite type $\kappa$-algebras.

The unramifiedness assumption implies that  $\mathfrak{m}  B \subset B$ is maximal for any maximal ideal $\mathfrak{m}\subset A$.  
In particular 
\[
\kappa \iso A/\mathfrak{m} \to B/  \mathfrak{m}  B\iso \kappa
\]
 is an isomorphism of $A$-modules, and  Nakayama's lemma implies that $ A \to B$  is surjective.  
 In other words,   $\pi$ is a closed immersion.

 The bijectivity of $A \to B$ on maximal ideals implies  that $I=\ker(A\to B)$ is contained in every maximal ideal of $A$. 
 As $A$ is a Jacobson ring, the intersection of its maximal ideals is equal to the  nilradical $\mathfrak{n} \subset A$.
 Thus $I \subset \mathfrak{n}$, which implies that  $\mathfrak{n} / I$ is the nilradical of $A/I =B$.
All claims of the lemma follow immediately.
\end{proof}

\begin{lemma}\label{lem:bundle criterion}
Let $\kappa$ be an algebraically closed field, suppose $X$ is a reduced scheme of finite type over $\kappa$, and $\mathcal{E}$ is a coherent $\co_X$-module.
If the $\kappa$-dimension of the fiber $\mathcal{E}_x$  is constant as $x\in X$ varies over all closed points, then $\mathcal{E}$ is locally free.
\end{lemma}

\begin{proof}
Let $r$ be the common dimension of all fibers.  At any closed point $x \in X$, one can use Nakayama's lemma to find an open  neighborhood $U\ni x$ over which there exists  a surjection $\co_U^r \to \mathscr{E}_U$.
By the constancy of fiber dimensions, such a map must induce an isomorphism fiber-by-fiber, and hence have kernel contained in the subsheaf $\mathcal{J} \co_U^r$, where $\mathcal{J} \subset \co_U$ is the Jacobson radical.
As $X$ is reduced of finite type over a field, $\mathcal{J}=0$, proving $\co_U^r \iso \mathscr{E}_U$.
\end{proof}

\begin{theorem}\label{thm:YtoDL}
There is an  isomorphism 
\[
Y^{k,\red}_\Lambda \iso \mathrm{DL}_\Lambda^k
\]
sending a pair $(M_0,N_0) \in Y^{k,\red}_\Lambda( \breve{\F}_p)$ as in Proposition \ref{prop:RZ tilde points} to the flag 
\[
0 \subset   \frac{ p M_0^*}{ p\breve{\Lambda}_0 }     \subset   \frac{pN_0}{ p\breve{\Lambda}_0} 
  \subset   \frac{ N_0^* }{ p\breve{\Lambda}_0   }   \subset  \frac{ M_0 } { p\breve{\Lambda}_0 }  \subset  \frac{  \breve{\Lambda}_0} { p \breve{\Lambda}_0 }.
\]
\end{theorem}

\begin{proof}
As in  \eqref{GMbundles}, Grothendieck-Messing theory provides us with a diagram of vector bundles
\[
\xymatrix{
{p  \breve{\Lambda}  \otimes \mathscr{D}(\overline{\mathbb{Y}})   }\ar[r]^a    \ar@/_/[drr]_{i =0 }  &  {\mathscr{D}(H)} \ar[dr]^b  \ar[r]|{\, d\, }   &  
{\mathscr{D}(G) } \ar[r]|{\, d^\vee \, }  & {\mathscr{ D }(H^\vee)  }  \ar[r]^{a^\vee}  & {p^{-1}   \breve{\Lambda}    \otimes \mathscr{D}(\overline{\mathbb{Y}})     } \\
& &  {   \breve{\Lambda}   \otimes \mathscr{D}(\overline{\mathbb{Y}})    } \ar[ur]^{b^\vee} \ar@/_/[urr]_{i^\vee=0}
}
\]
on $Y^{k,\red}_\Lambda$.
Applying the functor $\Hom_{\co_{\breve{E}}}( \mathscr{D}(\overline{\mathbb{Y}})  , - )$ yields a diagram
\[
\xymatrix{
{p  \breve{\Lambda}  \otimes \co_{Y^\red_\Lambda}    }\ar[r]^a    \ar@/_/[drr]_{i =0 }  &  {\mathscr{M}} \ar[dr]^b  \ar[r]  &  
{\mathscr{N}_0 \oplus \mathscr{M}_1 } \ar[r]   & {\mathscr{ N } }  \ar[r]^{a^\vee}  & {p^{-1}   \breve{\Lambda} \otimes   \co_{Y^\red_\Lambda}    } \\
& &  {   \breve{\Lambda}   \otimes \co_{Y^\red_\Lambda}    } \ar[ur]^{b^\vee} \ar@/_/[urr]_{i^\vee=0}
}
\]
 of vector bundles on $Y^{k,\red}_\Lambda$, in which 
\[
\mathscr{M} =  \underline{\Hom}_{\co_E} ( \mathscr{D}(  \overline{\mathbb{Y}} ) ,  \mathscr{D}(H) )
\quad\mbox{and}\quad
\mathscr{N} =   \underline{\Hom}_{\co_E} ( \mathscr{D}(  \overline{\mathbb{Y}} ) ,  \mathscr{D}(H^\vee) ) 
\]
are defined exactly as in \eqref{the L bundle}.
Note that in the first diagram the tensor products are over $\co_{\breve{E}}$, while in the second they are over $\breve{\Z}_p$.  
 Note also that we are making use of the fact  that the natural maps
\[
\mathcal{D}(H)_1 \to  \mathcal{D}(G)_1 \quad\mbox{and}\quad \mathcal{D}(G)_0 \to \mathcal{D}(H^\vee)_0
\]
are isomorphisms, so that 
$
\mathscr{N}_0 \oplus \mathscr{M}_0 \iso 
\underline{\Hom}_{\co_E} ( \mathscr{D}(  \overline{\mathbb{Y}} ) ,  \mathscr{D}(G) ) .
$
  Indeed,  this can be checked  on fibers,  where it follows from Lemma \ref{lem:intermediate crystal}.

The vector bundles above come with filtrations 
\begin{align*}
0\subset 
 \mathrm{Fil}^1 \mathscr{M} \subset 
 \mathrm{Fil}^0 \mathscr{M} \subset 
\mathscr{M} \quad\mbox{and}\quad 
0\subset 
 \mathrm{Fil}^1 \mathscr{N} \subset 
 \mathrm{Fil}^0 \mathscr{N} \subset 
\mathscr{N}  
\end{align*}
induced by the Hodge filtrations \eqref{GMbundles} on 
$\mathscr{D}(  \overline{\mathbb{Y}} )$,   $\mathscr{D}(H)$, and  $\mathscr{D}(H^\vee)$.
More precisely,  $\Fil^0 \mathscr{M}$ consists of morphisms that carry $\Fil^0 \mathscr{D}(  \overline{\mathbb{Y}} )$ into $\Fil^0 \mathscr{D}(H)$, while $\Fil^1\mathscr{M}$ consists of morphisms that  carry 
$ \mathscr{D}(  \overline{\mathbb{Y}} )$ into $\Fil^0 \mathscr{D}(H)$ and kill $\Fil^0 \mathscr{D}(  \overline{\mathbb{Y}} )$.
The filtration on $\mathscr{N}$ is defined in the same way.

As in Remark \ref{rem:untwist L}, a choice of $D(\overline{\mathbb{Y}}) \iso \co_{\breve{E}}$
determines isomorphisms
\begin{equation}\label{MN untwist}
\mathscr{M} \iso \mathscr{D}(H) \quad\mbox{and}\quad \mathscr{N} \iso \mathscr{D}(H^\vee).
\end{equation}
While these do not respect the filtrations, they do identify
\begin{equation}\label{MN filters}
\Fil^0 \mathscr{M}_0 \iso \Fil^0 \mathscr{D}(H)_0 
\quad \mbox{and}\quad  \Fil^0 \mathscr{N}_0 \iso \Fil^0 \mathscr{D}(H^\vee)_0 .
\end{equation}
This follows from the signature $(0,1)$ condition on $\overline{\mathbb{Y}}$, which is equivalent to
$\Fil^0 \mathscr{D}(  \overline{\mathbb{Y}} )_0 = \mathscr{D}(  \overline{\mathbb{Y}} )_0$.

\begin{lemma}\label{lem:MN filter fibers}
Fix a point $s\in Y^{k,\red}_\Lambda( \breve{\F}_p)$ corresponding to a lattice chain 
\[
 p\breve{\Lambda}_0 \stackrel{k-1}{\subset}  p M_0^* \stackrel{1}{\subset}  pN_0 \stackrel{n-2k}{\subset}  N_0^* \stackrel{1}{\subset}  M_0 \stackrel{k-1}{\subset} \breve{\Lambda}_0  
 \]
 as in Proposition \ref{prop:RZ tilde points}.  There are canonical identifications
 \[
\frac{N^*_0}{pM_0}  =  \Fil^0 \mathscr{M}_{0,s}   
 \subset  \mathscr{M}_{0,s} = \frac{M_0}{pM_0}
 \]
and
 \[
\frac{M^*_0}{pM_0}  =  \Fil^0 \mathscr{N}_{0,s}   
 \subset  \mathscr{N}_{0,s} = \frac{N_0}{pN_0}.
 \]
\end{lemma}

\begin{proof}
Fix an isomorphism $D(\overline{\mathbb{Y}}) \iso \co_{\breve{E}}$, so that 
\[
\mathscr{M}_{0,s} \stackrel{ \eqref{MN untwist}}{\iso} \mathscr{D}(H_s)_0 
= \frac{D(H_s)_0}{ p D(H_s)_0 } \stackrel{\eqref{MNlattices}}{\iso} \frac{M_0}{pM_0} .
\]
The first and third isomorphisms depend on the choice of $D(\overline{\mathbb{Y}}) \iso \co_{\breve{E}}$, but the composition does not.
Similarly, 
\[
\Fil^0 \mathscr{M}_{0,s} \stackrel{ \eqref{MN filters}}{\iso} \Fil^0 \mathscr{D}(H)_{0,s} 
= \frac{V D(H_s)_1}{pD(H_s)_0} 
\stackrel{  \ref{contraction diagrams}  }{\iso} 
  \frac{\Phi^{-1} M_1}{pM_0}
\stackrel{\eqref{fake dual shift}}{=}   \frac{ N_0^*}{pM_0}.
\]
 Once again, the first and third isomorphisms depend on the choice of  $D(\overline{\mathbb{Y}}) \iso \co_{\breve{E}}$, but the composition does not.
The claims about $\mathscr{N}_{0,s}$ are proved in the same way, using $H_s^\vee$ in place of $H_s$.
\end{proof}

Now consider the subsheaves  of $\breve{\Lambda}_0 \otimes \co_{Y^{k,\red}_\Lambda}$ defined by
 \begin{align*}
   \mathscr{J}  & = \mathrm{Image}(         \mathrm{Fil}^0 \mathscr{N}_0    \to  p^{-1}\breve{\Lambda}_0 \otimes \co_{Y^{k,\red}_\Lambda}  \iso \breve{\Lambda}_0 \otimes \co_{Y^{k,\red}_\Lambda}     ) \\
   \mathscr{K} & = \mathrm{Image}(   \mathrm{Fil}^0 \mathscr{M}_0    \to  \breve{\Lambda}_0 \otimes \co_{Y^{k,\red}_\Lambda}      ) .
 \end{align*}
 Lemma \ref{lem:MN filter fibers}  implies that the quotients by these subsheaves have constant fiber dimension, hence by Lemma \ref{lem:bundle criterion} the quotients are locally free, 
 and hence $ \mathscr{J}$ and $\mathscr{K}$ are local direct summands.
Again by checking on fibers, using  Lemma \ref{lem:MN filter fibers}, one sees that these subsheaves have left annihilators
 \begin{align*}
  \mathscr{J}^\perp & = \mathrm{Image}(   \mathscr{M}_0    \to  \breve{\Lambda}_0 \otimes \co_{Y^{k,\red}_\Lambda}      ) \\  
    \mathscr{K}^\perp & = \mathrm{Image}(   \mathscr{N}_0    \to  p^{-1}\breve{\Lambda}_0 \otimes \co_{Y^{k,\red}_\Lambda}   \iso \breve{\Lambda}_0 \otimes \co_{Y^{k,\red}_\Lambda}    ) .
\end{align*}
Yet again checking on fibers,  one finds using Lemma \ref{lem:MN filter fibers} that these satisfy 
\begin{equation}\label{Ylambda flag}
0 \stackrel{k-1}{\subset}  \mathscr{J} \stackrel{1}{\subset} \mathscr{K}^\perp 
\stackrel{n-2k}{\subset}      \mathscr{K} \stackrel{1}{\subset} \mathscr{J}^\perp \stackrel{k-1}{\subset}
 \breve{\Lambda}_0 \otimes \co_{Y^{k,\red}_\Lambda}     .
\end{equation}

The flag  \eqref{Ylambda flag} defines the desired morphism 
\begin{equation}\label{YtoDLconstruction}
Y^{k,\red}_\Lambda \to \mathrm{DL}_\Lambda^k,
\end{equation}
 and it remains to show that it is an isomorphism.  
 The key to this is the following lemma, which tells us how to recover the filtration \eqref{MN filters} from the flag \eqref{Ylambda flag}, using the quotient maps
 $\mathscr{M}_0 \to \mathscr{J}^\perp$ and  $\mathscr{N}_0 \to \mathscr{K}^\perp$.

\begin{lemma}\label{lem:miracle filtration}
The above vector bundles on $Y^{k,\red}_\Lambda$ satisfy
\[
 \mathrm{Fil}^0 \mathscr{M}_0  = \ker(  \mathscr{M}_0 \to \mathscr{J}^\perp / \mathscr{K} ) 
 \quad \mbox{and}\quad 
  \mathrm{Fil}^0 \mathscr{N}_0  = \ker(  \mathscr{N}_0 \to \mathscr{K}^\perp / \mathscr{J} ) .
\]
\end{lemma}

\begin{proof}
For the first equality, one uses Lemma \ref{lem:MN filter fibers} to check that the cokernel of 
$ \mathscr{M}_0 \to \mathscr{J}^\perp / \mathscr{K}$ has constant fiber dimension, and so is locally free by Lemma \ref{lem:bundle criterion}.  This implies that the kernel of this morphism is a local direct summand of $\mathscr{M}_0$.  The desired equality can therefore be checked on fibers, which is again done using Lemma \ref{lem:MN filter fibers}.  Indeed, at a point $(M_0,N_0) \in  Y^{k,\red}_\Lambda$  the desired equality is precisely
\[
\frac{N^*_0}{pM_0}  = \ker\left(   \frac{M_0}{pM_0}  \to  \frac{ M_0 / p \breve{\Lambda}_0 }{ N_0^*  / p \breve{\Lambda}_0 }  \right) ,
\]
which is clear.
The second equality is proved in exactly the same way. 
\end{proof}

\begin{lemma}\label{lem:YtoDLunr}
The morphism \eqref{YtoDLconstruction} is  unramified.
\end{lemma}

\begin{proof}
We show that the morphism in question is formally unramified.
Suppose we are given a diagram 
\[
\xymatrix{
{ S }  \ar[d] \ar[r]  & { Y^{k,\red}_\Lambda   }  \ar[d] \\
{ \widetilde{S} } \ar[r]  \ar@{-->}[ur] & {  \mathrm{DL}_\Lambda^k }  
}
\]
of solid arrows such that the left vertical arrow is a square-zero thickening of $\breve{\F}_p$-schemes.
We must show there is at most one dotted arrow making the diagram commute.

Suppose we have two such arrows $a,b : \widetilde{S} \to Y^{k,\red}_\Lambda$, and consider the $p$-divisible groups $a^*H$ and $b^*H$ on $\widetilde{S}$.  
We do not yet know that these are isomorphic, but they have the same restriction to $S \subset \widetilde{S}$, and so  Grothendieck-Messing deformation theory provides us with  canonical isomorphisms of vector bundles 
\[
a^*\mathscr{D}(H)\iso 
\mathscr{D}(a^*H) \iso \mathscr{D}(b^*H)  \iso b^*\mathscr{D}(H)
\]
 making the diagram 
\[
\xymatrix{
{ a^*  \mathscr{D}( H) }  \ar[dr]   \ar@{=}[rr]& &{ b^*  \mathscr{D}(H)  }  \ar[dl]  \\
&  {  \mathscr{D}( \breve{\Lambda}_0 \otimes \overline{\mathbb{Y}}_{\widetilde{S}} )}
}
\]
commute.  
Here the diagonal arrows are the pullbacks (via $a$ and $b$) of the morphism $\mathscr{D}(H) \to \mathscr{D}( \breve{\Lambda}_0 \otimes \overline{\mathbb{Y}})$ induced by the universal isogeny $H \to  \breve{\Lambda}_0 \otimes \overline{\mathbb{Y}}$ over $Y^{k,\red}_\Lambda$.

Fixing a choice of $D(\overline{\mathbb{Y}}) \iso \co_{\breve{E}}$, and hence isomorphisms \eqref{MN untwist} and \eqref{MN filters} of vector bundles on $Y^{k,\red}_\Lambda$, we find a diagram 
\[
\xymatrix{
{  a^*\mathscr{M}_0 }  \ar[dr]   \ar@{=}[rr]& &{ b^*\mathscr{M}_0  }  \ar[dl]  \\
&  { \breve{\Lambda}_0 \otimes \co_{\widetilde{S}} .}
}
\]
As $a$ and $b$ induce the same map $\widetilde{S} \to \mathrm{DL}_\Lambda^k$, we have equalities 
$a^*\mathscr{J}^\perp=b^*\mathscr{J}^\perp$ and $a^*\mathscr{K}=b^*\mathscr{K}$ as local direct summands of  $\breve{\Lambda}_0 \otimes \co_{\widetilde{S}}$, and hence a commutative diagram
\[
\xymatrix{
{ a^* \mathscr{D}(H)_0 } \ar@{=}[r] &{  a^*\mathscr{M}_0 }  \ar[d]   \ar@{=}[rr]& &{ b^*\mathscr{M}_0  }  \ar[d]  \ar@{=}[r] 
&  {  b^* \mathscr{D}(H)_0  }  \\
& {  a^*(\mathscr{J}^\perp / \mathscr{K} ) }    \ar@{=}[rr]& &{ b^*  ( \mathscr{J}^\perp / \mathscr{K} ) . }  
}
\]

Combining this diagram with Lemma \ref{lem:miracle filtration} and \eqref{MN filters}, and applying the same reasoning to $H^\vee$, we find that the canonical isomorphisms 
\[
\mathscr{D}(a^*H)_0 \iso \mathscr{D}(b^*H)_0  \quad\mbox{and}\quad \mathscr{D}(a^*H^\vee)_0 \iso \mathscr{D}(b^*H^\vee)_0 
\]
of Grothendieck-Messing theory respect Hodge filtrations.  By the duality between $\mathscr{D}(H)_1$ and $\mathscr{D}(H^\vee)_0$, under which the Hodge filtrations are annihilators of one another, the second of these implies that the isomorphism
\[
\mathscr{D}(a^*H)_1 \iso \mathscr{D}(b^*H)_1 
\]
also respects Hodge filtrations, and it follows that $a^*H \iso b^*H$.

It only remains to show that $a^*G \iso b^*G$, but this follows from $a^*H \iso b^*H$, exactly as in  the proof of Proposition \ref{prop:pi_0 isomorphism}.
\end{proof}

At last we complete the proof of Theorem \ref{thm:YtoDL}.
The morphism \eqref{YtoDLconstruction} is bijective on closed points, by comparing the definition of $\mathrm{DL}_\Lambda^k$ with Proposition \ref{prop:RZ tilde points}.
It is unramified by Lemma \ref{lem:YtoDLunr}.  
It is proper, as  the source is a projective  $\breve{\F}_p$-scheme by the same reasoning as in Remark \ref{rem:projectivity}.
We conclude now from  Lemma \ref{lem:general iso} that \eqref{YtoDLconstruction} is an isomorphism.
\end{proof}


\section{Fiber varieties  over Deligne-Lusztig varieties}
\label{s:fibration}


We continue to work with  a fixed  self-dual hermitian $\co_E$-lattice $\Lambda$ of rank $n\ge 2$, and a fixed $1\le k\le \lfloor n/2 \rfloor$, as in  \S \ref{s:enhanced}.  Using the morphism  
\[
\pi_1: \mathbf{RZ}^{\le k}_\Lambda \to   Y^k_\Lambda  
\]
from \eqref{augment diagram}, we will realize (the reduced scheme underlying) $\mathbf{RZ}^{\le k}_\Lambda$ as a closed subscheme of the relative Grassmannian parametrizing rank $k-1$ local direct summands of a  rank $2k-1$ vector bundle $\mathscr{V}$ on $Y^k_\Lambda$.


\subsection{A special vector bundle}
\label{ss:special bundle}


Recall the scheme $Y_\Lambda^k$ of \S \ref{ss:new moduli}, whose underlying reduced scheme
$Y_\Lambda^{k,\red} $ we have identified with a Deligne-Lusztig variety in Theorem \ref{thm:YtoDL}.

Our first goal is to construct a  filtered vector bundle 
\begin{equation}\label{Vfilter}
0 \subset \mathscr{V}^{(1)} \subset \mathscr{V}^{(k)} \subset \mathscr{V}
\end{equation}
 on $Y_\Lambda^{k,\red}$ from   the filtered vector bundles 
\begin{align*}
\mathscr{M} &=  \underline{\Hom}_{\co_E} ( \mathscr{D}(  \overline{\mathbb{Y}} ) ,  \mathscr{D}(H) ) \\\mathscr{N} & =   \underline{\Hom}_{\co_E} ( \mathscr{D}(  \overline{\mathbb{Y}} ) ,  \mathscr{D}(H^\vee) )
\end{align*}
used in the proof of Theorem \ref{thm:YtoDL}.
Here $H \to H^\vee$ is the isogeny of  $p$-divisible groups appearing in the universal diagram \eqref{Y tilde functor} over $Y^{k,\red}_\Lambda$.

\begin{remark}
Recall  from \eqref{MN untwist} that  a choice of $D(\overline{\mathbb{Y}}) \iso \co_{\breve{E}}$ determines isomorphisms
\[
\mathscr{M}_0 \iso \mathscr{D}(H)_0 \quad \mbox{and}\quad \mathscr{N}_0 \iso \mathscr{D}(H^\vee)_0,
\]
each of which identifies the subsheaves $\Fil^0$ on source and target.
\end{remark}

\begin{remark}\label{rem:filter is dual}
At a point $s \in Y_\Lambda^{k,\red}(\breve{\F}_p)$, corresponding to a pair of lattices $(M_0,N_0)$ under Proposition \ref{prop:RZ tilde points}, there are canonical identifications
\begin{align*}
\Fil^0 \mathscr{M}_{0,s} = \frac{N^*_0}{pM_0} &  \subset  \frac{M_0}{pM_0} = \mathscr{M}_{0,s}\\
\Fil^0 \mathscr{N}_{0,s} = \frac{M^*_0}{pN_0}   & \subset  \frac{N_0}{pN_0} = \mathscr{N}_{0,s}.
\end{align*}
This is just a reminder of Lemma \ref{lem:MN filter fibers}.
\end{remark}

Define coherent sheaves on $Y_\Lambda^{k,\red}$  by 
\begin{align}
\mathscr{V}  &= \mathrm{coker}( \Fil^0 \mathscr{M}_0 \to \Fil^0 \mathscr{N}_0 )  \label{VWdef}  \\
\mathscr{W}  & = \mathrm{coker}(  \mathscr{M}_0 \to  \mathscr{N}_0 ) \nonumber  \\
\mathscr{V}^{(1)} & = \ker(  \mathscr{V} \to \mathscr{W}   ) .  \nonumber
\end{align}
The  arrow $b^\vee$ in the universal diagram \eqref{Y tilde functor} determines  a tautological inclusion 
$
\Lambda \subset \Hom_{\co_E}( \overline{\mathbb{Y}} ,  H^\vee).
$
This induces a morphism of vector bundles
\[
\Lambda \otimes_{\Z_p} \co_{Y^{k,\red}_\Lambda}
\to \underline{\Hom}_{\co_E}(  \mathcal{D}(  \overline{\mathbb{Y}} ) , \mathcal{D} (H^\vee) )  ,
\]
which in turn restricts to a morphism
\[
\breve{\Lambda}_0 \otimes \co_{Y^{k,\red}_\Lambda}  \to \Fil^0 \mathscr{N}_0.
\]
This allows us to define 
\begin{equation}\label{Vk}
 \mathscr{V}^{(k)} = \mathrm{Im}  ( \breve{\Lambda}_0 \otimes \co_{Y^{k,\red}_\Lambda} \to \Fil^0 \mathscr{N} \to \mathscr{V} ).
\end{equation}

\begin{proposition}\label{prop:VWfibers}
At a point $s\in Y^{k,\red}_\Lambda(\breve{\F}_p)$, represented by  lattices $(M_0,N_0)$ as in Proposition \ref{prop:RZ tilde points},  there are canonical identifications
 \[
\mathscr{V}_s     = M_0^*/N_0^*
\quad \mbox{and} \quad 
\mathscr{W}_s    = N_0 / M_0 .
 \]
 Moreover, the fiber at $s$ of \eqref{Vfilter} is identified with
\[
0 \subset  \frac{M_0}{N_0^*} \subset   \frac{\breve{\Lambda}_0}{N_0^*} \subset  \frac{M_0^*}{N_0^*}  .
\]
 \end{proposition}
 
 \begin{proof}
 This is clear from Remark  \ref{rem:filter is dual}.
 \end{proof}
 
 \begin{corollary}
 The coherent sheaves on $Y_\Lambda^{k,\red}$ defined above satisfying the following properties.
 \begin{enumerate}
 \item
 $\mathscr{V}$ and $\mathscr{W}$ are locally free of rank $2k-1$,
 \item
 $\mathscr{V}^{(1)}$ and $\mathscr{V}^{(k)}$ are local direct summands of $\mathscr{V}$ of ranks $1$ and $k$, respectively.
 \end{enumerate}
 \end{corollary}
 
 \begin{proof}
 Proposition \ref{prop:VWfibers} implies that $\mathscr{V}$ and $\mathscr{W}$ have constant fiber dimension $2k-1$, 
 and so the first claim follows from Lemma \ref{lem:bundle criterion}.
 The same reasoning shows  that the cokernel of $\mathscr{V}^{(i)} \to \mathscr{V}$ is locally free of rank $2k-1-i$, and so $\mathscr{V}^{(i)}$ is a local direct summand of rank $i$.
\end{proof}

We next endow  $\mathscr{V}$ with additional structure, derived from the pairing 
\[
b   : \breve{\Lambda}_0  \times \breve{\Lambda}_0 \to \breve{\Z}_p
\]
of \eqref{twisted pairing}.
 For any $\co_{Y^{k,\red}_\Lambda}$-module $\mathscr{F}$, denote by 
\begin{equation}\label{coherent frob twist}
\sigma^* \mathscr{F} = \co_{Y^{k,\red}_\Lambda} \otimes_{\sigma, \co_{Y^{k,\red}_\Lambda}} \mathscr{F}
\end{equation}
the pullback by the  $p$-power Frobenius  $\sigma:\co_{Y^{k,\red}_\Lambda} \to \co_{Y^{k,\red}_\Lambda}$.

\begin{proposition}\label{prop:beta def}
There is a  unique  $\co_{Y_\Lambda^{k,\red}}$-linear map
\[
\beta : \mathscr{V} \otimes \sigma^*\mathscr{V} \to \co_{Y_\Lambda^{k,\red}}
\]
that, when viewed as a pairing 
\[
\beta : \mathscr{V} \times \mathscr{V} \to \co_{Y_\Lambda^{k,\red}}
\]
(linear in the first variable and $\sigma$-linear in the second),  satisfies the following property:
At a point $s \in Y_\Lambda^{k,\red}(\breve{\F}_p)$,  represented by a pair of lattices $(M_0,N_0)$, 
 as in Proposition \ref{prop:RZ tilde points}, the fiber 
\[
\beta_s : \frac{M_0^*}{N_0^*} \times \frac{M_0^*}{N_0^*} \to \breve{\F}_p
\]
is identified with the reduction  of 
\[
pb : M_0^* \times M_0^* \to  \breve{\Z}_p .
\]
\end{proposition}

\begin{proof}
The uniqueness claim is clear, because we have specified the pairing on fibers.  The existence will follow from the next two lemmas.

\begin{lemma}\label{lem:goof map}
There is an $\co_{Y^\red_\Lambda}$-linear morphism
\[
\sigma^* \Fil^0 \mathscr{N}_0 \to \mathscr{M}_1
\]
whose fiber at any $s\in Y^{k,\red}_\Lambda(\breve{\F}_p)$ is identified with 
\[
 \sigma^* (   M_0^* /  p N_0  )  \map{ 1 \otimes x \mapsto p \Phi x }  M_1 / pM_1  .
\]
Here $M$ and $N$ are the lattices in $\breve{\Lambda}[1/p]$ associated to  $s$ as in \eqref{fake minuscule lattices}.
\end{lemma}

\begin{proof}
Over $Y^{k,\red}_\Lambda$ there is a universal diagram \eqref{Y tilde functor} of $p$-divisible groups.
By \eqref{Hpol}, there is a unique isogeny  $\varpi : H^\vee \to H$  making the diagram 
\[
\xymatrix{
{H^\vee}  \ar[dr]_p  \ar[r]^{\varpi}  & { H  }  \ar[d]^{ d^\vee \circ d}  \\
{    }   &  {   H^\vee  } 
}
\]
commute. This induces an $\co_E$-linear morphism of vector bundles
\[
  \mathscr{N}= \underline{\Hom}_{\co_E}  ( \mathscr{D}(\overline{\mathbb{Y}} ) , \mathscr{D}(H^\vee)  )
  \map{  \varpi  \circ }
   \underline{\Hom}_{\co_E}  ( \mathscr{D}(\overline{\mathbb{Y}} ) , \mathscr{D}(H)  )
   = \mathscr{M},
\]
whose fiber a point of $Y^{k,\red}_\Lambda(\breve{\F}_p)$ is identified with
\[
N / pN \map{p }  M/pM .
\]

Now consider the Frobenius   and Verscheibung morphisms
\[
\mathrm{Fr} :  H \to \sigma^*H
\quad \mbox{and}\quad \mathrm{Ver} : \sigma^* H\to H .
\]
These induce morphisms of vector bundles
\[
 \mathscr{D}( H) \map{ V=\mathscr{D}(  \mathrm{Fr})  }  \sigma^*  \mathscr{D}( H )
 \quad \mbox{and}\quad 
 \sigma^* \mathscr{D}( H ) \map{ F=\mathscr{D}(  \mathrm{Ver})  }   \mathscr{D}( H ),
\]
and similarly with $H$ replaced by $\overline{\mathbb{Y}}$ or $H^\vee$.  
We use this to define a morphism 
\[
\sigma^* \mathscr{N} = 
 \underline{\Hom}_{\co_E} ( \sigma^*\mathscr{D}(\overline{\mathbb{Y}} ) , \sigma^* \mathscr{D}(H^\vee)  ) 
 \map{\alpha}
  \underline{\Hom}_{\co_E}  ( \mathscr{D}(\overline{\mathbb{Y}} )  , \mathscr{D}(H^\vee) )
  = \mathscr{N}
\]
by $x \mapsto  F \circ x \circ V$.
The   fiber of $\alpha$ at a geometric point is identified with (the linearization of)
\[
p\Phi :  N  /pN  \to N / pN.
\]

The maps $\alpha$ and $\varpi \circ $ above sit in a commutative diagram
\[
\xymatrix{
{ \sigma^* \mathscr{N} =  \underline{\Hom}_{\co_E} ( \sigma^*\mathscr{D}(\overline{\mathbb{Y}} ) , \sigma^* \mathscr{D}(H^\vee) ) } \ar[r]^\alpha  \ar[d]_{ F \circ}
&
{  \underline{\Hom}_{\co_E} ( \mathscr{D}(\overline{\mathbb{Y}} ) ,  \mathscr{D}(H^\vee) ) = \mathscr{N} }  \ar[d]^{ \varpi  \circ }   \\
{  \underline{\Hom}_{\co_E} ( \sigma^*\mathscr{D}(\overline{\mathbb{Y}} ) ,   \mathscr{D}(H^\vee) ) } \ar[r]_\beta \ar[d]_{V\circ}
&
{  \underline{\Hom}_{\co_E} ( \mathscr{D}(\overline{\mathbb{Y}} ) ,  \mathscr{D}(H) )   = \mathscr{M} }   \\
{  \sigma^*\mathscr{N}= \underline{\Hom} _{\co_E}( \sigma^*\mathscr{D}(\overline{\mathbb{Y}} ) ,   \sigma^* \mathscr{D}(H^\vee) )  } 
}
\]
in  which the arrow labeled $\beta$ sends $x\mapsto \varpi \circ x \circ V$.
The vertical column on the left is exact, and the image of the final vertical arrow is 
\[
\sigma^* \Fil^0\mathscr{N} \subset \sigma^* \mathscr{N} .
\]
 On fibers, the  composition from the upper left corner to the lower right corner of the square is (the linearization of)
 \[
 N/pN \map{ p^2\Phi}  M/pM .
 \]
 This last map need not be $0$, but the relation $p N_0 \subset N_0^*=\Phi^{-1}M_1$ of Proposition \ref{prop:RZ tilde points} (and \eqref{fake dual shift}, for the equality) implies that its restriction
 \[
 N_0/pN_0 \map{ p^2\Phi}  M_1/pM_1
 \]
vanishes.  Thus the diagram above restricts to a commutative diagram of solid arrows
\[
\xymatrix{
{ \sigma^* \mathscr{N}_0 } \ar[rrr]^\alpha  \ar[d]_{ F \circ}   \ar[drrr]^0&   &  &   {   \mathscr{N}_1 }  \ar[d]^{ \varpi  \circ }   \\
{  \underline{\Hom} ( \sigma^*\mathscr{D}(\overline{\mathbb{Y}} )_0 ,   \mathscr{D}(H^\vee)_1 ) } \ar[rrr]_\beta  \ar[d]_{V\circ}
&  &   &   {  \mathscr{M}_1 }   \\
{  \sigma^* \Fil^0 \mathscr{N}_0 } \ar[d] \ar@{-->}[urrr]  \\
{ 0, }
}
\]
and the exactness of the vertical column on the left implies the existence of a unique dotted arrow making the diagram commute.  This dotted arrow is the morphism we seek.
\end{proof}

\begin{lemma}\label{lem:beta2}
There is a $\co_{Y^{k,\red}_\Lambda}$-linear map
\[
\mathscr{N}_0 \otimes \sigma^* \Fil^0 \mathscr{N}_0 \to \co_{Y^{k,\red}_\Lambda}
\]
whose fiber at any $s\in Y^{k,\red}_\Lambda(\breve{\F}_p)$ is identified with the $\breve{\F}_p$-linear map
\[
\frac{N_0}{pN_0} \otimes \sigma^*\left(    \frac{M_0^*}{pN_0} \right) \to \breve{\F}_p
\]
obtained by reducing $x\otimes (1\otimes y) \mapsto p b(x,y) \in \breve{\Z}_p$ for $x\in N_0$ and $y\in M_0^*$.
\end{lemma}

\begin{proof}
The polarization on  $\overline{\mathbb{Y}}$ induces  a perfect bilinear pairing on the constant vector bundle $\mathcal{D}(\overline{\mathbb{Y}})$, which induces an isomorphism
\[
\mathcal{D}(\overline{\mathbb{Y}})_0 \otimes  \mathcal{D}(\overline{\mathbb{Y}})_1 \iso  \co_{Y^{k,\red}_\Lambda}.
\]
Similarly, there is a  canonical perfect bilinear pairing
\[
\mathcal{D}( H^\vee )_0 \otimes \mathcal{D}(H)_1 \to   \co_{Y^{k,\red}_\Lambda}
\]
on the vector bundles associated to the universal $p$-divisible groups $H$ and $H^\vee$ over $Y^{k,\red}_\Lambda$.  
These induce a perfect bilinear pairing
\begin{align*}
\mathscr{N}_0  \otimes  \mathscr{M}_1
& =  \underline{\Hom}( \mathcal{D}(\overline{\mathbb{Y}})_0 , \mathcal{D}(H^\vee)_0) 
\otimes \underline{\Hom}( \mathcal{D}(\overline{\mathbb{Y}})_1 , \mathcal{D}(H)_1)  \\
& \iso  \underline{\Hom}( \mathcal{D}(\overline{\mathbb{Y}})_0 \otimes \mathcal{D}(\overline{\mathbb{Y}})_1 , \mathcal{D}(H^\vee)_0 \otimes  \mathcal{D}(H)_1)  \\
& \to \co_{Y^{k,\red}_\Lambda}.
\end{align*}
Pulling this  back via the morphism $\sigma^*\Fil^0 \mathscr{N}_0 \to \mathscr{M}_1$ of Lemma \ref{lem:goof map} yields a pairing $\mathscr{N}_0 \otimes \sigma^* \Fil^0 \mathscr{N}_0 \to \co_{Y^{k,\red}_\Lambda}$ having the desired form on  $\breve{\F}_p$-valued points.
\end{proof}

We  now complete the proof of Proposition \ref{prop:beta def}. 
By construction, there is a canonical map $\mathscr{M} \to \mathscr{N}$ respecting filtrations, and we claim that the  composition 
\[
\mathscr{M}_0 \otimes \sigma^* \Fil^0 \mathscr{M}_0 \to \mathscr{N}_0 \otimes \sigma^* \Fil^0 \mathscr{N}_0 \to \co_{Y^{k,\red}_\Lambda}
\]
is trivial. Indeed, taking  fibers  at a point $(M_0,N_0) \in Y_\Lambda^{k,\red}(\breve{\F}_p)$,  the composition is identified with the pairing
\[
\frac{M_0}{pM_0} \times \frac{N_0^*}{pM_0} \to  \frac{N_0}{pN_0} \times \frac{M_0^*}{pN_0}
 \map{pb( - , - )} \breve{\F}_p,
\]
which is trivial as
$
b(M_0,N_0^*) \subset b(M_0,M_0^*) \subset \breve{\Z}_p.
$

Recalling the definitions \eqref{VWdef}, the triviality of the above composition implies that  the morphism of Lemma \ref{lem:beta2}  descends to a morphism
\[
\mathscr{W} \otimes \sigma^* \mathscr{V} \to \co_{Y^{k,\red}_\Lambda}.
\]
The composition
\[
\mathscr{V} \otimes \sigma^*\mathscr{V} \to \mathscr{W} \otimes \sigma^*\mathscr{V} \to \co_{Y_\Lambda^{k,\red}}
\]
(the first arrow is the canonical map $\mathscr{V} \to \mathscr{W}$ on the first tensor factor, and the identity on the second)  
at last defines the desired pairing $\beta$. 
\end{proof}

Modulo the fact that the pairing $\beta$ is not bilinear, the following proposition essentially says that the radical of $\beta$ is $\mathscr{V}^{(1)}$, and
\[
\mathscr{V}^{(k)}/ \mathscr{V}^{(1)}  \subset  \mathscr{V}/ \mathscr{V}^{(1)}
\]
 is a maximal isotropic subbundle.

\begin{proposition}\label{prop:beta isotropy}
We have 
\[
\beta( \mathscr{V}^{(k)}  \otimes \sigma^*\mathscr{V}^{(k)} ) =0 \quad\mbox{and}\quad \beta(\mathscr{V}^{(1)} \otimes \sigma^* \mathscr{V})=0.
\] 
Moreover,  the induced map
\[
\mathscr{V}^{(k)}/ \mathscr{V}^{(1)}  \map{ v \mapsto ( w \mapsto \beta(v\otimes w) ) }  \underline{\Hom}( \sigma^*   \mathscr{V} /  \sigma^*\mathscr{V}^{(k)}   , \co_{ Y_\Lambda^{k,\red} } )
\]
is an isomorphism.
\end{proposition}

\begin{proof}
All claims can be checked on fibers, where they follow from Propositions \ref{prop:VWfibers} and \ref{prop:beta def}.
\end{proof}

Recalling Theorem \ref{thm:YtoDL}, over  $Y^{k,\red}_\Lambda \iso \mathrm{DL}_\Lambda^k$ we have the universal flag 
\[
0 \stackrel{k-1}{\subset}  \mathscr{J} \stackrel{1}{\subset} \mathscr{K}^\perp 
\stackrel{n-2k}{\subset}      \mathscr{K} \stackrel{1}{\subset} \mathscr{J}^\perp \stackrel{k-1}{\subset}
 \breve{\Lambda}_0 \otimes \co_S
\]
of Definition \ref{def:non-minuscule DL}.
The relation between this flag and the filtered vector bundle $\mathscr{V}$
is a bit unclear.  
 It is not hard to see that there is a short exact sequence
\[
0 \to  \mathscr{V}^{(k)} \to \mathscr{V} \to \mathscr{J} \to 0,
\]
and  isomorphisms
\[
\mathscr{V}^{(1)} \iso  
\frac{  \mathscr{J}^\perp }{  \mathscr{K} } 
\quad\mbox{and} \quad    \mathscr{V}^{(k)} \iso \frac{ \breve{\Lambda}_0 \otimes \co_{Y^{k,\red}_\Lambda}}{  \mathscr{K} }.
\]
In particular,  $\mathscr{V}$ can be realized as an extension of two vector bundles that are each intrinsic to the Deligne-Lusztig variety $\mathrm{DL}_\Lambda^k$, in the sense that their construction does not rely on the isomorphism $Y^{k,\red}_\Lambda \iso \mathrm{DL}_\Lambda^k$.  It is not obvious to authors how one can express $\mathscr{V}$ itself in a way intrinsic to 
$\mathrm{DL}_\Lambda^k$.


\subsection{Analysis of $\pi_1$}


We can now make explicit the structure of the morphism
\[
 \pi_1 :  \mathbf{RZ}^{\le k}_\Lambda \to   Y^k_\Lambda ,
 \]
at least on the level of underlying reduced schemes, in terms of the morphism
$\beta : \mathscr{V} \otimes \sigma^*\mathscr{V} \to \co_{Y^{k,\red}_\Lambda}$ of Proposition \ref{prop:beta def}.
 
To this end, consider the  scheme
\[
R^{\le k}_\Lambda \to Y^{k,\red}_\Lambda
\]
 whose functor of points assigns to any $Y_\Lambda^{k,\red}$-scheme $S$ the set
\[
R^{\le k}_\Lambda(S) 
=
\left\{  
\begin{array}{c}  \mbox{rank $k-1$ local direct summands } \\  \mathscr{F} \subset \mathscr{V}_S  \mbox{ satisfying } \beta (\mathscr{F} \otimes \sigma^* \mathscr{F}) =0  \end{array}  \right\} .
\]
Denote by $R^k_\Lambda \subset R^{\le k}_\Lambda$ the open subscheme with functor of points 
\[
R^k_\Lambda(S) 
=
\{ 
\mathscr{F} \in R^{\le k}_\Lambda(S)  :  \mathscr{V}_S =  \mathscr{F} \oplus \mathscr{V}^{(k)}_S
\} ,
\]
where $\mathscr{V}^{(k)} \subset \mathscr{V}$ is the rank $k$ local direct summand defined by \eqref{Vk}.

\begin{proposition}\label{prop:Rstructure}
The morphism $R_\Lambda^k \to Y^{k,\red}_\Lambda$ is smooth of relative dimension $k-1$.
\end{proposition}

\begin{proof}
We will show that the morphism in question is formally smooth.  Suppose we have a commutative diagram of solid arrows
\[
\xymatrix{
{  \Spec(B/I) }  \ar[r]   \ar[d]   & {  R_\Lambda^k  }  \ar[d]   \\
{ \Spec( B ) }  \ar[r] \ar@{-->}[ru]   &  {  Y^{k,\red}_\Lambda  } 
}
\]
in which the left vertical arrow is a square-zero thickening of affine $\breve{\F}_p$-schemes.  We must show that, Zariski locally on $\Spec( B)$, there exists a dotted arrow making the diagram commute.

The  pullbacks to $\Spec( B )$ of the vector bundles 
$
\mathscr{V}^{(1)} \subset \mathscr{V}^{(k)} \subset \mathscr{V}
$ 
on $Y_\Lambda^{k,\red}$ correspond to inclusions of locally  free $B$-modules
\[
V^{(1)} \subset V^{(k)} \subset V,
\] 
with $V$ endowed with a function 
$
\beta : V \times V \to B
$
that is $B$-linear in the first variable and $\sigma$-linear in second.
The top horizontal arrow in the diagram corresponds to a choice of complementary summand 
\[
F \subset V/IV
\]
  to $V^{(k)}/ I V^{(k)}$,  satisfying the isotropy condition $\beta( F , F) =0$.

\begin{lemma}
The set of dotted arrows making the diagram commute admits a simply transitive action of  (the additive group underlying) the $B/I$-module $\Hom(F, I V^{(1)})$.
\end{lemma}

\begin{proof}
Consider  the set  $\mathcal{X}$ of all lifts of $F$ to $B$-submodules $F' \subset V$ satisfying 
\[
V = F' \oplus V^{(k)}  
\]
(with no isotropy condition on $F'$). By standard arguments, this
is a principal homogenous space under  $\Hom(F, I V^{(k)}  )$, where 
the action is defined as follows: given a point $F' \in \mathcal{X}$ and a $\phi : F \to I V^{(k)}$, we let $\phi'$ denote the composition
\[
F' \to F'/IF' = F \map{\phi} I V^{(k)},
\]
and define
\[
\phi+ F' = \{ x-\phi'(x) : x\in F' \}.
\]

The set of dotted arrow making the diagram commute is in bijection with the set of all $F' \in \mathcal{X}$ that satisfy the isotropy condition $\beta(F',F')=0$.
Let us now fix one $F' \in \mathcal{X}$, and parametrize those $\phi \in \Hom(F, I V^{(k)}  )$ for which $\phi+F'$ is isotropic.

To say that $\phi + F'$ is isotropic is equivalent to saying that 
\[
\beta(  x -\phi'(x)  ,  y-\phi'(y) )=0
\]
 for all $x,y\in F'$.   By assumption $I^2= 0$, and so the $\sigma$-linearity of $\beta$ in the second variable implies 
\[
\beta( F' , \phi'(F')) \subset \beta(F', IV^{(k)}) = I^p \beta(F',  V^{(k)}) =0.
\]
The isotropy condition on $\phi+F'$ therefore simplifies to 
\begin{equation}\label{isotropy lift}
\beta(x,y) = \beta( \phi'(x) , y) 
\end{equation}
for all $x,y\in F'$.

Proposition \ref{prop:beta isotropy} tells us that the natural map
\[
V^{(k)}/ V^{(1)}  \map{ x \mapsto ( y \mapsto \beta(x, y) ) } 
 \Hom_\sigma(  V /  V^{(k)}   , B )
 =  \Hom_\sigma(  F'  , B )
\]
is an isomorphism, where $\Hom_\sigma$ means $\sigma$-linear homomorphisms of $B$-modules.
In particular, there is a unique homomorphism
\[
\psi' : F' \to  V^{(k)}/ V^{(1)} 
\]
satisfying
$
\beta(x,y) =\beta(\psi'(x),y)
$
for all $x,y\in F'$.  Using the isotropy of the original $F \subset V$, we see that $\psi'$ takes values in 
$IV^{(k)}/IV^{(1)}$, and hence factors through a morphism
\[
F= F'/IF'  \map{\psi} IV^{(k)}/IV^{(1)}.
\]

What we have shown is that the $\phi \in \Hom(F, I V^{(k)}  )$ for which $\phi+F'$ is isotropic are precisely those that satisfy \eqref{isotropy lift}, and these are precisely the $\phi$ that lift $\psi \in \Hom(F, I V^{(k)} /IV^{(1)} )$.  This is clearly a principal homogeneous space under 
$\Hom(F,  IV^{(1)} )$.
\end{proof}

The lemma shows first that the set of dotted arrows making the diagram commute is nonempty, and hence 
 $R_\Lambda^k \to Y^{k,\red}_\Lambda$ is formally smooth.  
 
 By taking $B=\breve{\F}_p[\epsilon]$ to be the ring of dual numbers, the lemma also implies that the relative tangent space to this morphism at a point $s\in  R_\Lambda^k (\breve{\F}_p)$, corresponding to an isotropic direct summand $\mathscr{F}_s \subset \mathscr{V}_s$, is isomorphic to the $\breve{\F}_p$-vector space $\Hom( \mathscr{F}_s,  \mathscr{V}^{(1)}_s)$ of dimension $k-1$.  Hence all  closed fibers of the morphism in question have dimension $k-1$.
\end{proof}

\begin{theorem}\label{thm:enhanced structure}
There is a canonical isomorphism of $Y_\Lambda^{k,\red}$-schemes
\[
\mathbf{RZ}_\Lambda^{ \le k,\red} \iso R^{\le k}_\Lambda,
\]
restricting to an isomorphism
 $
 \mathbf{RZ}^{k,\red}_\Lambda \iso R^{k}_\Lambda .
 $
\end{theorem}

\begin{proof}
Over $\mathbf{RZ}_\Lambda^{\le k, \red}$, the universal diagram \eqref{RZ tilde functor} determines morphisms of filtered vector bundles
\[
\mathscr{D}(H) \to \mathscr{D}(X) \to \mathscr{D}(H^\vee).
\]
These induce morphisms 
$
 \mathscr{M}_0 \to \mathscr{L}_0 \to \mathscr{N}_0
$
between the filtered vector bundles from  \eqref{the L bundle} and the proof of Theorem \ref{thm:YtoDL}.

Recall from  \eqref{VWdef} that $\mathscr{V}$ is a quotient of $\Fil^0 \mathscr{N}_0$.
The key step of the proof is to note that the coherent sheaf
\[
\mathscr{F} \define \mathrm{Image}( \Fil^0 \mathscr{L}_0 \to  \Fil^0 \mathscr{N}_0 \to   \mathscr{V} )
\]
on $\mathbf{RZ}_\Lambda^{\le k, \red}$ is a rank $k-1$  local direct summand of $\mathscr{V}$ satisfying $b(\mathscr{F} \otimes  \sigma^*\mathscr{F} ) =0$, while  the coherent sheaf
\[
\mathscr{H} \define \ker(  \Fil^0 \mathscr{N}_0 \to \mathscr{V} / \mathscr{F}) 
\]
is a local direct summand of $ \Fil^0 \mathscr{N}_0$ satisfying 
\begin{equation}\label{magic hodge}
\Fil^0 \mathscr{L}_0 = \ker( \mathscr{L}_0 \to \mathscr{N}_0 / \mathscr{H} ) .
\end{equation}
All of these claims  can be verified on fibers at $\breve{\F}_p$-valued points, where (under the identifications of Proposition \ref{prop:RZ tilde points}) we have
\begin{align*}
\Fil^0 \mathscr{L}_{0,s}=  L_0^* / pL_0  &  \subset  L_0/pL_0 = \mathscr{L}_{0,s}  \\
\mathscr{F}_s=  L_0^* / N_0^*  &  \subset  M_0^*/N_0^*  = \mathscr{V}_s  \\
\mathscr{H}_s=  L_0^* / p N_0  &  \subset  M_0^*/p N_0  = \Fil^0 \mathscr{N}_{0,s}.
\end{align*}
In particular,  the isotropic subbundle $\mathscr{F} \subset \mathscr{V}$ just constructed   defines a morphism
\begin{equation}\label{RZtilde is grass}
\mathbf{RZ}_\Lambda^{ \le k, \red }  \to R^{\le k}_\Lambda,
\end{equation}
which is bijective on $\breve{\F}_p$-valued points.

We will use the auxiliary vector bundle $\mathscr{H} \subset \Fil^0 \mathscr{N}_0$ to show  that the morphism we have constructed is formally unramified.  To this end,  assume  we have a commutative diagram
\[
\xymatrix{
{ S }  \ar[r] \ar[d] &  { \mathbf{RZ}_\Lambda^{\le k,\red}  }  \ar[d]  \\
{  \widetilde{S} }  \ar[r] \ar@{-->}[ur] & { R_\Lambda   } 
}
\]
of solid arrows, in which the left vertical arrow is a square-zero thickening.
To show that our morphism is formally unramified we must show that there is at most one dotted arrow making the diagram commute, so suppose we have two such arrows
$
a,b : \widetilde{S} \to \mathbf{RZ}_\Lambda^{ \le k,\red} .
$

Because $a$ and $b$ define the same $\widetilde{S}$-point of $R_\Lambda$, there are  canonical identifications (respecting filtrations)
\[
\xymatrix{
{ a^*\mathscr{M}_0 } \ar[r]  \ar@{=}[d]  & {  a^*\mathscr{N}_0 }  \ar[r]   \ar@{=}[d] & {  a^*\mathscr{V} } \ar[r]  \ar@{=}[d] & 0  \\ 
{ b^*\mathscr{M}_0 } \ar[r] & {  b^*\mathscr{N}_0 }  \ar[r]  & {  b^*\mathscr{V} }  \ar[r]& 0  ,
}
\]
 of vector bundles on $\widetilde{S}$, the last of which
 identifies $a^*\mathscr{F} = b^* \mathscr{F}$.  It follows that there are canonical identifications 
 \[
\xymatrix{
{ a^*\mathscr{H}} \ar[r]  \ar@{=}[d]  & {  a^* \Fil^0\mathscr{N}_0 }  \ar[r]   \ar@{=}[d] & {  a^*(\mathscr{V}/\mathscr{F} )} \ar[r]  \ar@{=}[d] & 0  \\ 
{ b^*\mathscr{H} } \ar[r] & {  b^*\Fil^0\mathscr{N}_0 }  \ar[r]  & {  b^*(\mathscr{V} / \mathscr{F})}  \ar[r]& 0  
}
\]
 of vector bundles on $\widetilde{S}$.

Now consider the $p$-divisible groups $a^*X$ and $b^*X$ over $\widetilde{S}$.
We do not yet know that these are isomorphic, but  they  are deformations of the same $p$-divisible group over $S$.  By Grothendieck-Messing theory there is a canonical isomorphism 
\begin{equation}\label{magic hodge strong}
a^* \mathscr{D}(X) = \mathscr{D}(a^*X) \iso \mathscr{D}(b^*X)=b^* \mathscr{D}(X) 
\end{equation}
of vector bundles on $\widetilde{S}$.
By Remark \ref{rem:untwist L}  this induces an isomorphism $a^*\mathscr{L} \iso b^*\mathscr{L}$, 
and we now have canonical identifications
\[
\xymatrix{
 {  a^* \mathscr{L}_0 }  \ar[r]   \ar@{=}[d] & {  a^*(\mathscr{N}_0/\mathscr{H} )}  \ar@{=}[d]   \\ 
 {  b^*\mathscr{L}_0 }  \ar[r]  & {  b^*(\mathscr{N}_0/\mathscr{H} )}  
}
\]
of vector bundles on $\widetilde{S}$.
It  follows from \eqref{magic hodge} that the isomorphism on the left identifies $a^*\Fil^0 \mathscr{L}_0 = b^*\Fil^0 \mathscr{L}_0$.  
As with \eqref{MN filters}, the  isomorphism of Remark \ref{rem:untwist L} identifies the $\Fil^0$ on source and target, and hence  \eqref{magic hodge strong} identifies
\[
a^* \Fil^0 \mathscr{D}(X)_0 = b^* \Fil^0 \mathscr{D}(X)_0.
\]

As the principal polarization on $X$ induces a perfect bilinear pairing between $\mathscr{D}(X)_0$ and $\mathscr{D}(X)_1$,  under which  the Hodge filtrations are exact annihilators of each other, the same identification holds if we replace $\mathscr{D}(X)_0$ with $\mathscr{D}(X)_1$.  
In other words, \eqref{magic hodge strong} respects Hodge filtrations, and so $a^*X\iso b^*X$ by Grothendieck-Messing theory.
From this it follows easily that $a=b$.

The  unramified morphism \eqref{RZtilde is grass} is proper (its source  is projective over $\breve{\F}_p$) and bijective on closed points, so  is an isomorphism by Lemma \ref{lem:general iso}. 

For the final claim, fix a point 
\[
s\in \mathbf{RZ}_\Lambda^{\le k,\red}(\breve{\F}_p) \iso R_\Lambda^{\le k}(\breve{\F}_p)
\]
 represented by a triple $(L_0,M_0,N_0)$ as in Proposition \ref{prop:RZ tilde points}.  
 In particular,
 \begin{equation}\label{recover R stratum}
 L_0 \cap \breve{\Lambda}_0  \subset  M_0 \stackrel{k-1}{\subset} \breve{\Lambda}_0.
 \end{equation}
 
By definition,  $s \in R_\Lambda^k(\breve{\F}_p)$ holds if and only if equality holds in
\[
 \mathscr{F}_s + \mathscr{V}_s^{(k)} = \frac{L_0^*+ \breve{\Lambda}_0}{N_0^*} \subset \frac{M_0^*}{N_0^*} = \mathscr{V}_s.
\]
 Dualizing, this is equivalent to $L_0 \cap \breve{\Lambda}_0 = M_0$, which is equivalent to  
\[
L_0 \cap \breve{\Lambda}_0   \stackrel{k-1}{\subset} \breve{\Lambda}_0
\]
by \eqref{recover R stratum}, which is equivalent to  $s \in \mathbf{RZ}_\Lambda^k(\breve{\F}_p)$ by 
Proposition \ref{prop:strata points} and the final claim of Proposition \ref{prop:RZ tilde points}.
\end{proof}

\begin{corollary}\label{cor:grassmannian}
There is a  smooth morphism
\[
\mathrm{RZ}_\Lambda^{k,\red} \iso R^k_\Lambda  \to Y_\Lambda^{k,\red}\iso \mathrm{DL}_\Lambda^k
\]
of relative dimension $k-1$.  In particular, $\mathrm{RZ}_\Lambda^{k,\red}$ is itself smooth of dimension $n-2$.
\end{corollary}

\begin{proof}
The first isomorphism is the composition
\[
\mathrm{RZ}_\Lambda^{k,\red} \iso \mathbf{RZ}_\Lambda^{k,\red} \iso R^k_\Lambda
\]
of the isomorphisms of Proposition \ref{prop:pi_0 isomorphism} and Theorem \ref{thm:enhanced structure}.  The arrow in the middle is the smooth morphism of Proposition \ref{prop:Rstructure}.  The final isomorphism is that of Theorem \ref{thm:YtoDL}.  The final claim of the corollary now follows from Proposition \ref{prop:DLstructure}.
\end{proof}

\begin{remark}
In the special case $k=1$ we  have
\[
\mathrm{RZ}_\Lambda^{1,\red} \iso R^1_\Lambda = Y_\Lambda^{1,\mathrm{red}} \iso \mathrm{DL}_\Lambda^1 .
\]
\end{remark}


\section{Supplementary results when $n$ is even}
\label{s:extremal supplement}


Intuitively,  points of the closed subscheme $\mathrm{RZ}_\Lambda \subset \mathrm{RZ}$ parametrize $p$-divisible groups  that are relatively close to the framing object $\Lambda \otimes   \overline{\mathbb{Y}}$ of signature $(0,n)$.   
Proposition \ref{prop:strata points} can be understood as saying that in the decomposition
\[
\mathrm{RZ}^\red_\Lambda = \bigsqcup_{1 \le k \le \lfloor n/2 \rfloor } \mathrm{RZ}_\Lambda^{k,\red},
\]
 the locally closed subschemes indexed by smaller $k$ parametrize points that are closer to $\Lambda \otimes   \overline{\mathbb{Y}}$ than those parametrized by larger $k$.

It turns out that in the extremal case in which $n$ is even and $k=n/2$, each point of $ \mathrm{RZ}_\Lambda^{n/2}$, while as far  from $\Lambda \otimes   \overline{\mathbb{Y}}$ as it allowed to be, is very close to one of finitely many other framing objects.
 These other framing objects, which again have signature $(0,n)$ but are endowed with  non-principal polarizations,   provide a different way to parametrize the points of $ \mathrm{RZ}_\Lambda^{n/2}$.  This is what we explore in this section.


\subsection{Another partial Rapoport-Zink space}


As always, $\Lambda$ is a self-dual hermitian $\co_E$-lattice of rank $n\ge 2$.
the hermitian form is denoted $h(-,-)$.

\begin{definition}\label{def:scalar dual}
An $\co_E$-lattice $\Lambda' \subset \Lambda[1/p]$ is \emph{scalar-self-dual} if there exists a $c\in \Q_p^\times$ such that 
\[
c \Lambda' = \{ x \in \Lambda[1/p] : h(\Lambda',x) \subset \co_E \}.
\]
\end{definition}

We are especially interested in scalar-self-dual lattices $\Lambda'$ satisfying 
\begin{equation}\label{scalar inclusions}
p\Lambda \subsetneq \Lambda' \subsetneq \Lambda.
\end{equation}
The self-duality of $\Lambda$ then forces both $\ord_p(c)= -1$ and
\[
\mathrm{length}_{\co_E}( \Lambda/\Lambda') =\frac{n}{2}  .
\]
In particular, such lattices can only exist when $n$ is even, and we assume this for the remainder of 
\S \ref{s:extremal supplement}.
The scalar self-duality condition on $\Lambda'$ can then be interpreted as saying that 
 $\Lambda'/p\Lambda \subset \Lambda/p\Lambda$ is maximal isotropic under the natural $\co_E/p\co_E$-valued hermitian form

\begin{definition}\label{def:heart}
For a scalar-self-dual $\co_E$-lattice $\Lambda'$ as in \eqref{scalar inclusions},  and an  
$\breve{\F}_p$-scheme $S$, define  $\mathrm{RZ}_{\Lambda'}^\heartsuit(S)$ to be the set of isomorphism classes of quadruples $(X,\lambda_X,\alpha_X,\beta_X)$ in which 
\begin{itemize}
\item
$X$ is a $p$-divisible group over $S$ equipped with an $\co_E$-action of signature $(2,n-2)$, 
\item
$\lambda_X : X \to X^\vee$ is an $\co_E$-linear principal polarization, 
\item
$\alpha_X$ and $\beta_X$ are $\co_E$-linear isogenies
\[
\Lambda' \otimes \overline{\mathbb{Y}}_S
  \map{\alpha_X}
  X
\map{\beta_X}
p^{-1} \Lambda' \otimes \overline{\mathbb{Y}}_S
\]
whose composition is induced by the inclusion $\Lambda' \subset p^{-1}\Lambda'$, and  such that  $\alpha_X^*\lambda_X$ agrees with the canonical (non-principal) polarization on $\Lambda' \otimes \overline{\mathbb{Y}}_S$.
\end{itemize}.
\end{definition}

The functor $\mathrm{RZ}_{\Lambda'}^\heartsuit$ of Definition \ref{def:heart} is represented by a projective $\breve{\F}_p$-scheme, denoted the same way.
For any point  $(X,\lambda_X,\alpha_X,\beta_X) \in \mathrm{RZ}_{\Lambda'}^\heartsuit(S)$ there is a unique quasi-isogeny $\varrho_X$ making the diagram 
\[
\xymatrix{
& {   X  } \ar[dr]^{\beta_X}  \ar@{-->}[d]^{\varrho_X} \\
{ \Lambda' \otimes \overline{\mathbb{Y}}_S  }  \ar[r]  \ar[ur]^{\alpha_X}& {  \Lambda \otimes \overline{\mathbb{Y}}_S   }  \ar[r]&  {  p^{-1}  \Lambda' \otimes \overline{\mathbb{Y}}_S   }  
}
\]
commute, and this determines a point  $(X,\lambda_X,\varrho_X) \in \mathrm{RZ}_{\Lambda}(S)$.
We use this construction to regard
\[
\mathrm{RZ}_{\Lambda'}^\heartsuit \subset \mathrm{RZ}_{\Lambda}
\]
as a closed subscheme.

\begin{proposition}\label{prop:heartpoints}
For any scalar self-dual lattice $\Lambda'$ satisfying \eqref{scalar inclusions}, 
the bijection of Corollary \ref{cor:half hyperspecial} restricts to a bijection
\[
\mathrm{RZ}^\heartsuit_{\Lambda'}( \breve{\F}_p) 
\iso
\left\{
\begin{array}{c}  \breve{\Z}_p\mbox{-lattices }   \\  L_0  \subset  \breve{\Lambda}_0[1/p]  \end{array} : 
\begin{array}{c}
pL_0 \subset L_0^* \stackrel{2}{\subset} L_0\\
    \breve{\Lambda}'_0  \subset  L_0  \subset p^{-1} \breve{\Lambda}'_0 
  \end{array}  \right\} .
\]
Moreover, the two chains of inclusions on the right hand side are equivalent to the single chain condition
\[
  \breve{\Lambda}'_0   \stackrel{\frac{n}{2}-1} { \subset } L_0^* \stackrel{2}{\subset}  L_0  \stackrel{\frac{n}{2}-1} { \subset }p^{-1} \breve{\Lambda}'_0   .
  \]
\end{proposition}

\begin{proof}
The first claim follows directly from the construction of the bijection of Corollary \ref{cor:half hyperspecial}.
The scalar-self-duality assumption on $\Lambda'$ implies that $\breve{\Lambda}_0'\oplus p^{-1} \breve{\Lambda}_1'$ is self-dual.  Using this and  \eqref{dual shift}, we deduce that the right dual   operator \eqref{right dual}  interchanges the lattices $\breve{\Lambda}'_0$ and $p^{-1} \breve{\Lambda}'_0$.  
Applying this operator throughout
\[
\breve{\Lambda}'_0  \subset L_0  \stackrel{r}{\subset} p^{-1} \breve{\Lambda}'_0
\]
(this is the definition of $r$) therefore results in
\[
\breve{\Lambda}_0'\stackrel{r}{\subset}  L_0^* \subset  p^{-1} \breve{\Lambda}_0' ,
\]
and the second claim follows immediately.
\end{proof}

\begin{proposition}\label{prop:heart decomp}
For every point $s\in \mathrm{RZ}_\Lambda^{n/2}(\breve{\F}_p)$ there exists a  scalar-self-dual lattice $\Lambda'$ as in \eqref{scalar inclusions} such that 
$s\in \mathrm{RZ}_{\Lambda'}^\heartsuit   (\breve{\F}_p)$.
\end{proposition}

\begin{proof}
Under the bijection of Proposition \ref{prop:RZ tilde points}, the point $s$ corresponds to a chain of lattices
\[
 p\breve{\Lambda}_0 \stackrel{\frac{n}{2}-1}{\subset}  p M_0^* \stackrel{1}{\subset}  pN_0 =  N_0^* \stackrel{1}{\subset}  M_0 \stackrel{ \frac{n}{2} -1}{\subset} \breve{\Lambda}_0  ,
 \]
together with a lattice $L_0 \subset \breve{\Lambda}_0$ satisfying
 \[
 pL_0 \subset  L_0^* \stackrel{2}{\subset} L_0\quad \mbox{and}\quad   M_0 \stackrel{\frac{n}{2}}{\subset} L_0 \stackrel{\frac{n}{2}-1}{\subset} N_0. 
 \]

The essential thing is the middle equality $pN_0 =  N_0^*$, which implies
\[
\Phi^{-2} N_0 \stackrel{ \eqref{twisted dual properties}}{=}N_0^{**} = ( p N_0) ^* = N_0.
\]
It follows that the $\co_{\breve{E}}$-lattice
$
 p( N_0 \oplus \Phi N_0) \subset \breve{\Lambda}_0  \oplus \breve{\Lambda}_1 
 $
 is fixed by $\Phi$, and we define an $\co_E$-lattice
 \[
 \Lambda' =  p( N_0 \oplus \Phi N_0)^{ \Phi=\mathrm{id} } \subset \Lambda.
 \]
  Using \eqref{dual shift} we see that 
  $
 N_0 \oplus \Phi N_0^*  \subset \breve{\Lambda}[1/p] 
$
is self-dual under the hermitian pairing.  
This implies that the dual lattice of $N_0 \oplus \Phi N_0$ is $p(N_0 \oplus \Phi N_0)$, and taking $\Phi$-fixed points shows that  the dual lattice of $\Lambda'$ is $p^{-1}\Lambda'$.  It follows that $\Lambda'$ is scalar-self-dual and satisfies \eqref{scalar inclusions}.

The inclusions $M_0 \subset L_0 \subset N_0$ imply
\[
 \breve{\Lambda}_0' = pN_0 \subset L_0 \subset  N_0= p^{-1}\breve{\Lambda}_0' ,
\]
and Proposition \ref{prop:heartpoints} shows that
$s\in \mathrm{RZ}_{\Lambda'}^\heartsuit   (\breve{\F}_p)$.
\end{proof}


\subsection{Another Deligne-Lusztig variety}


We continue to assume that $n$ is even, and fix a 
 scalar-self dual lattice $\Lambda' \subset \Lambda[1/p]$ satisfying \eqref{scalar inclusions}.

 The pairing $b$ of \eqref{twisted pairing} determines a pairing 
\[
b' =p^{-1}b   : \breve{\Lambda}'_0  \times \breve{\Lambda}'_0 \to \breve{\Z}_p.
\]
Exactly as in \S \ref{ss:DL}, this induces a pairing of $\co_S$-modules
\[
 ( \breve{\Lambda}_0 \otimes_{ \breve{\Z}_p} \co_S ) \times ( \breve{\Lambda}_0 \otimes_{ \breve{\Z}_p}  \co_S )
  \to \co_S, 
\]
 linear in the first variable and $\sigma$-linear in the second, for any $\breve{\F}_p$-scheme $S$.
 Once again, for any local direct summand $\mathscr{F} \subset  \breve{\Lambda}_0 \otimes_{ \breve{\Z}_p}  \co_S$, we denote   its left annihilator under this pairing by
\[
\mathscr{F}^\perp = \{ x \in  \breve{\Lambda}_0 \otimes_{ \breve{\Z}_p}   \co_S :  b'( x, \mathscr{F} ) =0 \}.
\]

\begin{definition}\label{def:heartDL}
Define $\mathrm{DL}_{\Lambda'}^\heartsuit$ to be the projective $\breve{\F}_p$-scheme whose functor of points assigns to any $\breve{\F}_p$-scheme $S$ the set of flags of $\co_S$-module local direct summands
\[
0 \stackrel{\frac{n}{2}-1} { \subset }  \mathscr{F} 
\stackrel{ 2}{\subset} \mathscr{F}^\perp \stackrel{\frac{n}{2}-1} { \subset }
 \breve{\Lambda}'_0 \otimes_{\breve{\Z}_p} \co_S
\]
of the indicated coranks.   
\end{definition}

\begin{proposition}
The scheme $\mathrm{DL}_{\Lambda'}^\heartsuit$  of Definition \ref{def:heartDL} is a Deligne-Lusztig variety for the unitary group of the finite hermitian space $\Lambda'/p\Lambda'$ (where the hermitian form on $\Lambda'$ is first multiplied by $p^{-1}$ to make it self-dual).
It is irreducible and smooth of dimension $n-2$.
\end{proposition}

\begin{proof}
The proof is the same as for Proposition \ref{prop:DLstructure}.
\end{proof}

\begin{theorem}\label{thm:heart comparison}
There is a isomorphism  
\[
\mathrm{RZ}_{\Lambda'}^{\heartsuit,\red} \iso \mathrm{DL}_{\Lambda'}^\heartsuit 
\]
sending an $\breve{\F}_p$-valued point of the left hand side  to the flag
\[
0 \subset  \frac{pL_0^*}{p\breve{\Lambda}'_0} 
\subset \frac{pL_0}{p\breve{\Lambda}'_0}  \subset
  \frac{ \breve{\Lambda}'_0 }{ p\breve{\Lambda}'_0 } 
\]
determined by the bijection of  Proposition \ref{prop:heartpoints}.
\end{theorem}

\begin{proof}
This is similar to the proof of Theorem \ref{thm:YtoDL}.  
The universal isogeny $\beta_X:X  \to   p^{-1} \Lambda' \otimes \overline{\mathbb{Y}} $
over $\mathrm{RZ}_{\Lambda'}^{\heartsuit,\red}$ induces a 
 morphism of vector bundles
\[
  \mathscr{D}(X)  \to \mathscr{D}( p^{-1} \Lambda' \otimes \overline{\mathbb{Y}} ) ,
\]
and we define coherent sheaves 
\begin{equation}\label{Gflag}
0 \subset \mathscr{G} \subset \mathscr{G}^\dagger \subset \mathscr{D}( p^{-1} \Lambda' \otimes \overline{\mathbb{Y}} )_0 
\end{equation}
on $\mathrm{RZ}_{\Lambda'}^{\heartsuit,\red}$ by
\begin{align*}
\mathscr{G} 
&  = \mathrm{Image} \big(  \Fil^0 \mathscr{D}(X)_0  \to 
 \mathscr{D}( p^{-1} \Lambda' \otimes \overline{\mathbb{Y}} )_0  \big) \\
 \mathscr{G}^\dagger
&  = \mathrm{Image} \big(  \mathscr{D}(X)_0  \to 
 \mathscr{D}( p^{-1} \Lambda' \otimes \overline{\mathbb{Y}} )_0  
 \big).
\end{align*}

Fix an isomorphism $\co_{\breve{E}}$-modules $D(\overline{\mathbb{Y}}) \iso \co_{\breve{E}}$.
This determines isomorphisms of $\co_{\breve{E}}$-modules
\[
D( \Lambda'\otimes \overline{\mathbb{Y}}) \iso  \breve{\Lambda}' 
\]
 as in \eqref{contraction}, and an isomorphism of vector bundles
\[
\mathscr{D}( \Lambda'\otimes \overline{\mathbb{Y}}) \iso 
 \breve{\Lambda}'  \otimes_{ \breve{\Z}_p} \co_{\mathrm{RZ}_{\Lambda'}^{\heartsuit,\red} }.
\]
Using the multiplication-by-$p$ isomorphism $p^{-1}\Lambda' \iso \Lambda'$, we now identify
\[
\mathscr{D}( p^{-1} \Lambda' \otimes \overline{\mathbb{Y}} ) 
\iso  \breve{\Lambda}'  \otimes_{\breve{\Z}_p} \co_{\mathrm{RZ}_{\Lambda'}^{\heartsuit,\red} },
\]
and identify \eqref{Gflag} with a flag of coherent sheaves
\begin{equation}\label{aux flag}
0 \subset \mathscr{F} \subset \mathscr{F}^\dagger \subset  \breve{\Lambda}' _0
 \otimes_{ \breve{\Z}_p} \co_{\mathrm{RZ}_{\Lambda'}^{\heartsuit,\red} }.
\end{equation}

At a point of $\mathrm{RZ}_{\Lambda'}^{\heartsuit,\red} (\breve{\F}_p)$, corresponding to a chain of lattices
\[
  \breve{\Lambda}'_0   \stackrel{\frac{n}{2}-1} { \subset } L_0^* \stackrel{2}{\subset}  L_0  \stackrel{\frac{n}{2}-1} { \subset }p^{-1} \breve{\Lambda}'_0   , 
  \]
  the fibers of $\mathscr{F}$ and $\mathscr{F}^\dagger$ are identified with the images of
\[
\frac{pL_0^*}{p^2L_0} \to  \frac{ \breve{\Lambda}'_0  }{   p \breve{\Lambda}'_0  } 
\quad\mbox{and}\quad
\frac{pL_0}{p^2 L_0}  \to  \frac{ \breve{\Lambda}'_0  }{   p \breve{\Lambda}'_0  } ,
\]
respectively. 
See the  proof of Theorem \ref{thm:YtoDL}, and especially Lemma \ref{lem:MN filter fibers}. 
In particular these coherent sheaves have constant fiber dimension, and one can deduce using Lemma \ref{lem:bundle criterion} that they are local direct summands of $\breve{\Lambda}' _0
 \otimes_{ \breve{\Z}_p} \co_{\mathrm{RZ}_{\Lambda'}^{\heartsuit,\red} }$.
 The equality $\mathscr{F}^\dagger = \mathscr{F}^\perp$ can also be checked on fibers, where it is clear from the definition of the pairing $b'$.
 
 All of this shows that \eqref{aux flag} defines a morphism 
 \[
 \mathrm{RZ}_{\Lambda'}^{\heartsuit,\red} \to \mathrm{DL}_{\Lambda'}^\heartsuit
 \]
  with the desired form on $\breve{\F}_p$-valued points, and it remains to show that it is an isomorphism.
  For this, once again by Lemma \ref{lem:general iso}, it suffices to show that the  morphism in question is  unramified.
  The proof is essentially the same as that of Lemma \ref{lem:YtoDLunr}, replacing the use of Lemma \ref{lem:miracle filtration} with the equality
 \[
 \Fil^0 \mathscr{D}(X) = \mathrm{ker} ( \mathscr{D}(X) \to \mathscr{G}^\dagger/\mathscr{G} \iso \mathscr{F}^\perp/\mathscr{F}).
 \]
This last equality can be verified on  fibers over points $s\in  \mathrm{RZ}_{\Lambda'}^{\heartsuit,\red}(\breve{\F}_p)$, where it is equivalent to the obvious equality
\[
\frac{L^*_0}{p L_0} 
 = \mathrm{ker}\left(
\frac{L_0}{pL_0} \map{p} \frac{pL_0 / p \breve{\Lambda}_0'}{pL_0^* / p \breve{\Lambda}_0'} 
\right) . 
\]
\end{proof}


\section{Irreducible components of the Rapoport-Zink space}


In this section we prove our main results on the structure of the reduced scheme $\mathrm{RZ}^\red$ underlying the formal $\breve{\F}_p$-scheme $\mathrm{RZ}$ of Definition \ref{def:fullRZ}.   
The key point is to explain the relation between the  locally closed subschemes  $ \mathrm{RZ}^{k,\red}_\Lambda \subset \mathrm{RZ}^\red$ of Definition \ref{def:RZk}, and  the irreducible components of $ \mathrm{RZ}^\red$, as described in \cite{XZ} and \cite{FI}.


\subsection{The affine Deligne-Luszig variety}


 Fix an integer $n\ge 2$.
 Let $W$ be an $n$-dimensional vector space over $E$ equipped with a hermitian form 
 $h : W \times W \to E$.  Up to isomorphism there are two such $W$, distinguished  by the value of 
 \[
 \det(W) \in \Q_p^\times /\mathrm{Nm}_{E/\Q_p}(E^\times) .
 \]
 
 We assume that $\det(W)=1$, which implies the existence of an $E$-basis $x_1,\ldots, x_n \in W$ such that the hermitian form is given by the matrix with $1$'s on the antidiagonal, and $0$'s elsewhere.  In other words
 \[
 h(x_i,x_j) = \begin{cases}
 1 & \mbox{if }i+j=n + 1 \\
 0 & \mbox{otherwise.}
 \end{cases}
 \]
 
The group of unitary similitudes $G=\GU(W)$ is an unramified reductive group over $\Q_p$, and our choice of basis determines subgroups
 \[
 T \subset B \subset G, 
 \]
 in which the Borel $B$ is  the stabilizer of the flag  $\mathcal{F}_1\subset \cdots \subset \mathcal{F}_n=V$ defined by $\mathcal{F}_i=\mathrm{Span}_E\{x_1,\ldots, x_i\}$, and  $T$ is the  maximal torus that acts through scalars on every $x_i$.

Regard $x_1 ,\ldots, x_n \in \breve{W}\define  W \otimes_{\Q_p} \breve{\Q}_p$ as an $\breve{E}$-basis, and denote by 
\[
y_1,\ldots, y_n \in  \breve{W}_0 \quad \mbox{and}\quad z_1,\ldots, z_n \in  \breve{W}_1
\]
the projections of these basis vectors to the two summands in the decomposition of \eqref{general decomp}.
Thus $x_i =y_i+z_i$, the Frobenius operator
$
\sigma : \breve{W} \to \breve{W}
$
interchanges $y_i$ with $z_i$,  and we have a $\breve{\Q}_p$-basis
\[
y_1,\ldots, y_n,z_1,\ldots, z_n  \in \breve{W} .
\]

We use this last basis to identify $G(\breve{\Q}_p) \subset \mathrm{GL}_{2n}(\breve{\Q}_p)$.
 The Borel  $B(\breve{\Q}_p)$ is then identified with the upper triangular matrices in $G(\breve{\Q}_p)$, while  $T(\breve{\Q}_p)$ is the subgroup  of diagonal matrices of the form
\[
[ t_0,t_1,\ldots, t_n]  \define \begin{pmatrix}
t_1 \\
& \ddots \\
& & t_n \\
& & &  t_n^{-1} t_0   \\
& & & & \ddots \\
& & & & &  t_1^{-1} t_0
\end{pmatrix}   .
\]  
The cocharacter lattice of $T$  has a $\Z$-basis
$
\epsilon_0,\epsilon_1,\ldots, \epsilon_n \in X_*(T)
$
given by
\begin{align*}
\epsilon_0(t)  &= [ t , 1 , 1 ,\ldots, 1] \\
 \epsilon_1(t) &  = [ 1 , t,1 ,\ldots, 1]  \\
 & \ \vdots \\
 \epsilon_n(t) & = [ 1,1,\ldots, 1,  t] .
\end{align*}

\begin{definition}
A cocharacter
\[
\lambda = a_0\epsilon_0+a_1\epsilon_1 \cdots + a_n\epsilon_n \in X_*(T)
\]
is \emph{minuscule} if $|a_i-a_j| \le 1$ for all $i,j \in \{1,\ldots, n\}$, and is  \emph{dominant} (with respect to  $B$) if  $a_1\ge a_2 \ge \cdots \ge a_n$.
\end{definition}

\begin{remark}
The action of the Frobenius $\sigma$ on $X_*(T)$ is given by 
\begin{align*}
\epsilon_0^\sigma  & = \epsilon_0+\epsilon_1+\cdots+\epsilon_n \\
\epsilon_i^\sigma & = - \epsilon_{i^\vee} \mbox{ for }1\le i \le n,
\end{align*}
where we abbreviate $i^\vee =  {n+1-i}$.  
\end{remark}

\begin{remark}
The center $Z \subset G$ is isomorphic to the Weil restriction $\mathrm{Res}_{E/\Q_p} \mathbb{G}_m$, and its cocharacter lattice is
\[
 X_*(Z) = \mathrm{Span}_\Z \{\epsilon_0,\epsilon_0^\sigma\}   \subset X_*(T).
\]
\end{remark}

The basis vectors $x_1,\ldots, x_n \in W$ span an $\co_E$-lattice 
\[
\Lambda = \mathrm{Span}_{\co_E}\{ x_1,\ldots, x_n\} \subset W
\]
self-dual under the hermitian form.
The group of unitary similitudes of $\Lambda$ determines an extension of $G$ to a reductive group scheme over $\Z_p$, denoted the same way.
The subgroup $G(\breve{\Z}_p) \subset G(\breve{\Q}_p)$ is the stabilizer of the $\co_{\breve{E}}$-lattice
\begin{equation}\label{ADLlattices}
\mathbb{D} \define \Lambda\otimes_{\co_E} \co_{\breve{E}} =
  \mathrm{Span}_{\breve{\Z}_p} \{y_1,\ldots, y_n,z_1,\ldots, z_n   \}  \subset \breve{W}.
\end{equation}
We now define the particular affine Deligne-Lusztig variety of interest.

\begin{definition}
The \emph{affine Deligne-Lusztig variety} is the set 
\[
X_\mu(b)   = \{ g \in G(\breve{\Q}_p) / G(\breve{\Z}_p) : g^{-1}  b g^\sigma \in  G(\breve{\Z}_p) \mu(p) G(\breve{\Z}_p) \}  
\]
where  $b = \epsilon_0(p) \in Z(\breve{\Q}_p)$ and 
$
 \mu = \epsilon_0+\epsilon_1+\epsilon_2 \in X_*(T) .
$
\end{definition}

\begin{remark}\label{rem:central b}
Because our chosen $b\in G(\breve{\Q}_p)$ is central, its twisted centralizer $J_b$ is canonically identified with $G$, and 
the Deligne-Lusztig variety $X_\mu(b)$ is stable under left multiplication by $G(\Q_p)$.
\end{remark}

\begin{proposition}\label{prop:DLtoRZ}
If we use the lattice $\Lambda$ above  to define the Rapoport-Zink space  of  \S \ref{ss:RZbasics},
there is a bijection of sets
\[
\mathrm{RZ}(\breve{\F}_p) \iso X_\mu(b) 
\]
identifying
\[
\mathrm{RZ}_\Lambda(\breve{\F}_p) \iso \{ g\in X_\mu(b) :   p \mathbb{D} \subset  g \mathbb{D} \subset p^{-1} \mathbb{D} \}.
\]
\end{proposition}

\begin{proof}
First note that $F=b\circ \sigma$  defines an isocrystal structure on $\breve{W}$ with  
\[
F y_i = p z_i \quad \mbox{and} \quad F z_i = y_i.
\]
  The lattice $\mathbb{D}$ is stable under $F$ and $V$, and it is easy to see from 
\eqref{dieudonne-lie} that it is the Dieudonn\'e module $\mathbb{D}=D(\mathbb{X})$ of a $p$-divisible group $\mathbb{X}$ with an $\co_E$-action of  signature $(0,n)$.  
Fixing a $u \in \co_E^\times$ with $\overline{u} = -u$, the alternating form 
\[
\lambda_W(x,y)\define \mathrm{Tr}_{E/\Q_p} h(ux,y)
\]
 on $W$ extends $\breve{\Q}_p$-bilinearly to a polarization of the isocrystal $\breve{W}$.
 The lattice $\mathbb{D} \subset \breve{W}$ is self-dual under this alternating form, which 
determines an $\co_E$-linear  principal polarization of $\mathbb{X}$.

Under the isomorphism $\mathbb{D} \iso \Lambda \otimes_{\co_E} \co_{\breve{E}}$ of  \eqref{ADLlattices}, the 
 orthogonal idempotents $e_0,e_1 \in \co_{\breve{E}}$ from \eqref{orthopotents} satisfy
 \[
 y_i = x_i \otimes e_0 \quad \mbox{and} \quad z_i = x_i \otimes e_1.
 \]
The operator $F$ on the left hand side therefore has the form $F=\mathrm{id} \otimes F'$ for a  unique $\sigma$-semi-linear operator $F'$ on $\co_{\breve{E}}$, namely $F' e_0 = p e_1$ and $F' e_1=e_0$.
This operator makes $ \co_{\breve{E}}$ into a Dieudonn\'e module isomorphic to $D(\overline{\mathbb{Y}})$, and a choice of such an isomorphism identifies
\[
\mathbb{X} = \Lambda \otimes \overline{\mathbb{Y}} .
\]

The rest  is routine. 
For any $g\in X_\mu(b)$ the $\co_{\breve{E}}$-lattice $g\mathbb{D}=D(X)$ is  the Dieudonn\'e module of a $p$-divisible group $X$ with an $\co_E$-action, and the cocharacter  $\mu$ was chosen to ensure that $X$ has signature $(2,n-2)$. 
The inclusion 
\[
D(X)=g \mathbb{D}  \subset \mathbb{D} [1/p] = D(\mathbb{X}) [1/p]
\]
 corresponds to an $\co_E$-linear quasi-isogeny $\varrho_X: X \dashrightarrow \mathbb{X}$, and the pullback of the principal polarization on the target can be rescaled by a unique power of $p$ to obtain a principal polarization $\lambda_X$ of $X$.
The triple $(X,\lambda_X,\varrho_X)$ defines a point of $\mathrm{RZ}(\breve{\F}_p)$, and this is the desired bijection.
\end{proof}


\subsection{Labeling the   components}


We now invoke the parametrization of irreducible components of Rapoport-Zink spaces due to Xiao-Zhu \cite{XZ}, and its refinement in the case of $\mathrm{GU}(2,n-2)$ worked out in \cite{FI}.

Given cosets 
$
g_1,g_2\in G(\breve{\Q}_p) / G(\breve{\Z}_p),
$
  the Cartan decomposition implies the existence of a unique
dominant $\beta_{g_1,g_2} \in X_*(T)$ such that 
\[
g_2^{-1} g_1 \in G(\breve{\Z}_p) \beta_{g_1,g_2}(p) G(\breve{\Z}_p) .
\]
We call this cocharacter  the \emph{relative position invariant} of the  lattices $g_1\mathbb{D}$ and $g_2\mathbb{D}$, and denote it by
\[
\inv_G(g_1 \mathbb{D}, g_2\mathbb{D})  \define \beta_{g_1,g_2} \in X_*(T) .
\]
The subscript $G$ is included to distinguish this from the invariant  of Definition \ref{def:GLinvariant}.
In this terminology, our affine Deligne-Lusztig variety becomes
\[
X_\mu(b)  = \{ g \in G(\breve{\Q}_p) / G(\breve{\Z}_p) :  \inv_G( b g^\sigma \mathbb{D} , g \mathbb{D}) = \mu  \} .
\]

Define cocharacters
$
\alpha_1,\ldots, \alpha_{ \lfloor n/2\rfloor}  \in X_*(T)
$
 by (recall  $k^\vee =  {n+1-k}$)
\begin{equation}\label{alpha}
\alpha_k = \begin{cases}
 (\epsilon_1+ \cdots +\epsilon_{k-1}) - (\epsilon_{k^\vee}+ \cdots +\epsilon_{1^\vee}) & \mbox{if } k <n/2 \\
 \epsilon_0+\epsilon_1+\cdots+\epsilon_{ k-1} & \mbox{if } k= n/2.
\end{cases}
\end{equation}
(See Remark \ref{rem:bad alpha} for an explanation of why the case $k=n/2$ is treated differently.)
Using the bijection 
\[
X_\mu(b)   \iso \mathrm{RZ}(\breve{\F}_p)
\]
 of Proposition \ref{prop:DLtoRZ}, for any $\gamma \in G(\Q_p)$ and $1\le k \le \lfloor n/2\rfloor$, denote by 
 \[
 \mathrm{RZ}_{(k,\gamma)} \subset \mathrm{RZ}^\red
 \]
the locally closed subset  (endowed with its reduced scheme structure) whose $\breve{\F}_p$-points are
\begin{equation}\label{XZcomponents}
 \mathrm{RZ}_{(k,\gamma)}(\breve{\F}_p) = 
 \left\{ g \in X_\mu(b) :  \mathrm{inv}_G( g \mathbb{D} , \gamma \mathbb{D} ) = \alpha_k \right\} .
\end{equation}
This only depends on the coset $\gamma \in G(\Q_p)/G(\Z_p)$, and satisfies
\begin{equation}\label{translation compatible}
\mathrm{RZ}_{(k,\gamma)}= \gamma \cdot \mathrm{RZ}_{(k,\mathrm{id})} .
\end{equation}

\begin{remark}
The cocharacters $\alpha_1$ and (if $n$ is even) $\alpha_{n/2}$ are minuscule, but the other $\alpha_i$'s are not.  This is closely tied up with the fact that $\mathrm{RZ}_{(1,\gamma)}$ and 
$\mathrm{RZ}_{(n/2,\gamma)}$ are isomorphic to  Deligne-Lusztig varieties, as we soon see.
\end{remark}

\begin{theorem}\label{thm:XZexplicit}
Abbreviate $r=\lfloor n/2 \rfloor$.
There is a bijection 
\[
\{ 1,\ldots,  r \} \times G(\Q_p)/G(\Z_p) \iso
\{ \mbox{irreducible components of } \mathrm{RZ} \}
\]
 sending  the pair $(k,\gamma)$ to the closure of  $\mathrm{RZ}_{(k,\gamma)}$.
  \end{theorem}

\begin{proof}
This is the parametrization of irreducible components from \cite{XZ}, made explicit in \cite{FI} in the special case of $\mathrm{GU}(2,n-2)$.  For the reader's benefit, we provide a very rough sketch.

The triple $T\subset B\subset G$ has  Langlands dual 
\[
\widehat{T} \subset \widehat{B} \subset \widehat{G} \iso \GL_{2n} \times \mathbb{G}_m,
\]
where $\widehat{B}$ is the subgroup of matrices upper triangular in the first factor, and $\widehat{T}$ is the subgroup of matrices diagonal in the first factor.
Let $V_\mu$ be the representation of $\widehat{G}$ of heighest weight 
\[
\mu  \in X_*(T)=X^*(\widehat{T}). 
\]
Every $\lambda \in X_*(T)=X^*(\widehat{T})$ determines a weight space $V_\mu(\lambda)$, 
and because $\mu$ is minuscule the nonzero weight spaces lie in a single orbit under the action of the Weyl group.  In other words
\begin{equation}\label{MVsize}
\dim V_\mu(\lambda)
= \begin{cases}
1 & \mbox{if }  \lambda= \epsilon_0+\epsilon_i+\epsilon_j \mbox{ for some } i\neq j\\
0 & \mbox{otherwise.}
\end{cases}
\end{equation}

For any $\lambda \in X_*(T)$ denote by $[\lambda] \in X_*(T)/ (\sigma-1) X_*(T)$ its image under the quotient map.
By  Theorem 4.4.14 of \cite{XZ}, and recalling the equality $J_b=G$ of Remark \ref{rem:central b},
the irreducible components of $X_\mu(b)$ are in bijection with 
\begin{equation}\label{XZparameters}
\bigsqcup_{ \substack{ \lambda\in X_*(T) \\  [ \lambda] =[   \epsilon_0 ]  }}\mathbb{MV}_\mu(\lambda) \times 
G(\Q_p)/G(\Z_p) ,
\end{equation}
where $\mathbb{MV}_\mu(\lambda)$ is a finite set of cardinality  \eqref{MVsize}.

As $[ \epsilon_0+\epsilon_i+\epsilon_j]  = [\epsilon_0]$ if and only if $j=i^\vee$,
  the calculation \eqref{MVsize} shows that the  $\lambda$  contributing to \eqref{XZparameters} are precisely those of the form
  \[
\lambda _i \define \epsilon_0 + \epsilon_i + \epsilon_{i^\vee} \in X_*(T)
\]
with $1\le i \le r$.
  Thus  Theorem 4.4.14 of \cite{XZ} establishes  a bijection 
 \begin{equation}\label{XZsimple}
 \{ \lambda_1,\ldots, \lambda_r\}  \times G(\Q_p)/G(\Z_p) \iso \{\mbox{irreducible components of }\mathrm{RZ}^\red \},
\end{equation}
which we must make explicit.

 For every $1\le i \le r$,  Lemma 4.4.3 of \cite{XZ} associates a dominant cocharacter $\nu_i\in X_*(T)$ to the unique element of  the set $\mathbb{MV}_\mu(\lambda_i)$.
This  cocharacter   is not uniquely determined by this recipe, but once it is chosen one defines
\[
\tau_i = \lambda_i+ \nu_i -  \nu_i^\sigma \in X_*(T),
\]
and chooses  any $\delta_i \in X_*(T)$  satisfying $\tau_i = \epsilon_0 + \delta_i -  \delta_i^\sigma$.
In the case at hand we make the following choices.
For $1\le i < n/2$ set 
\begin{align*}
\nu_i  & =  (\epsilon_1+ \cdots +\epsilon_{i-1}) - (\epsilon_{i^\vee}+ \cdots +\epsilon_{1^\vee}) \\
  \tau_i & =\epsilon_0  \\
  \delta_i &= 0.
\end{align*}
When $n$ is even, so that $r=n/2$,  set
\begin{align*}
\nu_r  & =  \epsilon_1+ \cdots +\epsilon_{r-1} \\
  \tau_r & = \epsilon_0+\epsilon_1+\cdots+\epsilon_n \\
  \delta_r &=   \epsilon_1+ \cdots +\epsilon_r.
\end{align*}

Each of the minuscule cocharacters $\tau_i$ determines a  $b_i=\tau_i(p) \in G(\breve{\Q}_p)$, 
with its own affine Deligne-Lusztig variety
\[
X_\mu(b_i)  = \{ g \in G(\Q_p)/G(\Z_p) :  \inv_G( b_i g^\sigma \mathbb{D} , g \mathbb{D}) = \mu  \} .
\]
The locally closed subset
\[
\mathring{X}_{\mu,\nu_i}(b_i)  = \left\{ g\in X_\mu(b_i)  :  \mathrm{inv}_G( g \mathbb{D} , \mathbb{D} ) = \nu_i  \right\} 
\]
is  irreducible (see the proof of Lemma 7.2 of \cite{FI})  and its closure in $ X_\mu(b_i)$ is an irreducible component.

There is a bijection
\[
\Delta_i : X_\mu(b_i)  \to  X_\mu(b) 
\]
defined by $\Delta_i(g)  = \delta_i(p^{-1}) g$.  
Note  that $\delta_i(p^{-1}) \in G(\breve{\Q}_p)$ may not be central, and so this  bijection  need not respect the natural left actions of $G(\Q_p)$ on the source and target.
As $\delta_i-\delta_i^\sigma$ is a central cocharacter,  conjugation by $\delta_i(p^{-1})$ defines an automorphism of $G(\Q_p)$.
All hyperspecial subgroups of $G(\Q_p)$ are conjugate, so we may fix  a $k_i \in G(\Q_p)$ such that 
\begin{equation}\label{conj stabilizer}
\delta_i(p^{-1}) G(\Z_p) \delta_i(p) = k_i^{-1} G(\Z_p) k_i .
\end{equation}
The closure  of 
\begin{equation}\label{full component}
  k_i \Delta_i( \mathring{X}_{\mu,\nu_i}(b_i) )      =   \left\{ g\in X_\mu(b)  :  \mathrm{inv}_G( g \mathbb{D} ,  k_i  \delta_i(p^{-1}) \mathbb{D} ) = \nu_i  \right\} 
\end{equation}
 is an irreducible component of $\mathrm{RZ}^\red$, and is the image of $(i,\mathrm{id})$ under \eqref{XZsimple}.  Indeed, this is the definition of the bijection \eqref{XZsimple}.

If $1\le i <n/2$ then
\begin{align*}
 \eqref{full component}  
 &=    \left\{ g \in X_\mu(b) :  \mathrm{inv}_G( g \mathbb{D} , \mathbb{D} ) 
  = \nu_i \right\}  \\
  & =   \left\{ g \in X_\mu(b) :  \mathrm{inv}_G( g \mathbb{D} , \mathbb{D} ) = \alpha_i \right\}  \\
  &= \mathrm{RZ}_{(i,\mathrm{id})},
\end{align*}
as desired.
Now suppose $n$ is even, and  $i=n/2$.  The cocharacter $\epsilon_0+\delta_i$ is fixed by $\sigma$, and taking
$
k_i = (\epsilon_0+\delta_i)(p) \in G(\Q_p)
$
we find 
\begin{align*}
 \eqref{full component}  
 &=    \left\{ g\in X_\mu(b)  :  \mathrm{inv}_G( g \mathbb{D} ,  \epsilon_0 (p) \mathbb{D} ) = \nu_i  \right\}   \\
  & =  \left\{ g\in X_\mu(b)  :  \mathrm{inv}_G( g \mathbb{D} ,   \mathbb{D} ) = \alpha_i  \right\} \\
  &= \mathrm{RZ}_{(i,\mathrm{id})},
\end{align*}
as desired.
\end{proof}


\subsection{Main results}
\label{ss:main}


We can now put everything together to state and prove our main results on the structure of $\mathrm{RZ}^\red$.

This amounts to describing the structure of the locally closed subsets $\mathrm{RZ}_{(k,\gamma)}$ appearing in  Theorem \ref{thm:XZexplicit}, which we do by comparing them with the locally closed subsets (Definition \ref{def:RZk}) appearing in the decomposition
\begin{equation}\label{final strata}
 \mathrm{RZ}_\Lambda^\red = \bigsqcup_{k\ge 1} \mathrm{RZ}_\Lambda^{k,\red}
\end{equation}
determined by our choice of framing object $\Lambda \otimes \overline{\mathbb{Y}}$.

By Proposition \ref{prop:strata points}, the subschemes on the right hand side of \eqref{final strata} are nonempty only when $1\le k \le \lfloor n/2\rfloor$.
The cases $k<n/2$ and $k=n/2$ will require separate treatment, as did the definition   of the cocharacter $\alpha_k$ appearing in \eqref{XZcomponents}.

\begin{theorem}\label{thm:main components}
For  any $k<n/2$ we have
\[
\mathrm{RZ}_{(k,\mathrm{id})} = \mathrm{RZ}_\Lambda^{k,\red} .
\]
Moreover, 
for any $\gamma \in G(\Q_p)$ there is a smooth morphism 
\[
\mathrm{RZ}_{(k,\gamma)} \to \mathrm{DL}_\Lambda^k
\]
to the smooth and proper Deligne-Lusztig variety of Definition \ref{def:non-minuscule DL}.
Over this Deligne-Lusztig variety there is a rank $2k-1$ vector bundle $\mathscr{V}$   equipped with a   morphism
\[
\beta : \mathscr{V} \otimes \sigma^*\mathscr{V} \to \co_{\mathrm{DL}_\Lambda^k}
\]
 and a rank $k$ local direct summand $ \mathscr{V}^{(k)} \subset \mathscr{V}$  such that
 \[
\mathrm{RZ}_{(k,\gamma)}(S)  \iso
\left\{  
\begin{array}{c} 
 \mbox{rank $k-1$ local direct summands } 
 \mathscr{F} \subset \mathscr{V}_S   \\
 \mbox{satisfying }   \beta (\mathscr{F} \otimes \sigma^*\mathscr{F}) =0 
\mbox{ and } \mathscr{V}_S =  \mathscr{F} \oplus \mathscr{V}^{(k)}_S
  \end{array}  \right\} 
\]
 for any $\mathrm{DL}_\Lambda^k$-scheme $S$.
Recall that $\sigma^*\mathscr{V}$ means the Frobenius twist \eqref{coherent frob twist}.
\end{theorem}

\begin{proof}
Recall the relative position invariant  of Definition \ref{def:GLinvariant}.
Under the bijection of Proposition \ref{prop:DLtoRZ}, an element $g\in X_\mu(b)$ corresponds to a $p$-divisible group $X\in  \mathrm{RZ}(\breve{\F}_p)$ satisfying
 \begin{align*}
\mathrm{inv}( g \mathbb{D}_0 ,  \mathbb{D}_0 )   
& = 
\mathrm{inv}(D(X)_0 , D(\Lambda \otimes \overline{\mathbb{Y}} )_0  )    \\
 \mathrm{inv}( g \mathbb{D}_1 ,  \mathbb{D}_1 )   
 & = 
 \mathrm{inv}(D(X)_1 , D(\Lambda \otimes \overline{\mathbb{Y}} )_1  )  ,
 \end{align*}
 which corresponds, under the bijection of Proposition \ref{prop:hyperspecial points},  to a lattice $L \subset \breve{\Lambda}[1/p]$ satisfying
 \begin{align*}
\mathrm{inv}(D(X)_0 , D(\Lambda \otimes \overline{\mathbb{Y}} )_0  )   
& = 
\mathrm{inv}(L_0  , \breve{\Lambda}_0   )    \\
 \mathrm{inv}(D(X)_1 , D(\Lambda \otimes \overline{\mathbb{Y}} )_1  ) 
 & =    
 \mathrm{inv}(L_1  , \breve{\Lambda}_1 ).
 \end{align*}

Directly from the definition \eqref{XZcomponents}, the element $g$ above lies in the subset $\mathrm{RZ}_{(k,\mathrm{id})}$ if and only if 
 \begin{align*}
  \mathrm{inv}( g \mathbb{D}_0 ,  \mathbb{D}_0 )  
  & =   ( \overbrace{1,\ldots, 1}^{k-1\mathrm{\ times}}, 0,\ldots, 0 , \overbrace{-1,\ldots, -1}^{k\mathrm{\ times}} )  \\
   \mathrm{inv}( g \mathbb{D}_1 ,  \mathbb{D}_1 )  
   & =    ( \underbrace{1,\ldots, 1}_{k\mathrm{\ times}}, 0,\ldots, 0 , \underbrace{-1,\ldots, -1}_{k-1\mathrm{\ times}} ).
 \end{align*}
By Proposition \ref{prop:strata points} and  the previous paragraph, these conditions are equivalent to $X\in \mathrm{RZ}_\Lambda^k(\breve{\F}_p)$. 
In other words, $\mathrm{RZ}_{(k,\mathrm{id})} = \mathrm{RZ}_\Lambda^{k,\red}$.

Given this last equality and \eqref{translation compatible}, the rest of the theorem is  a restatement of Corollary \ref{cor:grassmannian}.
\end{proof}

\begin{remark}
We remind the reader that the  vector bundle $\mathscr{V}$ on 
\[
 Y_\Lambda^{k,\red} \iso  \mathrm{DL}_\Lambda^k
 \]
   was constructed in Section \ref{ss:special bundle},
which contains more information about its filtration and the morphism  $\beta$.  See especially Proposition \ref{prop:beta isotropy}.
\end{remark}

\begin{remark}
Theorem \ref{thm:main components} implies that
\[
\mathrm{RZ}_{(1,\gamma)} \iso  \mathrm{RZ}_\Lambda^{1,\red}  \iso  \mathrm{DL}_\Lambda^1
\]
is itself a Deligne-Lusztig variety.
\end{remark}

Now we turn to the case $k=n/2$, so suppose $n$ is even.
In this case we have  inclusions (Proposition \ref{prop:heart decomp})
\begin{equation}\label{supplement inclusion}
\mathrm{RZ}_\Lambda^{n/2,\red} \subset 
 \bigcup_{ p\Lambda  \subsetneq \Lambda'  \subsetneq \Lambda}  \mathrm{RZ}_{\Lambda'}^{\heartsuit ,\red}
 \subset \mathrm{RZ}^\red
\end{equation}
 in which the union is over all scalar-self-dual (Definition \ref{def:scalar dual})   $\co_E$-lattices  $\Lambda' $ lying between $p\Lambda$ and $\Lambda$, and each 
\[
 \mathrm{RZ}_{\Lambda'}^{\heartsuit ,\red} \subset \mathrm{RZ}_\Lambda^\red
 \]
  is the  closed subscheme of Definition \ref{def:heart}.  
  The following implies that every such $\mathrm{RZ}_{\Lambda'}^{\heartsuit ,\red}$ is an irreducible component of $\mathrm{RZ}^\red$.

\begin{theorem}\label{thm:aux components}
Assume that $k=n/2$, and fix a $\Lambda'$ as in  \eqref{supplement inclusion}
\begin{enumerate}
 \item
There is an $h \in G(\Q_p)$ such that 
$\mathrm{RZ}_{( k ,h)} =  \mathrm{RZ}_{\Lambda'}^{\heartsuit ,\red} $.
\item
For any $\gamma \in G(\Q_p)$ there is an isomorphism
 \[
\mathrm{RZ}_{(k,\gamma)} \iso \mathrm{DL}_{\Lambda'}^\heartsuit
\]
where the smooth and proper Deligne-Lusztig variety on the right is that  of Definition \ref{def:heartDL}.  
\end{enumerate}
\end{theorem}

\begin{proof}
Recall that a hermitian space over $E$ is determined up to isometry by its dimension and determinant, viewed as an element of $\Q_p^\times/ \mathrm{Nm}_{E/\Q_p} (E^\times)$.  
As $n=\dim(W)$ is even, its determinant (hence its isometry classes) is unchanged if we rescale the hermitian form by an element of $\Q_p^\times$.  In other words, the similitude character $G(\Q_p) \to \Q_p^\times$ is surjective.  Combining this with the observation that all self-dual lattices in $W$ are isometric, it follows that any two scalar self-dual lattices in $W$ lie in the same $G(\Q_p)$-orbit.

By the previous paragraph, we may now fix an $h \in G(\Q_p)$ satisfying 
\[
h \Lambda = p^{-1}\Lambda'.
\]
Directly from the definition \eqref{XZcomponents}, an element $g\in X_\mu(b)$  lies in the subset $\mathrm{RZ}_{(k,h)}(\breve{\F}_p)$ if and only if 
\[
ph  \mathbb{D}_0   \stackrel{k+1}{\subset}  g \mathbb{D}_0  \stackrel{k-1}{\subset} h \mathbb{D}_0
\quad \mbox{and} \quad
p h\mathbb{D}_1   \stackrel{k-1}{\subset}  g \mathbb{D}_1  \stackrel{k+1}{\subset}  h\mathbb{D}_1.
\]
By our choice of $h$,  this is equivalent to the corresponding p-divisible group $X \in \mathrm{RZ}(\breve{\F}_p)$ satisfying
\begin{align*}
  D( \Lambda' \otimes\overline{\mathbb{Y}})_0
  \stackrel{k+1}{\subset}  
  D(X)_0  
  \stackrel{k-1}{\subset} 
 D( p^{-1}\Lambda' \otimes\overline{\mathbb{Y}})_0  \\  
 D( \Lambda'  \otimes\overline{\mathbb{Y}})_1
  \stackrel{k-1}{\subset}  
  D(X)_1  
  \stackrel{k+1}{\subset} 
D(p^{-1}\Lambda' \otimes\overline{\mathbb{Y}})_1.
  \end{align*}
Such $p$-divisible groups lie in the subset $\mathrm{RZ}_\Lambda(\breve{\F}_p) \subset \mathrm{RZ}(\breve{\F}_p)$, by the inclusions 
\[
p^{-1}\Lambda \subset \Lambda' \subset p^{-1}\Lambda' \subset \Lambda, 
\]
and correspond under the bijection of Corollary \ref{cor:half hyperspecial}  to lattices $L_0 \subset \breve{\Lambda}_0[1/p]$
satisfying
\[
\breve{\Lambda}_0' \stackrel{k-1}{\subset} L_0^* \subset L_0 \stackrel{k-1}{\subset} p^{-1} \breve{\Lambda}_0' .
\]
This is exactly the characterization of $\mathrm{RZ}_{\Lambda'}^{\heartsuit,\red}(\breve{\F}_p)$ from Proposition \ref{prop:heartpoints}, proving that 
\[
\mathrm{RZ}_{(k,h)}(\breve{\F}_p) = \mathrm{RZ}_{\Lambda'}^{\heartsuit,\red}(\breve{\F}_p).
\]
This completes the proof of the first claim.

Given the first claim and \eqref{translation compatible}, the second claim follows immediately from Theorem \ref{thm:heart comparison}. 
\end{proof}

Recall that Theorem \ref{thm:XZexplicit} presents the irreducible components of $\mathrm{RZ}^\red$ as Zariski closures of certain locally closed subsets  $ \mathrm{RZ}_{(k,\gamma)}$.  A priori, these locally closed subsets could be rather small; indeed, deleting a proper closed subset from any of them would not change the statement of that theorem.
The following result says that they are large enough to cover the entire Rapoport-Zink space, without needing to take their Zariski closures.  It would be interesting to know if this phenomenon is particular to the $\mathrm{GU}(2,n-2)$ Rapoport-Zink space, or if it holds in the greater generality of \cite{XZ}.

\begin{corollary}\label{cor:component cover}
 The locally closed subschemes of \eqref{XZcomponents} satisfy
\[
\mathrm{RZ}^\red
 = \bigcup_{    \substack{   1 \le k \le \lfloor n/2\rfloor \\  \gamma  \in  G(\Q_p)/G(\Z_p) }  }  
 \mathrm{RZ}_{(k,\gamma)} .
\]
\end{corollary}

\begin{proof}
Combining  \eqref{translation compatible} with Theorems \ref{thm:main components} and \ref{thm:aux components}, we find that every $\mathrm{RZ}_{(k,\gamma)}$  is contained in some  $G(\Q_p)$-translate of $\mathrm{RZ}_\Lambda^\red$. 
As $\mathrm{RZ}_\Lambda^\red \subset \mathrm{RZ}^\red$ is closed, it follows now from Theorem \ref{thm:XZexplicit}  that every
irreducible component of $\mathrm{RZ}^\red$  is also contained in such a translate.
This proves that 
\[
\mathrm{RZ}^\red = \bigcup_{\gamma \in G(\Q_p) } \gamma \cdot \mathrm{RZ}_\Lambda^\red.
\]
Combining this with \eqref{final strata}, we find that
\[
 \mathrm{RZ}^\mathrm{red} 
 =  \bigcup_{  \substack{ \gamma \in G(\Q_p) \\ 1 \le k \le \lfloor n/2 \rfloor } }
  \gamma\cdot \mathrm{RZ}_\Lambda^{k,\mathrm{red} }.
 \]
 
Now fix a point   $s\in \mathrm{RZ}^\red$.  The paragraph above shows that $s$ is contained in some 
$\gamma \cdot \mathrm{RZ}_\Lambda^{k,\mathrm{red} }$, and we consider two cases.
If $k<n/2$ then Theorem  \ref{thm:main components} implies
\[
s \in \gamma  \cdot  \mathrm{RZ}_{ (k,\mathrm{id})} 
  =  \mathrm{RZ}_{ (k,\gamma)} .
\]
 If $k=n/2$ then we  use \eqref{supplement inclusion} and Theorem \ref{thm:aux components} to see that 
 \[
 s \in  \gamma \cdot \mathrm{RZ}_{\Lambda'}^{\heartsuit , \red}  = \mathrm{RZ}_{(n/2, \gamma')}
 \]
 for some scalar-self-dual $\Lambda'$ and some $\gamma' \in G(\Q_p)$.
\end{proof}

\begin{remark}\label{rem:bad alpha}
Recall that the cocharacters $\alpha_k \in X_*(T)$ of \eqref{alpha} were defined differently in the cases $k<n/2$ and $k=n/2$. What would have happened if, when $n$ is even,  we had instead defined
\[
\alpha_{ n/2 }= (\epsilon_1+ \cdots +\epsilon_{(n/2) -1}) - (\epsilon_{( n/2 )^\vee}+ \cdots +\epsilon_{1^\vee})
\]
 in conformity with the case   $k<n/2$?
Theorem \ref{thm:main components} would still hold, proving that the locally closed subset
\[
\mathrm{RZ}_{( n/2 ,\gamma)}(\breve{\F}_p) = 
 \left\{ g \in X_\mu(b) :  \mathrm{inv}_G( g \mathbb{D} , \gamma \mathbb{D} ) = \alpha_{ n/2}  \right\} 
 \]
 satisfies
\[
\mathrm{RZ}_{( n/2 ,\mathrm{id})} = \mathrm{RZ}_\Lambda^{ n/2 ,\red} ,
\]
and (as  $k=n/2$ is allowed throughout  \S \ref{s:enhanced} and \S \ref{s:fibration}, and in particular in 
Corollary \ref{cor:grassmannian}), admits  a smooth morphism
\[
\mathrm{RZ}_\Lambda^{ n/2 ,\red} \to \mathrm{DL}_\Lambda^{n/2}
\]
identifying the domain as a moduli space of local direct summands of a vector bundle on the  codomain, word-for-word as in 
Theorem \ref{thm:main components}. 
The only new phenomenon we encounter is that neither the source nor the target of this morphism is irreducible
(compare with  \eqref{supplement inclusion} and  Proposition \ref{prop:DLstructure}) and hence neither is $\mathrm{RZ}_{( n/2 ,\mathrm{id})}$.  
Thus from our point of view,  the above naive definition of $\alpha_{n/2}$ is not the right one to make.
\end{remark}

\end{document}